\documentclass{amsart}
\usepackage[margin=3cm]{geometry}

\usepackage{amsmath,amssymb,amsthm}
\usepackage{hyperref}
\usepackage{graphicx}
\usepackage{float}
\usepackage{subfigure}
\usepackage{todonotes}

\numberwithin{equation}{section}
\numberwithin{figure}{section}

\usepackage{relsize}

\usepackage{thmtools, thm-restate}
\newtheorem{question}{Question}[section]
\newtheorem{prop}[question]{Proposition}
\newtheorem{thm}[question]{Theorem}
\newtheorem{cor}[question]{Corollary}
\newtheorem{lemma}[question]{Lemma}

\theoremstyle{definition}
\newtheorem{definition}[question]{Definition}
\newtheorem{Ex}[question]{Example}
\newtheorem{remark}[question]{Remark}

\def\epsilon{\varepsilon}

\usepackage[numbers]{natbib}

\title{Uniform Length Estimates for Trajectories on Flat Cone Surfaces}

\author{Kai Fu}
\address[Kai Fu]{Max Planck Institute for Mathematics in the Sciences, Leipzig, Germany}
\email{kai.fu@mis.mpg.de}

\date{\today}

\begin{document}

\begin{abstract}
This paper studies length estimates for trajectories on flat cone surfaces in terms of their self-intersection numbers. For an area-one flat cone surface, we obtain a lower bound for the length of a trajectory, with constants depending only on the flat metric. Our main focus is the case of convex flat cone spheres. We show that these constants can be chosen uniformly for such spheres with a positive curvature gap and a fixed number of singularities. Explicit values for these constants are also provided. Combined with a previously established upper bound, this yields uniform two-sided estimates for trajectory lengths on such flat cone spheres. As an application, we obtain uniform bounds for counting functions of trajectories on convex flat cone spheres and on convex polygonal billiards.
\end{abstract}

\maketitle

\setcounter{tocdepth}{1}
\tableofcontents

\section{Introduction}\label{bigsec:intro}

\subsection{Trajectories on flat cone surfaces}
Let $P$ be a polygon in the plane. A \emph{billiard path} in $P$ is a finite polygonal path $\{s_1, \ldots, s_m\} \subset P$, composed of line segments, such that each vertex $s_i \cap s_{i+1}$ lies in the interior of some edge of $P$, and the angles that $s_i$ and $s_{i+1}$ make with this edge are complementary.

Let $\{s_1, \ldots, s_m\}$ be a billiard path in $P$. We say that the path is \emph{periodic} if the polygonal path $\{s_1, \ldots, s_m, s_{m+1}\}$, with $s_{m+1} = s_1$, is also a billiard path. We say that the path is a \emph{generalized diagonal} if $s_1$ and $s_m$ both have one endpoint at a vertex of $P$ (possibly the same).

A significant amount of literature focuses on the study of billiard paths. One classical question is the existence of a periodic path in a given polygonal billiard. It has been proved in~\cite{RSch06, RSch08, tokarsky2018point} that triangles with angles of at most $112.3^\circ$ all have a periodic billiard path. However, the existence of a periodic path remains open for general polygons.

Another classical question relates to the growth rate of the number of generalized diagonals of bounded length. We denote by $N^{diag}(P, R)$ the number of generalized diagonals of length at most $R$ in the polygon $P$. Katok proved in~\cite{Kat} that the growth of $N^{diag}(P, R)$ with $R$ is subexponential. Scheglov provided an explicit subexponential upper bound for almost every triangle in~\cite{Sch}. To date, explicit subexponential upper bounds for $N^{diag}(P, R)$ are not fully resolved for general polygonal billiards.
\newline

The study of billiard paths in a polygon can be reformulated as the study of trajectories on a flat cone sphere. Specifically, to a polygon $P$ in the plane, we associate a flat cone sphere obtained by gluing two copies of $P$ along their boundaries (this construction is detailed in~\cite{FoxKer}, with a more thorough description provided in Section~\ref{sec:appli}). A flat cone sphere is a sphere endowed with a flat metric and finitely many conical points. Thurston studied flat cone spheres in~\cite{thu} and a formal definition of flat cone spheres is provided in Section~\ref{sec:basicdef}.

This construction is particularly useful for rational polygons, which are polygons where all angles are rational multiples of $\pi$. In this case, the flat cone sphere has a branched cover that is a translation surface. Although we do not use translation surfaces in this paper, we refer readers to~\cite{Zor} for their definition. There are many questions about trajectories on flat cone surfaces for which answers are known in the case of translation surfaces. For instance, Masur proved in~\cite{Mas86} that every translation surface contains a periodic geodesic, and showed in~\cite{Ma1, Ma2} that the number of periodic geodesics of bounded length has quadratic upper and lower bounds as the length bound increases. It is worth noting that there is an $SL(2, \mathbb{R})$-action on moduli spaces of translation surfaces, and all the results mentioned above rely on this action. However, such an action does not extend to flat cone spheres. For more results on translation surfaces, readers can consult~\cite{AthreyaMasur2024, Filip2024, McMullen2023, MasurTabachnikov2002, Wright2015, Zor}.
\newline

This paper investigates the behavior of trajectories on flat cone spheres and general flat cone surfaces. We approach the study of these trajectories by comparing their lengths with their self-intersection numbers. 
In fact, recent studies have addressed such a comparison for the lengths of closed geodesics in hyperbolic surfaces, such as those in~\cite{Ara} and~\cite{BaPV}. However, no such estimates are known for flat cone surfaces. This estimate establishes groundwork for future work~\cite{Fu25}. See Remark~\ref{rmk:futurework} for further details.\newline

The definition of flat cone surfaces is provided in Section~\ref{sec:basicdef}. A \emph{trajectory} on a flat cone surface $X$ is defined as a geodesic of finite length that does not pass through any singularities except possibly at its endpoints. We denote the length of a trajectory $\gamma$ by $|\gamma|$. The definition of self-intersection numbers follows from~\cite{Ara}.

\begin{definition}\label{def:intersectionnumber}
    Let $\gamma$ be a trajectory on a flat cone surface $X$. For a point $p$ on $\gamma$, we say that two tangent vectors to $\gamma$ at $p$ form a \emph{transverse pair} if they span the tangent space at $p$. The \emph{self-intersection number} of $\gamma$ is defined as the sum of transverse pairs over all points in the interior of $\gamma$, and we denote it by $\iota(\gamma, \gamma)$.
\end{definition}

We recall the notion of curvature gap introduced in the joint work with Tahar~\cite{FT}. Let $X$ be a flat cone surface with genus zero and $n$ conical points. Let $x_1, \ldots, x_n$ be the $n$ conical points. We define the \emph{curvature} of $x_i$ as
\[k_i := \frac{2\pi - \theta_i}{2\pi}\] where $\theta_i$ is the cone angle at $x_i$. We denote the tuple of curvatures by $(k_1, \ldots, k_n)$ and refer to it as $\underline{k}$. The \emph{curvature gap} of $\underline{k}$, denoted by $\delta(\underline{k})$, is defined by
\begin{equation}\label{equ:gap}
\delta(\underline{k}) = \inf\limits_{I \subset \{1, \ldots, n\}} \left|1 - \sum\limits_{\substack{i \in I}} k_i\right|.
\end{equation}

    If a flat cone sphere $X$ has positive curvature gap, then there is no \emph{simple} closed geodesic on $X$, where ``simple'' means that the self-intersection number of the geodesic is zero. Indeed, such a geodesic divides $X$ into two connected components, and by the Gauss--Bonnet formula, the sum of the curvatures in each component is equal to one. Hence the curvature gap is zero. 
    
    Note that the case of positive curvature gap does not include examples such as the cube, where an embedded cylinder exists. Indeed, one should think of flat cone spheres with positive curvature gap as \emph{generic}. More precisely, in the space of curvature vectors
\[
\{\underline{k} = (k_1,\ldots,k_n)\in\mathbb{R}^n \mid k_1+\cdots+k_n = 2\},
\]
the condition of having zero curvature gap is defined by finitely many linear equalities
\[
1 - \sum_{i\in I} k_i = 0, \qquad I\subseteq\{1,\ldots,n\},
\]
and thus corresponds to a subset of measure zero.

The following theorem from~\cite{FT} provides an upper bound for the lengths of trajectories in terms of their self-intersection numbers:

\begin{thm}[{\cite[Theorem~1.4, Theorem~1.6]{FT}}]
\label{thm:finite}
    Let $X$ be a flat cone sphere with unit area and positive curvature gap. Let $n$ be the number of conical singularities and $\delta$ be the curvature gap. Then the number of trajectories with at most $s$ self-intersections is at most
    $$
    (3n-6)2^{20n\delta^{-1}((n-1)\sqrt{s} + 1)}.
    $$
    Moreover, for any trajectory $\gamma$ in $X$, we have that
    \begin{equation}\label{equ:upperlengthformula}
        |\gamma| \le a_1\sqrt{\iota(\gamma,\gamma)} + a_2,
    \end{equation}
    where $a_1$ and $a_2$ are positive constants depending only on $n$ and $\delta$.
\end{thm}

Formulas for the constants $a_1$ and $a_2$ are explicitly given in~\cite{FT} as follows:
\begin{equation}\label{equ:upperexplicivalues}
a_1(n,\delta) = \frac{20n(n-1)}{\sqrt{\pi}}\left( \frac{2}{\delta} + \frac{1}{\sqrt{2}\delta^{3/2}} \right)
 \mbox{ and } a_2(n,\delta) = \frac{40n}{\delta\sqrt{\pi}} + \frac{20n}{\delta^{3/2}\sqrt{2\pi}}.    
\end{equation}

In this paper, we focus on obtaining lower bounds for the lengths of trajectories. Specifically, we investigate how these lower bounds relate to the self-intersection numbers, complementing the upper bound results established in~\eqref{equ:upperlengthformula}.

We compare our bounds for flat surfaces with the bounds in ~\cite{Ara} for hyperbolic surfaces in Remark~\ref{rmk:comparehyper1} and Remark~\ref{rmk:comparehyper2}. These comparisons reveal the relationships and differences between the bounds in flat and hyperbolic settings.

To present the first main result, we introduce some necessary definitions.
A \emph{saddle connection} on a flat cone surface is a geodesic that starts and ends at conical singularities without passing through any singularities in the interior. A \emph{regular closed geodesic} on a flat cone surface is a geodesic that returns to its starting point with the same tangent direction. The following result, which applies to flat cone surfaces of any genus, provides bounds on the lengths of these trajectories in terms of their self-intersection numbers.

\begin{restatable}{thm}{mainindividual}\label{thm:mainindividual}
Let $X$ be a unit area flat cone surface of genus $g$. For any trajectory $\gamma$ in $X$, there exist positive constants $b_1$ and $b_2$, depending only on the metric of $X$, such that
\begin{equation}\label{equ:compareindividual}
       b_1\sqrt{\iota(\gamma,\gamma)} - b_2 \le |\gamma|.
\end{equation}
Moreover, for any saddle connection or regular closed geodesic $\gamma$ in $X$, we have that
\begin{equation}\label{equ:compareindividual2}
       b_1\sqrt{\iota(\gamma,\gamma)} \le |\gamma|,
\end{equation}
where $b_1$ is the same constant as in the inequality~\eqref{equ:compareindividual}.
\end{restatable}

In~\eqref{eq:finalb1b2}, we construct explicit values for $b_1$ and $b_2$ based on certain geometric quantities of $X$.

\begin{remark}
    Note that the inequality~\eqref{equ:compareindividual} provides a bound for any trajectory $\gamma$, including a constant term $b_2$ in the lower bound, whereas the inequality~\eqref{equ:compareindividual2} applies specifically to saddle connections and regular closed geodesics without the constant term. It is important to mention that $b_2$ cannot be zero for all flat cone surfaces. A counterexample illustrating this is provided in Example~\ref{ex:b2notzero}.
\end{remark}

\begin{remark}\label{rmk:comparehyper1}
    In the case of hyperbolic surfaces, Theorem~1.1 of Basmajian's work~\cite{Ara} shows that for any compact hyperbolic surface $S$, there exists a constant $b_S$ depending on the hyperbolic metric of $S$ such that for any closed geodesic $\gamma$,
    \begin{equation}\label{equ:hyper1}
        b_S\sqrt{\iota(\gamma,\gamma)} \leq L(\gamma),
    \end{equation}
    where $L(\gamma)$ is the hyperbolic length of $\gamma$, and $b_S$ converges to zero as the systole of $S$ (the shortest length of a closed geodesic) tends to zero.
    \par
    Our result~\eqref{equ:compareindividual2} can thus be seen as an analogue of~\eqref{equ:hyper1} in the setting of flat surfaces. Furthermore, we show in Proposition~\ref{prop:estimatedegenerate} that the coefficient $b_1$ in~\eqref{equ:compareindividual2} also tends to zero as the systole of the flat cone surface goes to zero.
\end{remark}

We restrict our attention to the case of genus zero, that is, flat cone spheres. We define a flat cone sphere to be \textit{convex} if all curvatures are positive. In the setting of convex flat cone spheres, we obtain the second main result, where the lower bounds for the lengths of trajectories become uniform under the conditions of the theorem. Furthermore, we provide explicit computations of the constants involved in the bounds.

\begin{restatable}{thm}{firstmainfav}\label{thm:mainfavorite}
    Let $\delta$ be a positive real number. Let $X$ be a unit area convex flat cone sphere with $n$ singularities, and assume that the curvature gap is bounded below by $\delta$. For any trajectory $\gamma$ in $X$, there exist positive constants $c_1$ and $c_2$, depending only on $n$ and $\delta$, such that
    \begin{equation}\label{equ:compare}
           c_1\sqrt{\iota(\gamma,\gamma)} - c_2 \le |\gamma|.
    \end{equation}
    Moreover, for any regular closed geodesic $\gamma$ in $X$, we have that
    \begin{equation}\label{equ:compare2}
           c_1\sqrt{\iota(\gamma,\gamma)} \le |\gamma|.
    \end{equation}
\end{restatable}

More precisely, we prove that Theorem~\ref{thm:mainfavorite} holds for the following values (see Section~\ref{sec:main}):
\begin{equation}\label{equ:explicitvalues}
    c_1(n,\delta)=  \frac{9\delta^{3/2}}{144n^2(n-2)}\left(\frac{\delta^2}{6n}\right)^{2n-4}
        \quad \mbox{ and }\quad
        c_2(n)= \frac{81}{4} \left(\frac{1}{54n}\right)^{2n-4}.
\end{equation}

We particularly highlight Theorem~\ref{thm:combinandmetriclengthinhull} and Corollary~\ref{cor:finitenessinfinitesphere}, which are essential for deriving uniform lower bounds. These results show that the behavior of trajectories on infinite non-negative flat spheres with one pole (see Section~\ref{sec:basicdef} for definition) is quite simple. Specifically, there are \textit{no} regular closed geodesics, and only \textit{finitely many} saddle connections. We also provide an explicit, uniform upper bound on the number of these saddle connections on such flat spheres.

\begin{remark}
    One might wonder whether a uniform lower bound can be established for any flat cone surface, regardless of convexity. Unfortunately, this is not the case. In Example~\ref{ex:negativecurvature}, we present a flat cone surface with a singularity of negative curvature where Theorem~\ref{thm:mainfavorite} does not apply. This highlights that the convexity condition in Theorem~\ref{thm:mainfavorite} is crucial. Consequently, we restrict our focus to surfaces with positive curvatures, a condition that arises only in genus-zero cases.

    Besides, we show in Example~\ref{ex:deltaisnecessary} that the dependence of $c_1$ on the lower bound $\delta$ of the curvature gap cannot be eliminated.

    In addition, Example~\ref{ex:c2notzero} demonstrates that the constant term $c_2$ cannot be identically zero for all convex flat cone spheres.
\end{remark}

\begin{remark}\label{rmk:comparehyper2}
    In the case of hyperbolic surfaces, Basmajian showed in~\cite{Ara}, Theorem~1.2, that for any closed geodesic $\omega$ in any hyperbolic surface,
    \begin{equation}\label{equ:hyper2}
        \frac{1}{2}\log\frac{|\iota(\omega,\omega)|}{2} \le L(\omega),
    \end{equation}
    where $L(\omega)$ is the hyperbolic length of $\omega$. Notably, Basmajian's lower bound is universal for all hyperbolic surfaces of any genus, providing a logarithmic dependence of the lower bound on the self-intersection number $\iota(\omega,\omega)$. In contrast, our result for convex flat cone spheres provides a uniform lower bound as shown in~\eqref{equ:compare2}, where the order of $\iota(\gamma, \gamma)$ remains as in~\eqref{equ:compareindividual2} but our bound depends on the curvature gap and the number of singularities in the flat cone sphere. 
\end{remark}

\begin{remark}
    Combining Theorem~\ref{thm:finite} with the results from Theorem~\ref{thm:mainfavorite}, we can approximately relate the length of a trajectory to the square root of its self-intersections. Specifically, for any unit area convex flat cone sphere $X$ with $n$ singularities and curvature gap bounded below by $\delta$, the following inequalities hold:
    \begin{equation}\label{equ:finalest0}
        c_1\sqrt{\iota(\gamma,\gamma)} - c_2 \le |\gamma| \le a_1\sqrt{\iota(\gamma,\gamma)} + a_2,
    \end{equation}
    for any trajectory $\gamma$ in $X$, and
    \begin{equation}\label{equ:finalest1}
        c_1\sqrt{\iota(\gamma,\gamma)} \le |\gamma| \le 2a_1\sqrt{\iota(\gamma,\gamma)},
    \end{equation}
    for any regular closed geodesic $\gamma$ in $X$.
\end{remark}

It is worth noting that we do not address the question of the optimal coefficients for the estimate~\eqref{equ:finalest0} in this paper. Consider the following quantities:
\begin{equation}
    \begin{aligned}
        &c^{sc}(n, \delta) := \inf\limits_{X} \liminf_{\gamma} \frac{|\gamma|}{\sqrt{\iota(\gamma,\gamma)}},\\
        &a^{sc}(n, \delta) := \sup\limits_{X} \limsup_{\gamma} \frac{|\gamma|}{\sqrt{\iota(\gamma,\gamma)}}
    \end{aligned}
\end{equation}
where the infimum and supremum are taken over all unit-area convex flat cone spheres $X$ with $n$ singularities and curvature gap at least $\delta$, and $\gamma$ runs over saddle connections on $X$.
An interesting question is whether we can further explore the dependence of $c^{sc}(n,\delta)$ and $a^{sc}(n,\delta)$ on $n$ and $\delta$. For instance, from the formulas~\eqref{equ:upperexplicivalues} and~\eqref{equ:explicitvalues},
we obtain that 
$$c_1(n,\delta)\le c^{sc}(n,\delta)\le a^{cs}(n,\delta)\le 2a_1(n,\delta).$$
The difference $2a_1(n,\delta) - c_1(n,\delta)$ tends to infinity as $\delta$ approaches zero with $n$ fixed. Is this also true for $a^{sc}(n, \delta) - c^{sc}(n, \delta)$?

\subsection{Applications to counting problems}
Let $X$ be a flat cone sphere with area one. Let $R$ be a positive real number. In Section~\ref{sec:appli}, we define $N^{sc}(X,R)$ as the number of saddle connections on $X$ with length at most $R$, and $N^{cg}(X,R)$ as the number of maximal families of regular closed geodesics with length at most $R$. The following corollary implies that these counting functions are uniformly bounded above over the moduli space of convex flat cone spheres.

\begin{cor}\label{cor:integrability}
    Let $X$ be an area-one convex flat cone sphere with $n$ singularities and positive curvature gap $\delta$. Then,
    $$
    N^{cg}(X,R) \le N^{sc}(X,R) \le (3n-6)2^{20n\delta^{-1}\left(c_1^{-1}(n-1)(R + c_2) + 1\right)},
    $$
    where $c_1$ and $c_2$ are as in Theorem~\ref{thm:mainfavorite}.
\end{cor}

\begin{remark}
    In the context of translation surfaces, the work of Athreya, Cheung, and Masur~\cite{athreya2019siegel} implies that the counting function of saddle connections over the moduli space of translation surfaces is square integrable. Surprisingly, the counting function is uniformly bounded when restricted to convex flat cone spheres.
\end{remark}

In addition, we apply Theorem~\ref{thm:mainfavorite} to convex polygonal billiards. Given a polygon $P$ with $n$ vertices in the plane, let $X_P$ be the flat cone sphere constructed by doubling $P$ (see Section~\ref{sec:appli}). For a vertex $x_i$ of $P$ with interior angle $\theta_i$, the corresponding cone angle of $X_P$ is $2\theta_i$. We define the \emph{curvature} $k_i$ of this vertex of $P$ to be the curvature of the corresponding singularity on $X_P$, namely, 
$$k_i = \frac{2\pi - 2\theta_i}{2\pi} = \frac{\pi - \theta_i}{\pi}.$$ 
The curvatures of the vertices of $P$ are all positive if and only if $P$ is a convex polygon.

Similarly, we define the \emph{curvature gap} of a polygonal billiard $P$ with $n$ vertices to be the curvature gap of the associated flat cone sphere $X_P$, namely,
\begin{equation}
        \inf_{I \subseteq \{1,\ldots,n\}}
    \left|1 - \sum_{i \in I} k_i\right| = 
    \inf_{I \subseteq \{1,\ldots,n\}}
    \left|1 - \sum_{i \in I} \frac{\pi - \theta_i}{\pi}\right|,
\end{equation}
where $\theta_i$ denotes the interior angle at the $i$-th vertex of $P$.

\begin{Ex}
    Let $P$ be a triangle with angles $\theta_1$, $\theta_2$, and $\theta_3$. Then the curvature vector of the corresponding billiard $P$ is 
$$
\underline{k} = (k_1,k_2,k_3) = \left(\frac{\pi - \theta_1}{\pi}, \frac{\pi - \theta_2}{\pi}, \frac{\pi - \theta_3}{\pi}\right).
$$
By definition, the curvature gap of $\underline{k}$ is
\begin{equation}\label{equ:curvaturegapfortriangle}
    \begin{aligned}
        &\delta = \inf\limits_{I \subset \{1, 2, 3\}} \left|1 - \sum\limits_{\substack{i \in I}} k_i\right|\\
        &= \min\left\{1-k_1, 1-k_2, 1- k_3, |1-k_1-k_2|,  |1-k_1 - k_3|, |1- k_2-k_3|, |1-k_1-k_2-k_3|\right\}
    \end{aligned}
\end{equation}
Since $k_1+k_2 + k_3 = 2$, we have $|1-k_1-k_2-k_3| = 1.$
Moreover,
$$1-k_i = 1-\frac{\pi-\theta_i}{\pi} = \frac{\theta_i}{\pi}$$
and
\begin{align*}
    |1-k_i-k_j| = \left|1 -\frac{\pi-\theta_i}{\pi} - \frac{\pi-\theta_j}{\pi}\right|= \left|1 -\frac{2\pi-\theta_i - \theta_j}{\pi}\right|= \left|1 -\frac{\pi+\theta_k}{\pi}\right| = \frac{\theta_k}{\pi},
\end{align*}
where $\{i,j,k\} = \{1,2,3\}$.
Therefore, the expression~\eqref{equ:curvaturegapfortriangle} simplifies to
\begin{align*}
    \delta = \min\left\{\frac{\theta_1}{\pi},\frac{\theta_2}{\pi},\frac{\theta_3}{\pi},\frac{\theta_3}{\pi},\frac{\theta_2}{\pi},\frac{\theta_1}{\pi},1\right\}=\frac{1}{\pi}\min\{\theta_1, \theta_2, \theta_3\}.
\end{align*}
Hence, a triangular billiard always has a positive curvature gap.
\end{Ex}

In Section~\ref{sec:appli}, we define $N^{diag}(P,R)$ as the number of generalized diagonals with lengths at most $R$ in $P$, and $N^{per}(P,R)$ as the number of maximal families of parallel periodic billiard paths with lengths at most $R$ in $P$. The following bounds are a rephrasing of Corollary~\ref{cor:integrability} for convex polygonal billiards.

\begin{cor}\label{cor:billiard}
    Let $P$ be an area-one convex polygonal billiard (possibly irrational) with $n$ vertices and positive curvature gap $\delta > 0$. Then,
    $$
    N^{diag}(P,R) \le (3n-6)2^{20n\delta^{-1}\left(c_1^{-1}(n-1)\left(\frac{1}{\sqrt{2}}R + c_2\right) + 1\right)},
    $$
    and
    $$
    N^{per}(P,R) \le (3n-6)2^{20n\delta^{-1}\left(c_1^{-1}(n-1)(\sqrt{2}R + c_2) + 1\right)},
    $$
    where $c_1$ and $c_2$ are as in Theorem~\ref{thm:mainfavorite}.
\end{cor}

\begin{remark}
    Compared with the sub-exponential upper bounds in~\cite{Kat, Sch}, our upper bound for $N^{diag}(P,R)$ is exponential, but it is uniform for all convex polygons whose curvature gaps are bounded below by a positive number. In particular, for any triangle with angles no less than $\pi\delta$, the above exponential upper bounds hold uniformly.
\end{remark}

\begin{remark}\label{rmk:futurework}
    The main results and tools introduced in this article, particularly those developed in Section~\ref{bigsec:hulls} and Section~\ref{bigsec:coor}, play a crucial role in our forthcoming work~\cite{Fu25}. In~\cite{Fu25}, we aim to establish a formula for the average of the counting functions over the moduli space of convex flat cone spheres. Corollary~\ref{cor:integrability} is used to show that the counting functions are $L^{\infty}$ over the moduli space. This forthcoming research builds directly on the foundations laid here.
\end{remark}

\subsection{Proof strategy}
For Theorem~\ref{thm:mainindividual}, the key observation is that a trajectory on a flat cone surface, given a triangulation, ``switches corners" after a bounded time. This observation is detailed in Section~\ref{bigsec:individual}.

To extend this lower bound to a uniform one, we recall the concept of moduli spaces. Specifically, the moduli space of convex flat cone spheres with curvatures $\underline{k} = (k_1,\ldots,k_n)$ is defined as the space of homothety classes of flat spheres with prescribed curvatures $k_1,\ldots,k_n$ at singularities (see Section~\ref{sec:basicdef}), denoted by $\mathbb{P}\Omega(\underline{k})$.

Deligne and Mostow initially constructed complex hyperbolic metrics on these moduli spaces for certain rational curvatures~\cite{DM}. Later, Thurston extended this by constructing a complex hyperbolic metric on any moduli space $\mathbb{P}\Omega(\underline{k})$ and discussing its metric completion~\cite{thu}. McMullen was the first to compute the total volume of these moduli spaces of convex flat cone spheres~\cite{McMullen}. Koziarz and Nguyen proposed an alternative method for computing volumes by analyzing the intersection theory of boundary divisors in the moduli spaces~\cite{KN}. Sauvaget~\cite{sauvaget2024volumesmodulispacesflat} was the first to prove the volume formulas for moduli spaces of flat cone surfaces with higher genus and rational (not necessarily positive) curvatures. Nguyen later extended these formulas to any real curvatures in~\cite{Ngu24}.

To prove Theorem~\ref{thm:mainfavorite}, we use a "thick-and-thin" decomposition of the moduli space. The spheres in a thick part share a common lower bound on the lengths of their saddle connections. We apply Theorem~\ref{thm:mainindividual} to establish a uniform lower bound on the saddle connections for spheres in the thick part.

For spheres in thin parts, we apply \emph{generalized Thurston surgeries} to collapse singularities that are close to each other. In this way, we convert length estimates for the thin parts to those of thick parts of lower-dimensional moduli spaces. 

Additionally, to derive a uniform lower bound for spheres in the thin parts, Theorem~\ref{thm:combinandmetriclengthinhull} plays a crucial role. This result roughly implies that the length of a trajectory contained within a relatively small area on a flat cone sphere is bounded. This behavior contrasts sharply with translation surfaces, where arbitrarily long trajectories may exist in certain small areas.

Through this approach, we achieve a uniform lower bound on the length of trajectories.

\subsection{Organization}
In Section~\ref{bigsec:moduli}, we define flat surfaces and their moduli spaces, and recall the necessary background.

In Section~\ref{bigsec:individual}, we prove Theorem~\ref{thm:mainindividual}, which addresses the case of trajectories in flat cone surfaces of any genus.

Starting from Section~\ref{bigsec:hulls}, we focus on convex flat cone spheres and the proof of Theorem~\ref{thm:mainfavorite}.

In Section~\ref{bigsec:hulls}, we introduce the concept of convex hulls and generalized Thurston surgeries. These tools are then utilized in Section~\ref{bigsec:coor} to understand the structure of the ``thin" parts of the moduli space. Furthermore, we apply these tools to prove Theorem~\ref{thm:combinandmetriclengthinhull}, which is crucial for the proof of Theorem~\ref{thm:mainfavorite} in Section~\ref{bigsec:ivslen}.

In Section~\ref{bigsec:coor}, we introduce a concrete decomposition of the moduli space by ``thick" and ``thin" parts. 

In Section~\ref{bigsec:ivslen}, we provide the proof of Theorem~\ref{thm:mainfavorite} by separately analyzing the surfaces contained in the thick and thin parts of the moduli space.

\vspace*{1\baselineskip}

\paragraph{\bf Acknowledgements.}
I would like to express my sincere gratitude to my advisors, Vincent Delecroix and Elise Goujard, for their helpful guidance and discussions during the preparation of this paper. I also thank the anonymous referee for helpful comments and suggestions.

\section{Flat Surfaces and Moduli Spaces}\label{bigsec:moduli}

In this section, we define flat surfaces and introduce the necessary background, following the definitions of flat surfaces and moduli spaces from~\cite{Veech, thu, Ngu24}.

We begin by clarifying some basic notions related to curves on a surface. Let $M$ be a Riemann surface. A \emph{curve} on $M$ is a continuous map from $[0,1]$ to $M$. Similarly, a \emph{loop} is a continuous map from $\mathbb S^1$ to $M$. A curve is said to be \emph{simple} if it is injective on the interval $(0,1)$, and a loop is \emph{simple} if it is injective on $\mathbb{S}^1$. Let $x_1,\ldots, x_n$ be labeled distinct points on a Riemann surface $M$. An \emph{arc} on $(M, x_1,\ldots, x_n)$ refers to a curve that starts and ends at labeled points, with no labeled points in its interior. 

\subsection{Flat metrics}\label{sec:basicdef}

Let $M$ be a Riemann surface of genus $g$. Let $\underline{k}$ be a real vector $(k_1, \ldots, k_n) \in \mathbb{R}^n$ and let $x_1, \ldots, x_n$ be $n$ distinct labeled points in $M$. A \emph{flat metric} $h$ on $M$ with curvature $k_i$ at $x_i$ is a conformal metric on $M$ satisfying the following conditions:
\begin{itemize}
    \item There is a conformal chart on a neighborhood of $x \in M \setminus \{x_1, \ldots, x_n\}$ such that $h$ is expressed as $|dz|^2$ in the chart.
    \item There is a conformal chart on a neighborhood $U_i$ of $x_i$ such that $h$ is expressed as $|z|^{-2k_i}|dz|^2$ in the chart.
\end{itemize}

We call the above holomorphic charts for the points in $M \setminus \{x_1, \ldots, x_n\}$ \emph{flat charts}. A \emph{flat surface} $X$ is defined as the triple $(M, \{x_1, \ldots, x_n\}, h)$. The points $x_1, \ldots, x_n$ are called \emph{singularities} of $X$ and $\underline{k}$ is called the \emph{curvature vector} of $X$. The Gauss-Bonnet formula implies that if a flat metric with curvature $k_i$ at $x_i$ exists on $M$, then
\begin{equation}
    \sum_{i=1}^{n} k_i = 2 - 2g.
\end{equation}

A singularity is called \emph{conical} if its curvature is less than $1$; otherwise, it is called a \emph{pole}. If $x_i$ is a conical singularity, the neighborhood $U_i$ is isometric to a neighborhood of the apex of the infinite cone with cone angle $2\pi(1 - k_i)$. When $x_i$ is a pole, the neighborhood $U_i$ is isometric to the complement of a neighborhood of the apex of the infinite cone with cone angle $2\pi(k_i - 1)$. 

A flat surface is called a \emph{flat cone surface} if all the singularities are conical. A flat cone surface $X$ can be obtained by gluing finitely many disjoint polygons in the plane with a pairing between the edges such that two paired edges are glued together by a Euclidean isometry.

Given a flat surface $X$, a \emph{developing map} associated to $X$ is an orientation-preserving and locally isometric map from $\widetilde{X}$ to $\mathbb{C}$, where $\widetilde{X}$ denotes the metric completion of the universal covering of $X \setminus \{x_1, \ldots, x_n\}$ endowed with the pull-back metric of $X$. We denote such a map by $\mathrm{Dev}: \widetilde{X} \to \mathbb{C}$.

The \emph{holonomy} associated to $X$ and $\mathrm{Dev}$ is the unique morphism  
\[
\rho: \pi_1(X \setminus \{x_1, \ldots, x_n\}) \to \operatorname{Iso}^+(\mathbb{R}^2)
\]
that satisfies  
\[
\mathrm{Dev}(\gamma \cdot p) = \rho(\gamma)\,\mathrm{Dev}(p),
\]
for every $\gamma \in \pi_1(X \setminus \{x_1, \ldots, x_n\})$ and $p \in \widetilde{X}$.  
Here $\operatorname{Iso}^+(\mathbb{R}^2)$ denotes the group of orientation-preserving isometries of the Euclidean plane $\mathbb{R}^2$.

For further details on developing maps and holonomy, see~\cite[Section~3.4]{Thurston1997}.

Let $X_1$ and $X_2$ be two flat surfaces. A homeomorphism $f: X_1 \to X_2$ is called a \emph{homothety} if it is of the form $az + b$ in the flat charts of flat metrics of $X_1$ and $X_2$, where $a \in \mathbb{C}^*$ and $b \in \mathbb{C}$. The absolute value $|a|$ is independent of the choice of flat charts. The map $f$ is an isometry if and only if $|a| = 1$.

We are in particular interested in the genus zero case.
A flat cone surface $X$ is called a \emph{flat cone sphere} if the underlying Riemann surface is of genus zero.
A flat surface with at least one pole is called an \emph{infinite flat  sphere}. Indeed, the area of an infinite flat sphere is infinite. A flat sphere is called \emph{convex} if all the curvatures are positive. A flat sphere is called \emph{non-negative} if all the curvatures are non-negative.

Recall that for a flat cone sphere, the curvature gap~\eqref{equ:gap} is defined by
\begin{equation}
    \delta(\underline{k}) = \inf_{I \subset \{1, \ldots, n\}} \left| 1 - \sum_{i \in I} k_i \right|.
\end{equation}

\begin{lemma}[Lemma~2.2, \cite{FT}]\label{lem:upperboundcurvaturegap}
    The curvature gap $\delta$ of any flat cone sphere satisfies $\delta \leq \frac{1}{3}$.
\end{lemma}

For later use, we generalize the notion of the curvature gap~\eqref{equ:gap} for a flat sphere with arbitrary curvature vector $\underline{k} = (k_1, \ldots, k_n)$ (possibly with poles) by
\begin{equation}
    \delta(\underline{k}) = \inf_{I \subset \{1, \ldots, n\}} \left| 1 - \sum_{i \in I, k_i < 1} k_i \right|.
\end{equation}

\subsection{Flat surfaces with boundaries}

For later use, we introduce flat surfaces with boundaries. Let $M$ be a compact Riemann surface with boundary. The \emph{boundary} and \emph{interior} of a Riemann surface with boundary refer to the boundary and interior of the underlying topological surface.
 A \emph{flat metric} $h$ on $M$ is a conformal metric on $M$ such that there exist finitely many points $x_1, \ldots, x_n$ in the interior of $M$ and finitely many points $y_1, \ldots, y_m$ on the boundary of $M$ satisfying that:
\begin{itemize}
    \item There is a conformal chart on a neighborhood $U_x$ of $x \in M$ away from the labeled points such that $h$ is expressed as $|dz|^2$ in the chart.
    \item There is a conformal chart on a neighborhood $U_i$ of $x_i$ such that $h$ is expressed as $|z|^{-2k_i}|dz|^2$ in the chart for some real number $k_i$.
    \item There is a conformal chart on a neighborhood $V_y$ of $y \in \partial M$ such that
    \begin{itemize}
        \item If $y$ is different from $y_1, \ldots, y_m$, then $V_y$ equipped with $h$ is isometric to a half-disk.
        \item If $y$ is the same as some $y_i$, then $V_{y_i}$ equipped with $h$ is isometric to a sector of some angle $A_i$.
    \end{itemize}
\end{itemize}

We call $X = (M,\{x_1,\ldots,x_n,y_1,\ldots,y_m\}, h)$ as a \emph{flat surface with boundary}. The points $x_1, \ldots, x_n$ and $y_1, \ldots, y_m$ are called \emph{singularities} of $X$, and $k_i$ is called the \emph{curvature} at $x_i$. The neighborhood $V_{y_i}$ of $y_i$ is called a \emph{corner} of $X$ at $y_i$, and $A_i$ is called the \emph{angle of the corner}. A singularity $x_i$ is called \emph{conical} if its curvature is less than one; otherwise, it is called a \emph{pole}. We remark that we do not consider poles on the boundary of flat surfaces.

To provide further clarity, we now introduce the following definition, which will be useful for our discussion on flat surfaces:

\begin{definition}\label{def:flatdomain}
    We say that a flat surface with boundary is a \emph{flat domain} if the underlying Riemann surface is a closed disk. We call a flat domain \emph{convex} if it has no poles and the cone angle of each singularity in the interior is less than $2\pi$.
\end{definition}

\subsection{Delaunay triangulation}\label{sec:delaunay}

In this section, we recall the definition of Delaunay triangulations of a flat cone surface. For more information on Delaunay triangulations, one can consult~\cite{thu, MasurSmillie1991}.

Let $x_1, \ldots, x_n$ be $n$ distinct labeled points in a Riemann surface $M$. A triangulation of $(M, x_1, \ldots, x_n)$ is a triangulation of $M$ whose vertices are $x_1, \ldots, x_n$ and the edges are simple arcs.

Let $X=(M,\{x_1,\ldots,x_n\},h)$ be a flat cone surface. A triangulation $T$ of $X$ is called \emph{geometric} if it is a triangulation of $(M, x_1, \ldots, x_n)$ whose edges are saddle connections.

A geometric triangulation $T$ of $X$ is called \emph{Delaunay} if each triangle is contained in an immersed disk such that no singularities lie in the interior of the disk image. According to~\cite{thu,MasurSmillie1991}, there exists a Delaunay triangulation for any flat cone surface, and there are only finitely many Delaunay triangulations for a given flat metric.

\subsection{Moduli spaces of flat surfaces}\label{sec:modulispace}

Given a vector $\underline{k} = (k_1, \ldots, k_n) \in \mathbb{R}^n$ such that $\sum_{i=1}^{n} k_i = 2 - 2g$, let $\mathbb{P}\Omega(\underline{k})$ denote the space of flat surfaces, considered up to homothety, with singularities $x_1, \ldots, x_n$, where each singularity $x_i$ has curvature $k_i$. We use the notation $[X]$ to represent the homothety class of $X$ in $\mathbb{P}\Omega(\underline{k})$. This space $\mathbb{P}\Omega(\underline{k})$ is called the \emph{moduli space} of flat surfaces with curvature vector $\underline{k}$.

It is worth noting that Troyanov showed in~\cite{Tro} that for a given $\underline{k}=(k_1,\ldots,k_n)$ with $\sum_i k_i = 2-2g$ and $k_i<1$ on any Riemann surface $M$ with $n$ distinct labeled points $x_1, \ldots, x_n$, there exists a unique flat metric, up to homothety, with curvature $k_i$ at $x_i$. This implies that $\mathbb{P}\Omega(\underline{k})$ is in bijection with the moduli space of Riemann surfaces with $n$ labeled points $\mathcal{M}_{0,n}$.

We next focus on the genus zero case and recall a construction of coordinate charts on the moduli spaces of flat spheres in~\cite{thu}.

Let $(M,x_1,\ldots,x_n)$ be a Riemann sphere with $n$ distinct labeled points. A \emph{forest} $F$ in $(M,x_1,\ldots,x_n)$ is the image of an embedding of a forest such that the vertices map to labeled points and the edges map to simple arcs. We refer to the \emph{vertices} and \emph{edges} of a forest as the images of the vertices and edges of the forest. A \emph{forest} is called a \emph{tree} if it has only one connected component. A tree in $(M,x_1,\ldots,x_n)$ is called \emph{spanning} if it connects all the labeled points. Two forests $F$ and $F'$ on two Riemann spheres with $n$ labeled points are said to be \emph{equivalent} if there is an orientation-preserving homeomorphism between two Riemann spheres with respect to the labeled points such that $F$ maps to $F'$.

Let $X=(M,\{x_1,\ldots,x_n\},h)$ be a flat cone sphere. The notion of forest is generalized to $X$ directly. A forest $F$ in $X$ is called \emph{geometric} if all the edges are saddle connections.

Let $F$ be a spanning tree in $X$. For later use, we always assume that $F$ is oriented and labeled, meaning each edge of $F$ is oriented and the edges of $F$ are labeled $e_0, \ldots, e_{n-2}$. Assume that for $[X]\in\mathbb P\Omega(\underline{k})$ has a geometric spanning tree $F$. We recall the construction of a coordinate chart in $\mathbb P\Omega(\underline{k})$ near $[X]$ as follows.

Note that  $X \setminus F$ is simply connected. Any developing map of $X$ can be induced to $X \setminus F$. Fix one developing map $Dev$ on $X\setminus F$. Then $X \setminus F$ maps to a polygon (possibly immersed) in $\mathbb{C}$. Each edge of $f(F)$ is mapped to two vectors in the boundary of the polygon, differing by a rotation. Since $F$ is oriented, we select the vector on the boundary so that $X \setminus F$ maps to the left side of the vector, and denote such vectors corresponding to edges $e_i$ by $z_i$ for $i=0,\ldots,n-2$ (note that $z_{n-2}$ is a complex linear combination of $z_0, \ldots, z_{n-3}$). Thurston~\cite{thu} showed that a neighborhood of the vector $(z_0, z_1, \ldots, z_{n-3}) \in \mathbb{C}^{n-2}$ parameterize the flat spheres in a neighborhood $U$ of $X$ in $\mathbb P\Omega(\underline{k}$). Specifically, the vector $(z_0',\ldots,z_{n-2}')$ induces a geometric spanning tree in $[X']$ that is equivalent to $F$. Moreover, a neighborhood of $[z_0, z_1, \ldots, z_{n-3}] \in \mathbb{C}P^{n-3}$ gives a coordinate chart on $U$ and we call it a \emph{spanning tree coordinate chart} associated to the (oriented and labeled) spanning tree $F$. 

When using the spanning tree charts, we usually normalize the vector $z_0$ to $1$, so the coordinates of surfaces in the chart $U$ are also given by $(z_1, \ldots, z_{n-3})$ in $\mathbb{C}^{n-3}$.

\subsection{Metric completions}\label{sec:completion}
We recall the complex hyperbolic metrics on the moduli spaces of convex flat cone spheres, constructed by Thurston in~\cite{thu}.

We first introduce complex hyperbolic spaces. Given a Hermitian form $H$ on $\mathbb{C}^{n+1}$ with signature $(1, n-1)$, the form $-H$ induces a Riemannian metric on the projectivized space of all positive vectors $\{v \in \mathbb{C}^{n+1} \mid H(v,v) > 0\}$. Denote the projectivized space of $\{v \in \mathbb{C}^{n+1} \mid H(v,v) > 0\}$ by $\mathbb{H}_{\mathbb{C}}^n$ and the induced metric by $g$. We call $(\mathbb{H}_{\mathbb{C}}^n, g)$ a \emph{complex hyperbolic space}.

Thurston constructed in~\cite{thu} a complex hyperbolic metric on the moduli space of flat cone spheres $\mathbb{P}\Omega(\underline{k})$. Indeed, for $[X]\in\mathbb P\Omega(\underline{k})$, let $F$ be an oriented and labeled spanning tree in $X$. Let $U$ be the corresponding spanning tree coordinate chart of $[X]$. Recall that in the construction of this chart, there is  a domain $V'$ in $\mathbb C^{n-2}$ such that the projection into $\mathbb CP^{n-3}$ gives the domain $V$ of the spanning tree chart.

The area function $\mathrm{Area}$ can be written as a Hermitian form on $V'\subset\mathbb{C}^{n-2}$.
Furthermore, if the convexity is assumed for flat cone spheres, Thurston showed in Proposition 3.3 of~\cite{thu} that the signature of this Hermitian form is $(1, n-3)$. It follows that $\mathrm{Area}$ induces a complex hyperbolic metric after projectivization to $V$. Thurston showed that this metric does not depend on the choice of spanning tree charts. We denote this metric on $\mathbb{P}\Omega(\underline{k})$ by $h_{\mathrm{Thu}}$ and the induced measure by $\mu_{\mathrm{Thu}}$.

In general, the metric space $(\mathbb{P}\Omega(\underline{k}), h_{\mathrm{Thu}})$ is not metric complete. In~\cite{thu}, Thurston studied the metric completion of the moduli space of convex flat cone spheres. We denote by $\overline{\mathbb{P}\Omega}(\underline{k})$ the metric completion of $\mathbb{P}\Omega(\underline{k})$ with respect to the metric $h_{\mathrm{Thu}}$.

To describe the structure of the metric completion, we introduce the following notions. Let $P$ be a partition of $(k_1, \ldots, k_n)$. We say that $P$ is \emph{admissible} if, for any set $p \in P$, the sum of curvatures in $p$ is less than $1$.  
Let $P = (p_1, \ldots, p_m)$ be an admissible partition. We say that the vector
$$
\Big(\sum_{k \in p_1} k, \ldots, \sum_{k \in p_m} k \Big)
$$
is obtained by \emph{collapsing} $\underline{k}$ by $P$, and denote it by $\underline{k}_P$.

A \emph{stratification} of $\overline{\mathbb{P}\Omega}(\underline{k})$ is a decomposition of the space into disjoint topological subspaces. Thurston introduced a stratification in~\cite{thu} corresponding to the admissible partitions of curvatures $\underline{k} = (k_1, \ldots, k_n)$. We summarize the information about this stratification in the following:

\begin{lemma}[\cite{thu}, Theorem 0.2]\label{lem:thurstoncompletion}
    Let $\mathbb{P}\Omega(\underline{k})$ be a moduli space of convex flat cone spheres with $n$ singularities. The metric completion $\overline{\mathbb{P}\Omega}(\underline{k})$ is an $(n-3)$-dimensional complex hyperbolic cone manifold, and it has a stratification such that:
    \begin{itemize}
        \item The strata in the stratification are in bijection with the admissible partitions of $\underline{k} = (k_1, \ldots, k_n)$.
        \item The stratum associated with an admissible partition $P$ is isometric to $\mathbb{P}\Omega(\underline{k}_P)$, and the codimension of the stratum is $\sum_{p \in P} (|p| - 1)$, where $|p|$ is the number of curvatures in $p$.
    \end{itemize}
    Moreover, the metric completion $\overline{\mathbb{P}\Omega}(\underline{k})$ is compact if and only if the curvature gap $\delta(\underline{k})$ is positive.
\end{lemma}

Notice that $\mathbb{P}\Omega(\underline{k})$ is the stratum of $\overline{\mathbb{P}\Omega}(\underline{k})$ corresponding to the partition of $(k_1, \ldots, k_n)$ into singletons. We call such a partition the \emph{trivial} admissible partition of $(k_1, \ldots, k_n)$. A \emph{boundary stratum} of $\overline{\mathbb{P}\Omega}(\underline{k})$ is a stratum corresponding to a non-trivial admissible partition.

\subsection{Thurston surgeries}\label{sec:thurstonsurgery}

Thurston introduced in~\cite{thu} a surgery to collapse two singularities. We now describe the surgery (the description mainly follows Section~6.3 of~\cite{KN}).

Let $X$ be a flat surface, and let $x_i$ and $x_j$ be two singularities in $X$. Denote the curvatures of $x_i$ and $x_j$ by $k_i$ and $k_j$, respectively. Assume that $k_i$ and $k_j$ are non-negative and that $k_i + k_j < 1$. Let $\gamma$ be a simple saddle connection joining $x_i$ and $x_j$.

We first consider the case where both $k_i$ and $k_j$ are positive. Let $(ABC)$ be a triangle in the plane with angles $\pi(1 - k_i - k_j)$, $\pi k_i$, and $\pi k_j$ at $A$, $B$, and $C$, respectively, and let $|BC| = |\gamma|$. Let $(A', B', C')$ be the mirror image of $(A, B, C)$. We glue the edge $AC$ to $A'C'$ and glue $AB$ to $A'B'$ by Euclidean isometry, respecting the endpoints. This results in a bounded cone, with the edges $BC$ and $B'C'$ forming the boundary of the cone. By direct computation, the curvature at the apex is $k_i + k_j$. We denote this cone by $C_{\gamma}$.

The \emph{Thurston surgery of $X$ along $\gamma$} is the operation that slits $\gamma$ open and glues $C_{\gamma}$ in such a way that $x_i$ (resp. $x_j$) is identified with $B$ (resp. $C$). Denote the new flat surface by $X^{(0)}$. Note that:
\begin{itemize}
    \item $x_i$ and $x_j$ become regular points in $X^{(0)}$.
    \item The apex of $C_{\gamma}$ becomes a new singularity in $X^{(0)}$ with curvature $k_i + k_j$.
\end{itemize}

When one of $k_i$ or $k_j$ is zero, the \emph{Thurston surgery of $X$ along $\gamma$} forgets the singularity $x_i$ or $x_j$ respectively. If both $k_i$ and $k_j$ are zero, the surgery forgets the singularity with the larger index. We still denote the resulting surface by $X^{(0)}$. In this case, one can think of $C_{\gamma}$ as being degenerated.

\begin{figure}[!htbp]
	\centering
	\begin{minipage}{0.8\linewidth}
	\includegraphics[width=\linewidth]{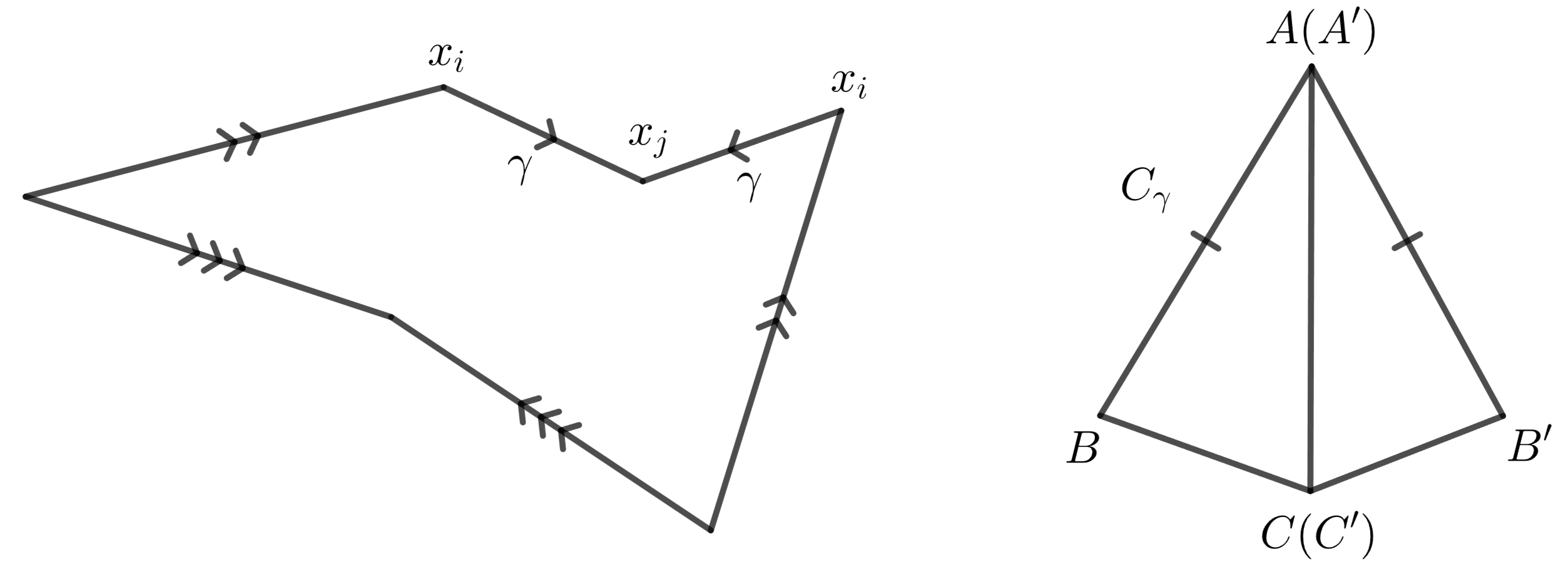}
	\footnotesize
	\end{minipage}
    \caption{Thurston surgery glues $C_{\gamma}$ on the right along the slit saddle connection $\gamma$ in the left flat cone sphere.}
\end{figure}

When we cut $\gamma$ open, the flat metric of $X$ induces a length metric on the complement $X \setminus \gamma$, meaning that the distance between two points in $X \setminus \gamma$ is the infimum of the lengths of the paths in $X \setminus \gamma$ joining them. Let $X_{\gamma}$ denote the metric completion of $X \setminus \gamma$ with the induced length metric. Note that $X_{\gamma}$ is a flat surface with boundary, and the boundary consists of two copies of $\gamma$.

\begin{lemma}\label{lem:remains}
    Let $X$ be a flat surface, and let $x_1, x_2$ be two singularities in $X$. Assume that the curvatures $k_1$ and $k_2$ are positive, and $k_1 + k_2 < 1$. Then $X_{\gamma}$ is isometrically embedded in $X^{(0)}$ in the sense that the embedding induces an isometry of the metrics on the tangent spaces.
\end{lemma}

\begin{proof}
    Since $X^{(0)}$ is obtained by adding a cone to the complement $X \setminus \gamma$, there is a canonical embedding from $X \setminus \gamma$ into $X^{(0)}$, which induces an isometry on the tangent spaces. Since the inclusion is uniformly continuous, and $X_{\gamma}$ is the metric completion of $X \setminus \gamma$, the embedding extends to the whole of $X_{\gamma}$ and remains isometric, as stated in the lemma.
\end{proof}

\section{Proof of Theorem~\ref{thm:mainindividual}}\label{bigsec:individual}

In this section, we prove Theorem~\ref{thm:mainindividual}. For convenience, we recall the statement of the theorem.

\mainindividual*

The explicit formulas for $b_1$ and $b_2$ in our proof are provided in~\eqref{eq:finalb1b2}. The key observation is that for a triangulation on a flat cone surface, a trajectory ``switches corners" after a bounded time.

\subsection{Corner-switches}\label{sec:cornerswitching}

Let $X$ be a flat cone surface, and let $T$ be a geometric triangulation of $X$. Recall that a geometric triangulation of $X$ has the singularities of $X$ as its vertices, and the edges are simple saddle connections. A \emph{trajectory} in $X$ is a geodesic of finite length that passes through singularities only at its endpoints.

Let $\gamma$ be a trajectory in $X$. Assume that $\gamma$ is oriented by its parameterization and is not contained within an edge of $T$. Then $\gamma$ is divided into sub-paths by the triangles of $T$. We call each sub-path of $\gamma$ that is cut out by a triangle of $T$ a \emph{thread} of $\gamma$ within the triangle. We denote the threads of $\gamma$ by $t_1, \ldots, t_m$, where $t_i$ and $t_{i+1}$ are consecutive threads along $\gamma$.

We say that a thread \emph{cuts a corner} of a triangle on the left (resp. right) side if there is a corner of the triangle on the left (resp. right) side of the thread, while the other two corners of the triangle lie on the opposite side of the thread. Note that a thread may start or end at a vertex of a triangle, in which case we say that the thread cuts a corner on both sides of the triangle.

We say that a trajectory $\gamma$ \emph{switches corners} of $T$ if two consecutive threads of $\gamma$ cut corners on opposite sides of $\gamma$. We also refer to such a pair of threads as a \emph{corner-switch}. Conversely, we say that two consecutive threads cut corners on the same side if they do not form a corner-switch. Figure~\ref{fig:switchandnonswitch} provides examples of two consecutive threads that either switch corners or cut corners on the same side.

\begin{figure}[!htbp]
	\centering
	\begin{minipage}{0.6\linewidth}
	\includegraphics[width=\linewidth]{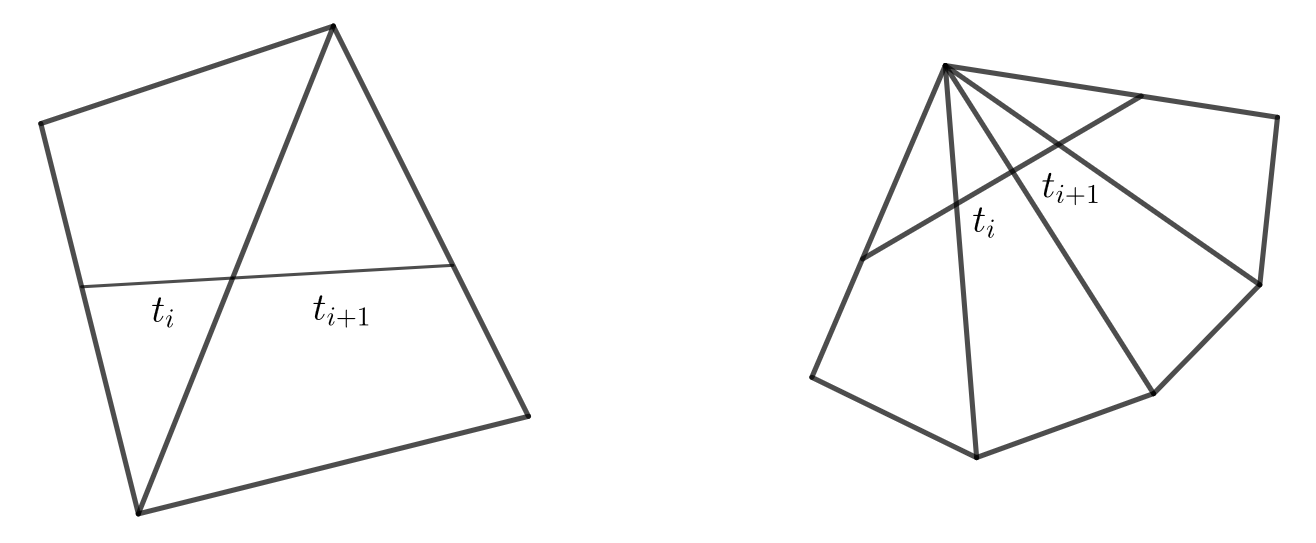}
	\footnotesize
	\end{minipage}
    \caption{Two consecutive threads $t_i$ and $t_{i+1}$ switch corners in the left picture, while they cut corners on the same side in the right picture.}\label{fig:switchandnonswitch}
\end{figure}

\subsection{Widths of triangulations}

We associate to a geometric triangulation $T$ a quantity that measures the ``thickness" of quadrilaterals in the triangulation.

\begin{definition}\label{def:widthoftri}
    Let $X$ be a flat cone surface and $T$ be a geometric triangulation of $X$. Let $e$ be an edge of $T$. Notice that $e$ is adjacent to two triangles (possibly the same triangle). Let $Q(e)$ be the quadrilateral obtained by developing the two adjacent faces in the plane. The \emph{width of the edge} $e$ is defined as the minimal distance between all opposite edges of $Q(e)$, denoted by $d(e)$. The \emph{width of the triangulation} $T$ is defined as the minimum of the widths of all edges of $T$, denoted by $d(T)$.
\end{definition}

The following lemma relates the length of a trajectory to the width of a geometric triangulation.

\begin{lemma}\label{lem:consecutivethick}
    Let $X$ be a unit area flat cone surface with genus $g$ and $n \geq 1$ singularities. Let $k_1, \ldots, k_n$ be the curvatures of all singularities in $X$. Let $T$ be a geometric triangulation of $X$. Let $\gamma$ be a trajectory in $X$ that is not contained in an edge of $T$. Let $t_1, \ldots, t_m$ be the consecutive threads of $\gamma$ that are cut out by $T$. Consider the positive integer
    \begin{equation}\label{equ:m0}
        m_0 = 6(4g + 2n - 4)\Big\lceil\frac{1}{2\min\limits_{1 \leq i \leq n}(1 - k_i)}\Big\rceil.
    \end{equation} 
    If $m \geq m_0 + 2$, then for any $m_0$ consecutive threads $t_{i+1}, \ldots, t_{i+m_0}$ of $\gamma$, with $1 \leq i \leq m - m_0 - 1$, there exists a corner-switch among them. In particular, we have that
    $$|t_{i+1}| + \ldots + |t_{i+m_0}| \geq d(T).$$
\end{lemma}

Note that we exclude the initial thread $t_1$ and the final thread $t_m$, 
as they may end in the interior of faces and thus do not cut any corner.

\begin{proof}
    We prove the lemma by contradiction. Assume that there exist $m_0$ consecutive threads $t_{i+1}, \ldots, t_{i+m_0}$ which contain no corner-switch pairs. Equivalently, $t_{i+1}, \ldots, t_{i+m_0}$ cut corners on the same side of $\gamma$. In particular, these corners are at the same singularity in $X$, denoted by $x$.
    
    Let $k$ be the curvature of $x$. Since the holonomy of a loop around $x$ is a rotation by angle $2\pi(1-k)$, the number of faces that a trajectory can pass through around $x$ is bounded above by
    \begin{equation}\label{equ:estimateforthesameside}
        3(4g + 2n - 4)\Big\lceil\frac{\pi}{2\pi(1 - k)}\Big\rceil,
    \end{equation}
    where $4g + 2n - 4$ is the total number of faces of $T$, and $3(4g + 2n - 4)$ is the total number of corners of $T$.
    
    However, since 
    $$m_0 = 6(4g + 2n - 4)\Big\lceil\frac{1}{2\min\limits_{1 \leq i \leq n}(1 - k_i)}\Big\rceil > 3(4g + 2n - 4)\Big\lceil\frac{\pi}{2\pi(1 - k)}\Big\rceil,$$ 
    the threads $t_{i+1}, \ldots, t_{i+m_0}$ cannot wind around the same singularity $x$.
\end{proof}

\subsection{Lower bounds on widths of triangulations}

To construct a lower bound, we first introduce the notion of systoles.

\begin{definition}
    Let $X$ be a flat cone surface. We define the \emph{systole} of $X$, denoted $\mathrm{sys}(X)$, as the shortest length of a regular closed geodesic in $X$. The \emph{relative systole} of $X$, denoted $\mathrm{relsys}(X)$, is defined as the shortest length of a saddle connection on $X$.
\end{definition}

The systole (resp. relative systole) of a flat cone surface is always positive and is realized by a saddle connection (resp. regular closed geodesic).

Notice that a regular closed geodesic has a cylindrical neighborhood foliated by parallel regular closed geodesics, and there exists a saddle connection at the boundary of the maximal cylindrical neighborhood. Thus, we always have $\mathrm{sys}(X) \geq \mathrm{relsys}(X)$.

\begin{lemma}\label{lem:widthvssys}
    Let $X$ be a unit area flat cone surface with at least one singularity and positive curvature gap $\delta$. Let $T$ be a geometric triangulation of $X$. Then we have that
    $$
    d(T) \geq \frac{\mathrm{relsys}(X)^2}{2R(T)},
    $$
    where $R(T)$ is the maximal radius of the circumscribed disks of the faces of $T$.
\end{lemma}

\begin{proof}
    First, notice that for each edge $e$ of $T$, the distance between the opposite edges of the quadrilateral $Q(e)$ is bounded below by the shortest height of the two triangles (see Definition~\ref{def:widthoftri}). It follows that the width $d(T)$ is bounded below by the shortest height among all triangles of $T$.
    
    Let $h$ be the shortest height among all triangles in $T$, and let $f$ be the triangle realizing this shortest height $h$. By developing $f$ in the plane, denote the vertices of $f$ by $A, B, C$ such that $h$ corresponds to the height from $C$ to the side $AB$.
    
    Denote by $R$ the radius of the circumscribed disk of $f$. Since $\operatorname{Area}(f) = h|AB|/2$, we have that 
    $$
    R = \frac{|AB||BC||CA|}{4\mathrm{Area}(f)} = \frac{|BC||CA|}{2h}.
    $$
    Since the lengths $|BC|$ and $|CA|$ are the lengths of saddle connections on $X$, it follows that
    $$
    d(T) \geq h \geq \frac{\mathrm{relsys}(X)^2}{2R} \geq \frac{\mathrm{relsys}(X)^2}{2R(T)}.
    $$
\end{proof}

\subsection{Combinatorial lengths}

Let $X$ be a flat cone surface with at least one singularity, and let $T$ be a geometric triangulation of $X$. Let $\gamma$ be a trajectory in $X$. Assume that $\gamma$ is not contained in an edge of $T$. Let $t_1, \ldots, t_m$ be the consecutive threads of $\gamma$ that are cut out by $T$. We call $m$ the \emph{combinatorial length of $\gamma$ with respect to $T$}. It serves as an intermediary linking the length and the self-intersection number of $\gamma$.

We first relate the length of a trajectory to its combinatorial length.

\begin{lemma}\label{lem:lenvscombinlen}
    Let $X$ be a unit area flat cone surface with genus $g$ and $n \geq 1$ singularities of curvatures $k_1, \ldots, k_n$. Let $T$ be a geometric triangulation of $X$. Consider the integer $m_0$ as in~\eqref{equ:m0}. Let $\gamma$ be a trajectory in $X$ that is not contained in an edge of $T$, and let $m$ be the combinatorial length of $\gamma$ with respect to $T$. If $m \geq 2m_0$, then
    $$
    |\gamma| \geq d(T) \frac{m}{3m_0}.
    $$
    It follows that for any positive combinatorial length $m \geq 1$, we have
    $$
    |\gamma| \geq d(T)\left(\frac{m}{3m_0} - 1\right).
    $$
\end{lemma}

\begin{proof}
    Let $t_1, \ldots, t_m$ be the consecutive threads of $\gamma$ that are cut out by $T$. 
    To estimate the length, we exclude the initial thread $t_1$ and the final thread $t_m$ as in Lemma~\ref{lem:consecutivethick}, as they may end in the interior of faces and thus do not cut any corner.

    If $m \geq 2m_0 > m_0 + 2$, by Lemma~\ref{lem:consecutivethick}, it follows that    
    \begin{align*}
        |\gamma| &\geq \sum\limits_{i = 1}^{\lfloor (m-2)/m_0 \rfloor} (|t_{(i-1)\cdot m_0 + 2}| + \ldots + |t_{(i-1)\cdot m_0+m_0+1}|)\\
        & \geq \left\lfloor \frac{m-2}{m_0} \right\rfloor d(T)\\
        & \geq \left(\frac{m-2}{m_0} - 1\right)d(T).
    \end{align*}
Note that
$$
\left(\frac{m-2}{m_0} - 1\right) - \frac{m}{3m_0} = \frac{1}{3m_0}(2m-6-3m_0)\ge\frac{1}{3m_0}(4m_0-6-3m_0) = \frac{1}{3m_0}(m_0 -6). 
$$
By~\eqref{equ:m0}, $m_0 \ge 6(4g + 2n -4)\ge 6$. Therefore,
$$
|\gamma|\ge \left(\frac{m-2}{m_0} - 1\right)d(T)\ge \frac{m}{3m_0}d(T)
$$

If $m < 2m_0$, it follows that $\frac{m}{3m_0} - 1 < 0$. Therefore, we obtain the second bound.
\end{proof}

Next, we relate the combinatorial length of a trajectory to its self-intersection number. This is part of the proof of Theorem~1.1 in~\cite{Ara}.

\begin{lemma}\label{lem:combinvsintersect}
    Let $X$ be a flat cone surface with at least one singularity, and let $T$ be a geometric triangulation of $X$. Let $\gamma$ be a trajectory in $X$ that is not contained in an edge of $T$, and denote by $m$ the combinatorial length of $\gamma$ with respect to $T$. Then we have
    $$
    m \geq \sqrt{2\iota(\gamma,\gamma)}.
    $$
\end{lemma}

\begin{proof}
    Let $f$ be a face of the triangulation $T$. Denote the number of threads of $\gamma$ in $f$ by $m_f$, and the number of self-intersections in $f$ (including those at the boundary of $f$) by $I_f$. Notice that 
    $$
    I_f \leq \binom{m_f}{2}.
    $$
    It follows that
    $$
    m_f \geq \sqrt{2I_f}.
    $$
    By summing over all faces of $T$, we obtain that
    $$
    m \geq \sqrt{2\iota(\gamma,\gamma)}.
    $$
\end{proof}

\subsection{Proof of Theorem~\ref{thm:mainindividual}}

\begin{proof}[Proof of Theorem~\ref{thm:mainindividual}]
    If $X$ has no singularities, then it must be a flat torus. Hence, all trajectories have zero intersection number. Therefore, one can take $b_1$ and $b_2$ to be any positive real numbers.
    
    Assume that $X$ has at least one singularity. Let $T$ be a Delaunay triangulation of $X$. If $\gamma$ is contained in a Delaunay edge, we have that $\iota(\gamma,\gamma) = 0$. Hence, the inequality~\eqref{equ:compareindividual} holds for any positive $b_1$ and $b_2$.
    
    If $\gamma$ is not contained in a Delaunay edge, it is cut by the triangles of $T$ into consecutive threads $t_1,\ldots,t_m$. By Lemma~\ref{lem:lenvscombinlen} and Lemma~\ref{lem:combinvsintersect}, it follows that
    \begin{align*}
        |\gamma| \ge d(T)\left(\frac{m}{3m_0} - 1\right) \geq d(T)\left(\frac{\sqrt{2}}{3m_0} \sqrt{\iota(\gamma,\gamma)} - 1\right),
    \end{align*}
    where $m_0$ is given by~\eqref{equ:m0}.
    Note that
\begin{equation}\label{equ:estimateform0}
        m_0 = 6(4g + 2n - 4)\Big\lceil\frac{1}{2\min\limits_{1 \leq i \leq n}(1 - k_i)}\Big\rceil \le 6(4g + 2n - 4)\cdot \frac{2}{2\min\limits_{1 \leq i \leq n}(1 - k_i)} = \frac{6(4g + 2n - 4)}{\min\limits_{1 \leq i \leq n}(1 - k_i)}
\end{equation}
    Combined with the lower bound on $d(T)$ given in Lemma~\ref{lem:widthvssys}, it follows that
    $$
    |\gamma| \ge \frac{\mathrm{relsys}(X)^2}{2R(T)}
    \left(\sqrt{2} \frac{\min\limits_{1\le i\le n}(1-k_i)}{18(4g+2n-4)} \sqrt{\iota(\gamma,\gamma)} - 1\right).
    $$
Consider
    \begin{equation}\label{equ:b1b2}
       \begin{aligned}
            b_1 &= \sqrt{2} \frac{\mathrm{relsys}(X)^2}{R(X)} \frac{\min\limits_{1\le i\le n}(1-k_i)}{36(4g+2n-4)},\\
            b_2 &= \frac{\mathrm{relsys}(X)^2}{2R(X)}
       \end{aligned}
    \end{equation}
    where $R(X)$ is defined to be the maximal $R(T)$ among Delaunay triangulations of $X$.
    It follows that 
    $$
    |\gamma|\geq b_1\sqrt{\iota(\gamma,\gamma)} - b_2.
    $$
    
    Next, consider the case when $\gamma$ is a saddle connection or a regular closed geodesic. We focus on the case when $X$ has at least one singularity. If $\gamma$ is contained in a Delaunay edge of $T$, we have $\iota(\gamma,\gamma) = 0$, so the inequality~\eqref{equ:compareindividual2} holds for any positive $b_1$. 
    
    If $\gamma$ is not contained in a Delaunay edge of $T$, let $t_1, \ldots, t_m$ be the threads of $\gamma$ cut out by $T$. When $m \geq 2m_0$, combining Lemma~\ref{lem:widthvssys}, Lemma~\ref{lem:lenvscombinlen}, and Lemma~\ref{lem:combinvsintersect}, we have
    $$
    |\gamma| \geq \sqrt{2} \frac{\mathrm{relsys}(X)^2}{2R(T)} \frac{1}{3m_0} \sqrt{\iota(\gamma,\gamma)}.
    $$
    When $m < 2m_0$, Lemma~\ref{lem:combinvsintersect} implies that 
    $$\sqrt{\iota(\gamma,\gamma)} \leq \frac{1}{\sqrt{2}}m < \sqrt{2}m_0.$$
    By definition, $|\gamma| \geq \mathrm{relsys}(X)$.  It follows that
    $$
    |\gamma| \geq \mathrm{relsys}(X) \geq \frac{\mathrm{relsys}(X)}{\sqrt{2}m_0} \sqrt{\iota(\gamma,\gamma)}> \frac{\sqrt{2}\mathrm{relsys}(X)}{3m_0} \sqrt{\iota(\gamma,\gamma)}.
    $$
    
    In summary, when $\gamma$ is a saddle connection or a regular closed geodesic, we always have
    $$
    |\gamma| \geq \frac{\sqrt{2} \mathrm{relsys}(X)}{3m_0} \min\left\{\frac{\mathrm{relsys}(X)}{2R(T)}, 1\right\} \sqrt{\iota(\gamma,\gamma)}.
    $$
    Since $2R(T)$ is larger than any edge of $T$ and $\mathrm{relsys}(X)$, we can simplify this to
    $$
    |\gamma| \geq \sqrt{2} \frac{\mathrm{relsys}(X)^2}{2R(T)} \frac{1}{3m_0} \sqrt{\iota(\gamma,\gamma)}
    $$

    Thus, using the bound~\eqref{equ:estimateform0} on $m_0$, we obtain that 
    $$|\gamma|\geq b_1\sqrt{\iota(\gamma,\gamma)},$$
    where $b_1$ is given in~\eqref{equ:b1b2}.
\end{proof}

\begin{remark}\label{rmk:b1b2}
    According to the proof, the constants $b_1$ and $b_2$ in the theorem have the expressions
    \begin{equation}\label{eq:finalb1b2}
       \begin{aligned}
            b_1 &= \sqrt{2}\frac{\mathrm{relsys}(X)^2}{R(X)} \frac{\min\limits_{1\le i\le n}(1-k_i)}{36(4g+2n-4)},\\
            b_2 &= \frac{\mathrm{relsys}(X)^2}{2R(X)},
       \end{aligned}
    \end{equation}
    where $R(X)$ is defined to be the maximal $R(T)$ among Delaunay triangulations of $X$, and $R(X) = 1$ if $X$ has no singularities.
\end{remark}

\begin{prop}\label{prop:estimatedegenerate}
    Let $b_1(X)$ and $b_2(X)$ be the real numbers $b_1$ and $b_2$ in the formula~\eqref{eq:finalb1b2} for a unit area flat cone surface $X$. Consider an infinite sequence of unit area flat cone surfaces $\{X_i\}_{i=1}^{\infty}$ with the same genus $g$, the same number of singularities $n\ge 1$ and the same curvatures $k_1,\ldots,k_n$.
    If $\mathrm{relsys}(X_i)\to 0$ as $i\to\infty$, then we have that $b_1(X_i)$ and $b_2(X_i)$ converge to zero as $i\to\infty$.
\end{prop}

\begin{proof}
    From the expression~\eqref{eq:finalb1b2}, we only need to make sure that $\frac{1}{R(X_i)}$ is bounded as $i\to\infty$. Assume that there is a subsequence of $\{X_i\}_{i=1}^{\infty}$ such that $R(X_i)$ goes to zero. Since $R(X_i)$ is an upper bound for all faces of any Delaunay triangulation of $X$, this implies that the area of $X_i$ goes to zero, which contradicts the assumption that the area of $X_i$ is one.
\end{proof}

We mention that one cannot hope that we can choose the constant $b_2$ in the inequality~\eqref{equ:compareindividual} to be identically zero for any flat cone surface and any trajectory on it. We explain this by an example in the following.

\begin{Ex}\label{ex:b2notzero}
    Let $P$ be a triangle with angles $\frac{\pi}{6},\frac{\pi}{3},\frac{\pi}{2}$ and denote the vertices by $x_1,x_2,x_3$, respectively. Let $X_P$ be the flat cone sphere obtained by gluing the double of $P$ along their boundaries. More precisely, let $P'$ be the reflection of $P$ along the edge $x_1x_3$ and denote the reflection of point $x_2$ by $x_2'$, as in Figure~\ref{fig:example1}. We glue the edges by rotations as in Figure~\ref{fig:example1}. The resulting surface is $X_P$. Then $X_P$ is a flat cone sphere with cone angles $\frac{\pi}{3}, \frac{2\pi}{3}, \pi$, and we denote the singularities corresponding to the cone angles still by $x_1,x_2,x_3$.

     \begin{figure}[!htbp]
    	\centering
    	\begin{minipage}{0.8\linewidth}
        	\includegraphics[width=\linewidth]{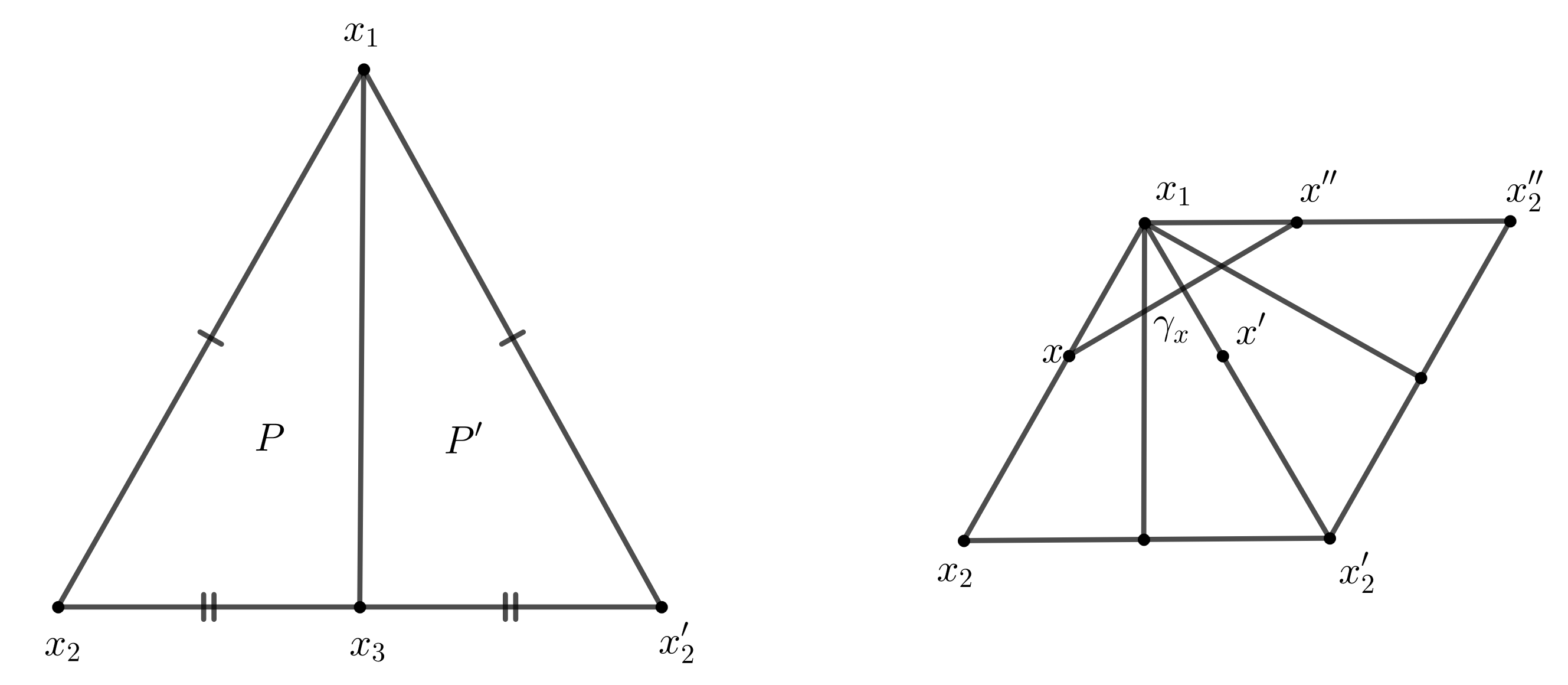}
        	\footnotesize
    	\end{minipage}
        \caption{The surface $X_P$ is obtained by the polygon on the left figure where the edges with the same decoration are glued by rotations. The trajectory $\gamma_x$ in $X_P$ can be unfolded on the plane as on the right figure, where the points $x,x',x''$ are identified as the same point in $X$.}\label{fig:example1}
    \end{figure}

    For any point $x$ in the interior of the edge $x_1x_2$ of $P$, we associate to it a trajectory $\gamma_x$ which starts and ends at $x$ as follows. Since the polygon $P \cup P'$ can be viewed as an unfolding of $X_P$, there is a point in the edge $x_1x_2'$ of $P'$ which corresponds to $x$ under the gluing, and we denote it by $x'$. We rotate the polygon $P \cup P'$ counterclockwise around $x_1$ by angle $\frac{\pi}{3}$ and denote the image of $x'$ by $x''$ after the rotation. We connect $x$ and $x''$ by a straight segment, and it induces a trajectory in $X_P$, denoted by $\gamma_x$. Notice that the self-intersection number of $\gamma_x$ is always one.
    
    \par
    Assume that there exists a positive constant $b_1$ such that 
    \begin{equation}\label{eq:ineqzeroconst}
        b_1\sqrt{\iota(\gamma,\gamma)} \leq |\gamma|
    \end{equation}
    for any trajectory $\gamma$ in $X_P$. In particular, for $\gamma_x$ we have that
    $$
    b_1 \leq |\gamma_x|.
    $$
    However, when the point $x$ in the edge $x_1x_2$ moves towards $x_1$, the length $|\gamma_x|$ goes to zero. Therefore, the inequality~\eqref{eq:ineqzeroconst} does not hold for $X_P$. As a result, the constant $b_2$ in the inequality~\eqref{equ:compareindividual} cannot be identically zero for all flat cone surfaces.
\end{Ex}

\section{Generalized Thurston surgeries}\label{bigsec:hulls}

In this section, we extend Thurston surgeries, introduced in Section~\ref{sec:thurstonsurgery}, by generalizing them to collapse clusters of singularities. To do this, we first define the notion of ``convex hulls" in Section~\ref{sec:convexhull}, followed by a detailed description of the generalized Thurston surgeries in Section~\ref{sec:multi}.

These tools are applied in Section~\ref{sec:trajectory} to prove Theorem~\ref{thm:combinandmetriclengthinhull}, which shows that the lengths of trajectories contained in convex hulls are bounded above. This theorem is a key component in the proof of Theorem~\ref{thm:mainfavorite}.

\par
Before proceeding, we clarify the concept of homotopies for later use. Let $X$ be a flat sphere with singularities at $x_1, \ldots, x_n$. The homotopy we consider is relative to the subset $\{x_1, \ldots, x_n\}$ of singularities on the flat cone sphere $X$. For an arc (or loop) $a$ on $X$, we denote its homotopy class by $\alpha$. We expand the class $\alpha$ to include curves on $X$ that are limits of sequences of curves within the original class. This enlargement allows the curves in $\alpha$ to pass through singularities, which enables us to discuss the shortest representative in the homotopy class $\alpha$.

\subsection{Core of infinite non-negative flat spheres}\label{sec:infiniteflatspheres}
The notion of the core of a meromorphic translation surface was introduced in~\cite{HKK}. In this subsection, we generalize this concept to infinite non-negative flat spheres.

Let $X$ be an infinite non-negative flat sphere. Throughout this paper, we assume that $X$ has exactly one \emph{pole}, that is, a singularity with curvature $k>1$.

In fact, there can be at most one pole with curvature greater than~1 on~$X$. Indeed, if there were two such poles, their total curvature would exceed~2, which would force the presence of negative curvatures to satisfy the Gauss--Bonnet formula. The only situation where $X$ may have two poles occurs when $X$ is an infinite cylinder, whose two ends correspond to poles of curvature~$-1$ and whose other conical singularities have zero curvature. Since this case does not arise in the present work, we exclude it from consideration.

\begin{lemma}\label{lem:uniqueshortest}
    Let $X$ be an infinite non-negative flat sphere. Let $\alpha$ denote the homotopy class of a simple loop around the pole. Then there exists a unique shortest representative $\gamma$ of~$\alpha$, which passes through some conical singularities of~$X$.
\end{lemma}

\begin{proof}
Assume that there exist two distinct shortest representatives $\gamma_1$ and $\gamma_2$ in the class $\alpha$, both passing through conical singularities of $X$.

We claim that these two curves must intersect. Otherwise, since they belong to the same homotopy class~$\alpha$, they would bound an annulus on~$X$. As each $\gamma_i$ also bounds a neighborhood of the pole, we may assume that one boundary of the annulus, say $\gamma_1$, lies inside the neighborhood of the pole enclosed by~$\gamma_2$. However, since there are no singularities in the neighborhood of the pole, $\gamma_1$ cannot pass through any conical singularity, a contradiction.

Then, by the bigon criterion (see Section~1.2.4 in \cite{FM}), $\gamma_1$ and $\gamma_2$ form a bigon $R$, meaning the boundary of $R$ can be decomposed into two sub-paths, $e_1$ and $e_2$, where $e_i$ is contained in $\gamma_i$, respectively. The region $R$, with the induced metric from $X$, is a flat domain.  Without loss of generality, assume that the interior of $e_1$ is on the side of $\gamma_2$ containing the pole. This implies that there are no singularities in the interior of $e_1$, although $e_2$ may contain singularities. Since $\gamma_2$ is the shortest representative of $\alpha$, the angle at any singularity in the interior of $e_2$ must be at least $\pi$. However, by the Gauss-Bonnet formula, such a flat domain $R$ cannot exist, leading to a contradiction.
\end{proof}

\begin{definition}\label{def:core}
    Let $X$ be an infinite non-negative flat sphere, and let $\alpha$ be the homotopy class of a simple loop around the pole. Let $\gamma$ be the unique shortest representative of $\alpha$. The \emph{core} of $X$ is defined as the closed subset enclosed by $\gamma$, which has finite area. 
    We denote this core by $core(X)$. The core is said to be \emph{degenerated} if $core(X) = \partial core(X)$.
\end{definition}

\begin{remark}\label{rmk:degeneratedcore}
    Since $X$ is non-negative, the core is degenerated if and only if $\partial core(X)$ is contained in a single geodesic that joins two distinct conical singularities with curvatures of at least $\frac{1}{2}$, and may possibly contain singularities of zero curvature in the interior.
\end{remark}

\begin{lemma}\label{lem:uniquecharacterization}
    Let $X$ be an infinite non-negative flat sphere, and let $\alpha$ denote the homotopy class of a simple loop around the pole. Then a representative $\gamma$ of~$\alpha$ that passes through conical singularities of~$X$ is the shortest representative if and only if
    \begin{itemize}
        \item $\gamma$ consists of saddle connections, and
        \item The curve $\gamma$ divides $X$ into at most two connected components. The metric completion of the component containing the pole is a flat domain whose boundary corners have angles at least~$\pi$.
    \end{itemize}
    Moreover, any saddle connection on $X$ is contained within $core(X)$.
\end{lemma}

\begin{proof}
If $\gamma$ is the shortest representative of $\alpha$, we need only verify the second condition. 
Let $R_1$ be the component of $X \setminus \gamma$ containing the pole. The flat metric induces a length metric on $R_1$.
Let $X_1$ be the metric completion of the component containing the pole, with the induced length metric.
There is a canonical isometric immersion from $X_1$ into $X$, and the boundary $\partial X_1$ maps onto $\gamma$.
Since $\gamma$ is the shortest representative of $\alpha$, the boundary corners of $X_1$ must have angles of at least $\pi$.
\par
To show that $X \setminus \gamma$ has at most two connected components, notice that if $core(X)$ is degenerated, then $X \setminus \gamma$ has only one component. If $core(X)$ is non-degenerated, assume $X \setminus \gamma$ has more than two components. Then $\gamma$ has a self-intersection point, denoted by $y$. This means two points on the boundary of $X_1$ map to $y$, forcing the total angle at $y$ to exceed $2\pi$. This contradicts the assumption that $X$ is non-negative.
\par
Conversely, let $\gamma$ be a representative of $\alpha$ that satisfies the two conditions. In fact, the proof of Lemma~\ref{lem:uniqueshortest} also implies that the representative satisfying these conditions is unique. Thus, $\gamma$ must be the shortest representative of $\alpha$.
\par
For the final statement, assume there is a saddle connection $\gamma'$ not contained in $core(X)$. Then $\gamma'$ would intersect the boundary loop $\gamma$ of $core(X)$. This would create a bigon $R$, and the Gauss-Bonnet formula rules out the existence of such a bigon.
\end{proof}

\subsection{Convex hulls}\label{sec:convexhull}
We generalize the definition of cores for a cluster of conical singularities on a flat sphere.
\par
Let $X$ be a non-negative flat sphere and let $x_{i_1},\ldots,x_{i_m}$ be $m$ conical singularities of $X$. Let $\alpha$ be the homotopy class of a simple loop enclosing the singularities $x_{i_1},\ldots,x_{i_m}$ on one side only. In general, the shortest representatives of $\alpha$ are not unique. However, if we restrict our attention to shortest representatives passing through singularities within the set $\{x_{i_1},\ldots,x_{i_m}\}$, then such a representative, if it exists, is unique. This is formalized in the following lemma.

\begin{lemma}\label{lem:characterizeconvexhulls}
    Let $X$ be a non-negative flat sphere. Let $\alpha$ be the homotopy class of a simple loop, and let $x_{i_1},\ldots,x_{i_m}$ be the conical singularities in one connected component of the complement of the loop. If there exists a shortest representative of $\alpha$ passing through singularities belonging to $\{x_{i_1},\ldots,x_{i_m}\}$, then such a shortest representative is unique.
\end{lemma}

The proof of this lemma is essentially the same as that of Lemma~\ref{lem:uniqueshortest}, so we omit it here.

\begin{definition}\label{def:convexhull}
    Let $X$ be a non-negative flat sphere. Let $\alpha$ be the homotopy class of a simple loop, and let $x_{i_1},\ldots,x_{i_m}$ be the conical singularities in one connected component of the complement of the loop. Assume that the shortest representative $\gamma$ of $\alpha$ passing through singularities belonging to $\{x_{i_1},\ldots,x_{i_m}\}$ exists. Then we define the \emph{convex hull of $x_{i_1},\ldots,x_{i_m}$ along $\alpha$} to be the closed subset enclosed by $\gamma$, which contains the singularities $x_{i_1},\ldots,x_{i_m}$. We denote the convex hull $D_{\alpha}$. The convex hull $D_{\alpha}$ is said to be \emph{degenerate} if $\partial D_{\alpha} = D_{\alpha}$.
\end{definition}

\begin{remark}\label{rmk:degeneratedconvexhull}
    Similar to Remark~\ref{rmk:degeneratedcore}, a convex hull $D_{\alpha}$ on a flat sphere is degenerated if and only if $D_{\alpha}$ is contained in a single geodesic that joins two distinct conical singularities with curvatures of at least $\frac{1}{2}$, and may possibly contain singularities of zero curvature in the interior.
\end{remark}

\begin{figure}[!htbp]
	\centering
	\begin{minipage}{0.8\linewidth}
		\includegraphics[width=\linewidth]{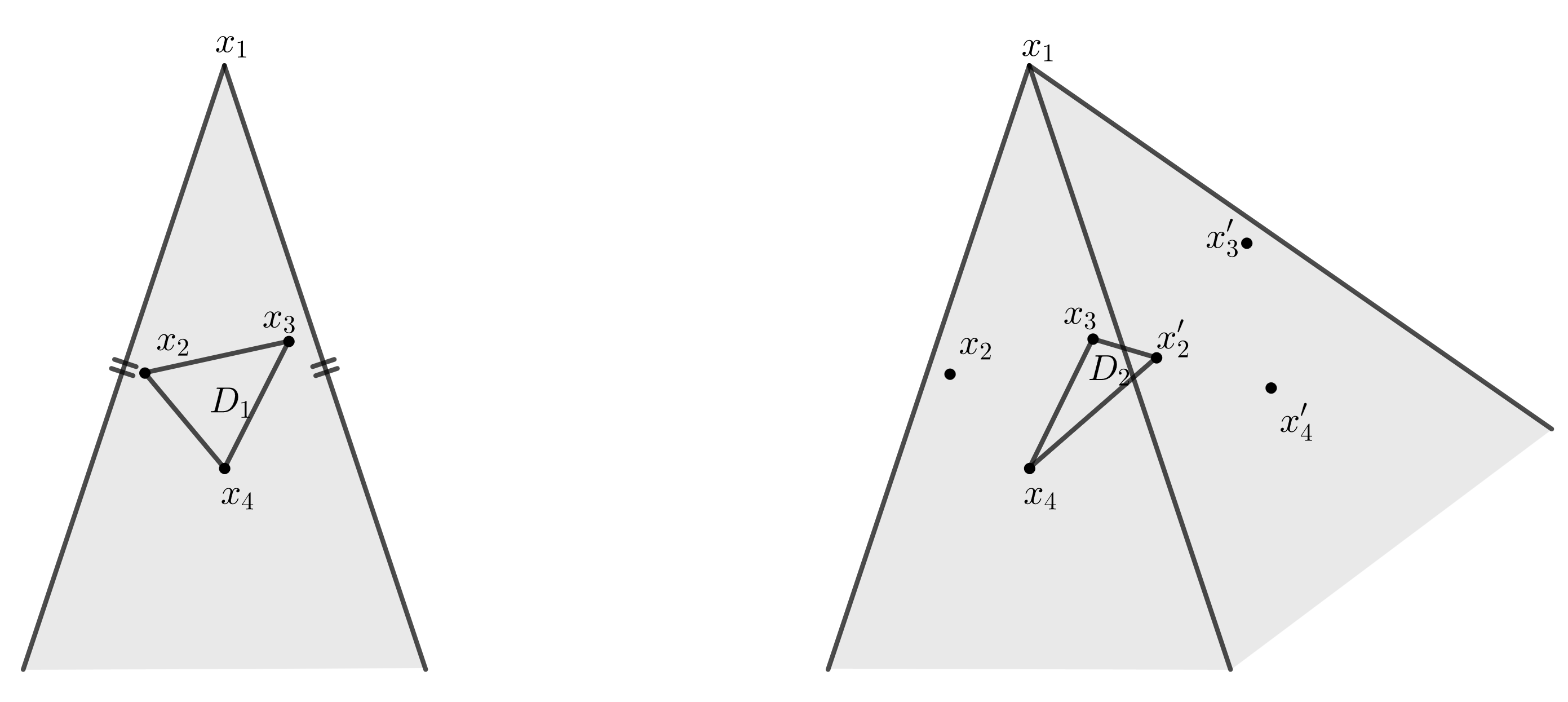}
		\footnotesize
	\end{minipage}
    \caption{The grey cone on the left is an unfolding of a neighborhood of a singularity $x_1$ on some flat cone sphere. The points $x_2,x_3,x_4$ are singularities of zero curvature. By developing the grey cone, we obtain the picture on the right. Notice that $D_1$ and $D_2$ induce two convex hulls of $x_2,x_3,x_4$ in the flat sphere but along different homotopy classes of loops.}
\end{figure}

As in Lemma~\ref{lem:uniquecharacterization}, the boundary $\partial D_{\alpha}$ can be characterized as follows. The proof is similar to that of Lemma~\ref{lem:uniquecharacterization}, so we omit it here.

\begin{lemma}
    Let $X$ be a non-negative flat sphere and let $x_{i_1},\ldots,x_{i_m}$ be $m$ conical singularities of $X$. Let $\alpha$ be the homotopy class of a simple loop enclosing the singularities $x_{i_1},\ldots,x_{i_m}$ on one side only. Assume that the convex hull $D_{\alpha}$ of $x_{i_1},\ldots,x_{i_m}$ along $\alpha$ exists. Then a representative $\gamma$ of $\alpha$ is the boundary loop $\partial D_{\alpha}$ if and only if $\gamma$ satisfies:
    \begin{itemize}
        \item $\gamma$ consists of saddle connections with endpoints in $x_{i_1},\ldots,x_{i_m}$, and
        \item The curve $\gamma$ divides $X$ into at most two connected components. The metric completion of the component containing the singularities not in $\{x_{i_1},\ldots,x_{i_m}\}$ is a flat domain whose boundary corners have angles at least~$\pi$.
    \end{itemize}
    Moreover, any saddle connection in $X$ whose endpoints lie in $\{x_{i_1},\ldots,x_{i_m}\}$ and whose homotopy class has zero geometric intersection number with~$\alpha$ is contained in~$D_{\alpha}$.
\end{lemma}

\begin{remark}\label{rmk:lessthan1}
    From the second condition above, the Gauss-Bonnet formula implies that if the convex hull of $x_{i_1},\ldots,x_{i_m}$ exists, then the sum of their curvatures is less than or equal to one.
\end{remark}

\subsection{Generalized Thurston surgeries}\label{sec:multi}
Let $X$ be a non-negative flat sphere. 
Assume that a convex hull, denoted by $D$, exists in $X$; see Figure~\ref{fig:1}. The flat metric restricted to $X\setminus D$ induces a length metric. We denote by $X_D$ the metric completion of $X\setminus D$ with this induced length metric. Note that there is a canonical immersion from $X_D$ into $X$.

\begin{lemma}\label{lem:infinitecone}
    Let $X$ be a non-negative flat sphere and let $D$ be a convex hull in $X$. Let $C$ be the infinite cone whose curvature at the apex is equal to the sum of the curvatures of the singularities in $D$. Then there exists a flat domain $C_D$ contained in $C$, unique up to isometry of $C$, such that:
    \begin{itemize}
        \item the boundary of $C_D$ is isometric to the boundary of $X_D$,
        \item for any point $y$ in $\partial C_D$ corresponding to a singularity $x$ in $\partial X_D$, the angle of the corner at $y$ outside $C_D$ is equal to the angle of the corner of $X_D$ at $x$.
    \end{itemize}
\end{lemma}

\begin{proof}
    We first prove the existence of $C_D$. If $D$ is degenerated, Remark~\ref{rmk:degeneratedconvexhull} implies that $D$ is contained in a single geodesic $\gamma$ which joins two distinct conical singularities and may contain singularities of zero curvature in its interior. We can disregard the singularities of zero curvature in the interior and treat $\gamma$ as a saddle connection. Then $C_D$ can be defined as the added cone $C_{\gamma}$ resulting from Thurston surgery along $\gamma$.
    \par
    Now assume that $D$ is non-degenerated. Let $\gamma$ be a saddle connection contained in $D$. By performing Thurston surgery along $\gamma$, we obtain a new flat sphere $X_1$. According to Lemma~\ref{lem:remains}, $X_{\gamma}$ is isometrically embedded into $X_1$. We view $X_D \subset X_{\gamma}$ as also being contained in $X_1$. The portion of $X_1$ that lies outside $X\setminus D$ is denoted $D_1$, with one fewer singularity than $D$. By repeatedly applying this procedure, we construct a sequence of flat spheres $X_1, X_2, \dots, X_m$ and a sequence of flat domains $D_1, D_2, \dots, D_m$, where $X_D$ is isometrically embedded onto the complement of the interior of $D_i$ in $X_i$, and only one singularity remains in $D_m$. The curvature at this singularity equals the total curvature of the convex hull $D$, denoted by $k$.
    \par
    Now we show that $D_m$ is the desired flat domain $C_m$.
    Since $X_{D}$ is isometrically embedded onto the complement of the interior of $D_m$ in $X_m$, we have that the boundary of $D_m$ is isometric to the boundary of $X_D$.
    
    Moreover, since the angle at each boundary corner of $X_D$ is at least $\pi$, the angles at the corners of $D_m$ are at most $\pi$. By cutting along a trajectory from the singularity in $D_m$ to the boundary, we unfold $D_m$ into a bounded sector with an angle of $2\pi(1-k)$. The position of the center of this sector is determined by the fixed point of the holonomy along the loop $\partial D_m$. By completing the bounded sector into an infinite sector and gluing the radial rays, we obtain the infinite cone $C$. Hence, $D_m$ becomes a flat domain contained in $C$. Furthermore, the way to construct $C$ from $D_m$ ensures that $D_m$ is unique up to rotation around the apex of $C$.
\end{proof}

Let $D_1, \ldots, D_s$ be disjoint convex hulls on a non-negative flat surface $X$. Let $C_{D_i}$ be the flat domain in $C$ associated with $D_i$, according to Lemma~\ref{lem:infinitecone}. We define the \emph{generalized Thurston surgery} of $X$ along $D_1, \ldots, D_s$ as follows:
\begin{itemize}
    \item Cut out each $D_i$ from $X$ and glue $C_{D_i}$ to the surface along $\partial D_i$ for each $i$. Denote the resulting flat sphere by $X^{(0)}$.
    \item For each $D_i$, glue the metric completion of $C\setminus C_{D_i}$ to $D_i$ along $\partial D_i$. Denote the resulting infinite flat sphere by $X_i$.
\end{itemize}
We call $X^{(0)}$ the \emph{top} flat sphere and $X_1, \ldots, X_s$ the \emph{infinitesimal} flat spheres obtained by the generalized Thurston surgery. We call $C_{D_1}, \ldots, C_{D_s}$ the \emph{added cones} during the surgery.

\begin{figure}[!htbp]
	\centering
	\begin{minipage}{0.8\linewidth}
	\includegraphics[width=\linewidth]{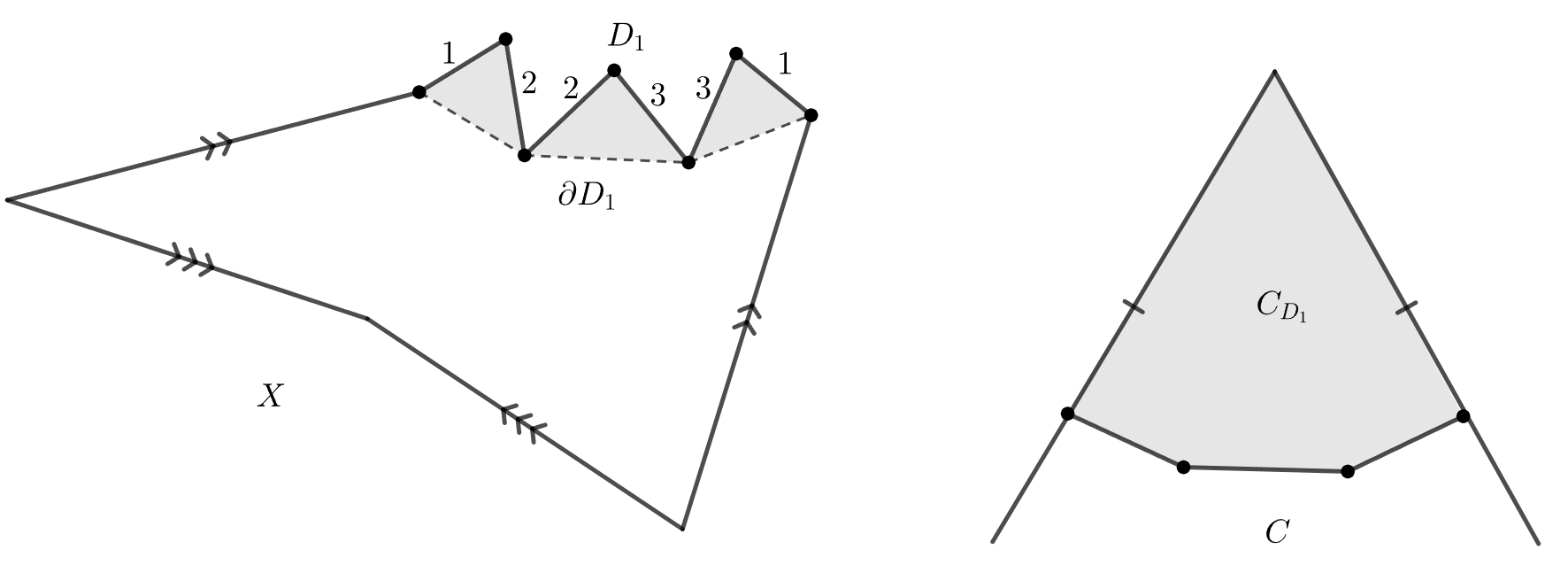}
	\footnotesize
	\end{minipage}
    \caption{The gray regions on the left form a convex hull $D_1$ in $X$; the generalized Thurston surgery cut out $D_1$ and glue $C_{D_1}$ (on the right) to $X\setminus D_1$ along the dotted line $\partial D_1$.}\label{fig:1}
\end{figure}

\begin{remark}
    In this paper, we refer to both the Thurston surgeries described in Section~\ref{sec:thurstonsurgery} and the generalized Thurston surgeries as \emph{Thurston surgeries}. Whenever a surgery is applied, we will specify whether it is performed along a saddle connection or along convex hulls.
\end{remark}

The following lemma is analogous to Lemma~\ref{lem:remains}. The proof is similar and is omitted here.

\begin{lemma}\label{lem:generalizedremains}
    Let $X$ be a non-negative flat sphere and let $D_1, \ldots, D_s$ be disjoint convex hulls in $X$. Then the metric completion of $X\setminus\bigcup_{1 \le i \le s} D_i$, together with the induced length metric, is isometrically embedded into $X^{(0)}$ in the sense that the embedding induces an isometry on the tangent spaces.
\end{lemma}

The following lemma will be frequently used in Section~\ref{bigsec:coor}.

\begin{lemma}\label{lem:areaestimates}
    Let $X$ be a convex flat cone sphere, and let $D_1, \ldots, D_s$ be disjoint convex hulls in $X$. Let $X^{(0)}$ be the top flat sphere obtained from $X$ by applying Thurston surgery along the convex hulls. Then we have
    $$
    \mathrm{Area}(D_i) \leq \mathrm{Area}(C_{D_i})\mbox{ and }\mathrm{Area}(X) \leq \mathrm{Area}(X^{(0)}).
    $$
\end{lemma}

\begin{proof}
    Recall that in the proof of Lemma~\ref{lem:infinitecone}, the flat domain $C_{D_i}$ is obtained by applying Thurston surgeries along saddle connections. Since Thurston surgeries add cones during the operation, the area of the sphere increases. Therefore, we have $\mathrm{Area}(D_i) \leq \mathrm{Area}(C_{D_i})$. Furthermore, we know that
    \begin{equation}\label{equ:area}
        \mathrm{Area}(X) = \mathrm{Area}(X^{(0)}) - \sum_{i=1}^s \Big(\mathrm{Area}(C_{D_i}) - \mathrm{Area}(D_i)\Big).
    \end{equation}
    It follows that $\mathrm{Area}(X) \leq \mathrm{Area}(X^{(0)})$.
\end{proof}

\subsection{A sufficient condition for the existence of convex hulls}
We have not yet discussed when a convex hull exists for a cluster of singularities. To address this, we first introduce some notation. Let $X$ be a flat sphere, and let $F$ be a geometric forest in $X$. The metric \emph{length} of $F$ is defined as the sum of the lengths of the edges of $F$, denoted by $|F|$.
\par
The following lemma provides a sufficient condition for the existence of a convex hull.

\begin{lemma}\label{lem:embeddedregion}
    Let $X$ be a non-negative flat sphere and $F$ a geometric tree in $X$ such that the sum of the curvatures at the vertices of $F$ is less than $1$. Denote by $\alpha$ the homotopy class of a loop encircling $F$. Assume that for any simple saddle connection $\gamma$ whose homotopy class has non-zero geometric intersection number with $\alpha$, one has:
    \begin{equation}\label{equ:embeddedregionassumption}
        |\gamma| \ge 2|F|.
    \end{equation}
    Then the convex hull $D$ of the vertices of $F$ along the class $\alpha$ exists in $X$.
\end{lemma}

\begin{proof}
    We prove the statement by contraposition. Assume that the convex hull $D$ of the vertices of $F$ in the class $\alpha$ does not exist in $X$. By Definition~\ref{def:convexhull}, this means that for every shortest representative $l$ of $\alpha$, the loop $l$ passes through a singularity of $X$ that does not belong to $F$. There are two cases.

    \emph{Case 1.} Assume that $l$ does not meet any vertex of $F$. Let $R$ be the flat domain bounded by $l$ that contains $F$. By the Gauss-Bonnet formula on the domain $R$,
    $$
    \sum_{x_i\in \operatorname{Int}(R)}2\pi k_i + \sum_{x_j\in \partial R}(\pi-\theta_j) = 2\pi,
    $$
    where $\operatorname{Int}(R)$ denotes the interior of $R$ and $\theta_j$ denotes the interior angle of $R$ at $x_j\in \partial R$. 
    
    We claim that every boundary angle satisfies that $\theta_j\ge \pi$. Indeed, since we assume that $l$ does not pass through any vertex of $F$, the singularity $x_j\in \partial R = l$ is not a vertex of $F$. If $\theta_j<\pi$, one can shorten $l$ by pushing it slightly into the corner of $R$ at $x_j$; this produces a strictly shorter loop in the class $\alpha$, contradicting that $l$ is a shortest representative.

    Hence $\theta_j\ge \pi$, so each $\pi - \theta_j\le 0$ and
    $$
    \sum_{x_i\in \operatorname{Int}(R)}2\pi k_i\ge \sum_{x_i\in \operatorname{Int}(R)}2\pi k_i + \sum_{x_j\in \partial R}(\pi-\theta_j) = 2\pi.
    $$
    Because $\alpha$ encircles $F$, the interior singularities of $R$ are precisely the vertices of $F$. Therefore, the above inequality implies that the sum of the curvatures at the vertices of $F$ is at least $1$, which is the negation of our assumption on $F$ that it is less than $1$.

    \emph{Case 2.} Assume that $l$ meets both singularities that are vertices of $F$ and singularities that are not.

    In particular, in this case, there is a saddle connection $\gamma$ contained in $l$ whose homotopy class has non-zero intersection number with $\alpha$.

    Consider the loop $l'$ in the class $\alpha$ obtained by following each edge of $F$ once forwards and once backwards. By construction, $|l'| = 2|F|$. If $l'$ were a shortest representative of $\alpha$, then, by Definition~\ref{def:convexhull}, the convex hull $D$ along $F$ would exist, which contradicts our initial assumption in this proof. Hence $l'$ is not a shortest representative of $\alpha$, and therefore
    $$|l| < |l'| = 2|F|.$$
    In particular, for the saddle connection $\gamma$ lying in $l$ as above, we then have 
    $$
    |\gamma|\le |l| < |l'| = 2|F|.
    $$
    This negates the assumption~\eqref{equ:embeddedregionassumption}.
\end{proof}

\subsection{Trajectories in convex hulls}\label{sec:trajectory}

In this subsection, we apply the tools introduced earlier to prove Theorem~\ref{thm:combinandmetriclengthinhull}, which will play a crucial role in the proof of the main result, Theorem~\ref{thm:mainfavorite}.

Let $X$ be a flat sphere. Recall that a trajectory in $X$ is a geodesic of finite length without singularities in its interior.

\begin{lemma}\label{lem:estdiscone}
    Let $C$ be a flat domain with only one singularity in the interior. Assume that the curvature $k$ of the singularity is non-negative.
    Let $\gamma$ be any trajectory from the singularity in the interior to the boundary $\partial C$.  
    Then the length of $\gamma$ is bounded above by 
    $$
    \frac{1}{2}\max\Big\{1, \frac{1}{2(1-k)}\Big\}|\partial C|.
    $$
\end{lemma}

\begin{proof}
    We cut $C$ open along $\gamma$ and let $D$ be the metric completion of $C\setminus \gamma$. Then $D$ is a flat domain without any singularity in the interior. 
    Notice that $\partial D$ contains two copies of $\gamma$, denoted by $\gamma_1$ and $\gamma_2$, and they form a corner at the boundary of angle $\theta = 2\pi(1-k)$. We denote by $\gamma_3\subset \partial D$ the path consisting of $\gamma_1$ and $\gamma_2$.
    \par
    When $\theta \ge\pi$, the concatenation $\gamma_3$ is the shortest representative of its homotopy class in $D$. Indeed, if there existed a shorter representative $\gamma'$, then $\gamma_3$ and $\gamma'$ would form a bigon. Since $\gamma'$ is a shortest path, the corners of the bigon in the interior of $\gamma'$ have angles at least $\pi$,  but the Gauss-Bonnet formula implies that such a bigon cannot exist.
    \par
    Since the remaining boundary $\partial D\setminus \gamma_3$ is in the same homotopy class as $\gamma_3$ and has length $|\partial C|$, we have
    \begin{equation}\label{equ:usingthis1stcase}
        2|\gamma| = |\gamma_3| \leq |\partial C|.
    \end{equation}
    
    When $\theta < \pi$, consider the developing map of $D$ into the plane and let $x_1$ and $x_2$ be the endpoints of image of $\gamma_3$ under the developing map. The Euclidean distance between $x_1$ and $x_2$ in the plane is $2|\gamma|\sin\frac{\theta}{2}$. Since $\partial D\setminus\gamma_3$ is a piecewise straight path joining $x_1$ and $x_2$ of length $|\partial C|$, we get 
    \begin{equation}\label{equ:usingthis}
        2|\gamma|\sin\frac{\theta}{2}\leq |\partial C|.
    \end{equation}
    
    Note that $\sin t\ge \frac{2}{\pi} t$ for $t\in(0,\frac{\pi}{2})$. Since $\theta < \pi$, we have
    $$\sin\frac{\theta}{2}\ge \frac{2}{\pi}\cdot\frac{\theta}{2} = \frac{\theta}{\pi}.$$
    Since $\theta = 2\pi(1-k)$, we obtain 
    $$\sin\frac{\theta}{2}\ge \frac{\theta}{\pi} = \frac{2\pi(1-k)}{\pi} = 2(1-k).$$
    Applying the above inequality to~\eqref{equ:usingthis}, we deduce
    $$
    |\gamma| \leq \frac{|\partial C|}{2\sin\frac{\theta}{2}} \leq \frac{|\partial C|}{2\cdot 2(1-k)} = \frac{1}{4(1-k)}|\partial C|.
    $$
    Therefore, combining the above with~\eqref{equ:usingthis1stcase}, we obtain that
    $$
    |\gamma|\le \max\left\{\frac{1}{2}|\partial C|, \frac{1}{4(1-k)}|\partial C|\right\} = \frac{1}{2}\max\Big\{1, \frac{1}{2(1-k)}\Big\}|\partial C|.
    $$
\end{proof}

Let $\underline{k} = (k_1,\ldots,k_n)$ be a curvature vector of a flat sphere. Recall that the curvature gap is defined as
$$
\delta(\underline{k}) = \inf\limits_{I \subset \{1,\ldots,n\}} \Big|1-\sum\limits_{\substack{i \in I\\k_i<1}} k_{i}\Big|.
$$

\begin{lemma}\label{lem:lengthbetweentwoconvexhulls}
    Let $X$ be a non-negative flat sphere with positive curvature gap $\delta$. Assume that $D_0$ is a flat domain in $X$ and $D_1$ is a convex hull of the singularities in the interior of $D_0$. Then, for any trajectory $\gamma$ contained in $D_0 \setminus D_1$, we have
    $$
    |\gamma| \leq 2\max\Big\{1, \frac{1}{2\delta}\Big\}|\partial D_0|.
    $$
\end{lemma}

\begin{proof}
    Applying Thurston surgery along $D_1$, denote the resulting top surface by $X^{(0)}$. Recall that this surgery replaces $D_1$ with a bounded cone $C_{D_1}$, where the curvature $k$ of the apex $x_0$ is the sum of the curvatures of the singularities in $D_1$.
    By Lemma~\ref{lem:generalizedremains}, we regard $X \setminus D_1$ as embedded in $X^{(0)}$. Let $R$ be the flat domain in $X^{(0)}$ enclosed by $\partial D_0$ and containing $C_{D_1}$, and let $\gamma$ be a trajectory contained in $C_{D_1}$.
    \par
    To estimate the length of $\gamma$, consider the universal cover of $R \setminus \{x_0\}$ and denote the metric completion of the universal covering by $\widetilde R$. 
    We still denote by $x_0$ the point in $\widetilde R$ mapping to $x_0$. Let $\gamma_1$ and $\gamma_2$ be the shortest path from $x_0$ to the two endpoints of $\gamma$ in $\widetilde R$. 
    Notice that the holonomy vector of $\gamma$ can be expressed as the difference of the holonomy vectors of $\gamma_1$ and $\gamma_2$. 
    By the triangle inequality, we know that $|\gamma|\le |\gamma_1| + |\gamma_2|$.
    
    \par
    Let $r$ be the maximum distance from $x_0$ to $\partial D_0$. Then, $|\gamma_i| \leq r$ for each $i$. Notice that  $r$ is realized by a trajectory from $x_0$ to a singularity in $\partial D_0$ possibly with a sub-path contained in $\partial D_0$.
    Using Lemma~\ref{lem:estdiscone}, we get
    $$
    r \leq \frac{1}{2}\max\Big\{1, \frac{1}{2(1-k)}\Big\}|\partial D_1| + \frac{1}{2}|\partial D_1| \leq \max\Big\{1, \frac{1}{2\delta}\Big\}|\partial D_0|.
    $$
    Therefore, $|\gamma| \leq 2r \leq 2\max\Big\{1, \frac{1}{2\delta}\Big\}|\partial D_0|$.
\end{proof}

Let $\gamma$ be a trajectory contained in a convex hull $D$ of $X$, and let $T$ be a geometric triangulation of $D$. The \emph{combinatorial length} of $\gamma$ with respect to $T$ is the number of threads of $\gamma$ cut by $T$.

\begin{thm}\label{thm:combinandmetriclengthinhull}
    Let $X$ be a non-negative flat sphere with positive curvature gap $\delta$. Let $D$ be a convex hull of $m$ conical singularities on $X$, and $T$ be a geometric triangulation of $D$.
    Then, for any trajectory $\gamma$ contained in $D$, its combinatorial length with respect to $T$ is bounded above by
    $$
    \frac{8m^2}{\delta},
    $$
    and its metric length is bounded above by 
    $$
    4m\max\Big\{1, \frac{1}{2\delta}\Big\} |\partial D|.
    $$ 
    Moreover, when $X$ is a flat cone sphere, the metric length of $\gamma$ is bounded above by $\frac{2m}{\delta}|\partial D_0|$.
\end{thm}

\begin{remark}
    The curvature gap cannot be eliminated from the upper bounds provided. To illustrate this, consider the following example. Let $C$ be an infinite cone with a singularity at the apex and a singularity of zero curvature, and let the cone angle of $C$ be denoted by $A$. Assume that $A$ is less than $\pi$, ensuring that the core of $C$ is not degenerate. We triangulate the core by adding an arc between the apex and the other singularity. The longest saddle connection on $C$ starts at the singularity of zero curvature, goes around the apex and comes back to the starting point. This saddle connection has a combinatorial length of $2\left\lfloor \frac{\pi}{A} \right\rfloor$ with respect to the triangulation of the core. As $A$ approaches zero, both the combinatorial length and the metric length of the saddle connection grow without bound.
\end{remark}

\begin{proof}[Proof of Theorem~\ref{thm:combinandmetriclengthinhull}]
    If $D$ is degenerate, the bounds hold trivially. 
    If $D$ contains no singularities in its interior, then $D$ is a convex polygon. Since $D$ has $m$ singularities, the geometric triangulation $T$ has $m-2$ faces. Therefore, the combinatorial length of any trajectory contained in $D$ is at most $m-2 \leq \frac{8m^2}{\delta}$. The metric length of any trajectory in $D$ is less than the perimeter, $|\partial D| < 4m\max\left\{1, \frac{1}{2\delta}\right\} |\partial D|$.

    Next, we assume that $D$ has at least one singularity in its interior. Let $\alpha_1$ be the homotopy class of a loop around the boundary $\partial D$ in $D$. Since the angle at each corner of $D$ is at most $\pi$, the shortest representative of $\alpha_1$ cannot pass through the singularities on $\partial D$. Hence, if there is more than one singularity in the interior of $D$, we conclude that the convex hull of these singularities along $\alpha_1$ exists, denoted by $D_1$. If there is only one singularity, $D_1$ refers to that singularity.

    We consider the following sequence of convex hulls in $D$:
    $$
    D_0 = D \supset D_1 \supset \ldots \supset D_r
    $$
    where $D_i$ is the convex hull of the singularities in the interior of $D_{i-1}$ along the homotopy class of $\partial D_{i-1}$, and $D_r$ is either a convex polygon or a single singularity.
    Since $D$ contains $m$ singularities, it follows that $r\le m$.

    Let $f$ be any face of $T$ and $e$ be any edge of $T$. We claim that 
    \begin{itemize}
        \item $f\cap D_i$ (resp. $e\cap D_i$) has at most one connected component,
        \item $f\cap (D_i\setminus D_{i+1})$ (resp. $e\cap (D_i\setminus D_{i+1})$) has at most two connected components.
    \end{itemize}
    Indeed, since $D_1$ is convex hull, any trajectory in $D$ that exits $T$ will not return to $D_1$. It follows that every edge of $T$ intersects with $D_1$ into at most one connected component, and intersects with $D_0\setminus D_1$ into at most two connected components. Hence, this implies that the face $f$ also satisfies the claim for $D_1$. Similarly, one can show that the claim holds for any $D_i$.

    Let $\gamma$ be a trajectory in $D$. Since $D_{i+1}$ is a convex hull, we have that
    \begin{itemize}
        \item $D_i\setminus D_{i+1}$ is a topological cylinder,
        \item any trajectory in $D_i\setminus D_{i+1}$ starting from $\partial D_{i+1}$ has to exit $D_i\setminus D_{i+1}$ through $\partial D_{i}$.
    \end{itemize}
    It follows that $\gamma$ first passes through the convex hulls $D_i$ in increasing order of~$i$, and then exits $D_j$ in decreasing order of~$j$. Therefore, $\gamma$ is cut by these convex hulls into sub-paths:
    $$
    \gamma_i, \gamma_{i+1},\ldots,\gamma_{j-1},\gamma_j,\gamma_{j+1},\ldots,\gamma_{j+l}
    $$
    where $\gamma_t$ lies in $D_t\setminus D_{t+1}$ for $1\le t<j$, $\gamma_j$ is contained in $D_j\setminus D_{j+1}$ (or $D_r$ if $j=r$), and $\gamma_{j+t}$ lies in $D_{j-t}\setminus D_{j-t+1}$ for $1\le t\le l$. 
    
For $\gamma_t$ with $i \leq t \leq j-1$, it may wind around the topological cylinder $D_t \setminus D_{t+1}$.
Note that the holonomy of the core curve of $D_t \setminus D_{t+1}$ is a rotation by an angle of $2\pi\left(1 - \sum_{D_{t+1}} k\right)$, where $\sum_{D_{t+1}} k$ is the sum of the curvatures of the singularities in $D_{t+1}$. It follows that $\gamma_t$ may wind around at most 
$$
\left\lceil \frac{\pi}{2\pi\left(1 - \sum_{D_{t+1}} k\right)} \right\rceil
$$
times.

Note that by the Euler characteristic formula, the number of faces of $T$ is bounded above by $2m$.
Since each face of $T$ intersects $D_i \setminus D_{i+1}$ into at most two connected components, the components of the faces of $T$ intersecting with the topological cylinder $D_t \setminus D_{t+1}$ are at most $4m$. 
Therefore, the combinatorial length of $\gamma_t$ is bounded above by
$$
\left\lceil \frac{\pi}{2\pi\left(1 - \sum_{D_{t+1}} k\right)} \right\rceil 4m \leq \left\lceil \frac{1}{2\delta} \right\rceil 4m \leq \frac{4m}{\delta}.
$$
The combinatorial length of $\gamma_{j+t}$ for $1 \leq t \le l$ shares the same upper bound. In addition, the combinatorial length of $\gamma_j$ can be estimated similarly, except when $j = r$ and $D_r$ is a convex polygon. In this case, the combinatorial length of $\gamma_j$ is bounded above by $m-2 < \frac{4m}{\delta}$.

Since $j\le r\le m$, the total combinatorial length of $\gamma$ is bounded above by 
$$2m\frac{4m}{\delta}= \frac{8m^2}{\delta}.$$

    For the metric length of $\gamma$, Lemma~\ref{lem:lengthbetweentwoconvexhulls} implies that $|\gamma_t|$ and $|\gamma'_t|$ are bounded above by
    $$
    2\max\Big\{1, \frac{1}{2\delta}\Big\} |\partial D_t|\le 2\max\Big\{1, \frac{1}{2\delta}\Big\} |\partial D| 
    $$
    Hence, the total metric length of $\gamma$ is bounded above by
    $$
    4m\max\Big\{1, \frac{1}{2\delta}\Big\} |\partial D|.
    $$
    
    When $X$ is a flat cone sphere, Lemma~\ref{lem:upperboundcurvaturegap} implies that $\frac{1}{2\delta} > 1$. Thus, the last assertion follows.    
\end{proof}

As an application, we deduce the following finiteness result for infinite convex flat spheres.

\begin{cor}\label{cor:finitenessinfinitesphere}
    Let $X$ be an infinite non-negative flat sphere with $n$ conical singularities and positive curvature gap $\delta$. Then the number of saddle connections on $X$ is bounded above by $(3n-6) 2^{8n^2/\delta - 1}$.
\end{cor}

\begin{proof}
    Let $T$ be a geometric triangulation of the core of $X$, denoted by $core(X)$. By setting $D = core(X)$, Theorem~\ref{thm:combinandmetriclengthinhull} implies that the combinatorial length of any saddle connection has a uniform upper bound of $\frac{8n^2}{\delta}$. Notice that, given the initial and subsequent corners of the faces of $T$ that an arc passes through, there is only one homotopy class of arcs in $X$. Since each homotopy class can contain at most one saddle connection, the total number of saddle connections on $X$ is at most $(3n-6) 2^{8n^2/\delta - 1}$.
\end{proof}

It is worth noting that this result is particularly interesting because it implies that no regular closed geodesic exists on an infinite non-negative flat sphere with a positive curvature gap. 
Indeed, the existence of a regular closed geodesic generates a parallel family of regular closed geodesics, and the maximal such family forms an immersed cylinder on the surface whose boundaries are saddle connections. 
Since the cylinder contains infinitely many saddle connections joining the singularities on its boundary, the surface would have infinitely many saddle connections. This contradicts the finite bound on the number of saddle connection obtained in Corollary~\ref{cor:finitenessinfinitesphere}.

This stands in contrast to the case of translation surfaces, where Masur showed in~\cite{Mas86} that every translation surface contains a periodic geodesic.

\section{Thick and Thin Parts of the Moduli Space}\label{bigsec:coor}

In this section, we introduce a finite cover of the moduli space of convex flat cone spheres, consisting of ``thick" and ``thin" parts. The thick parts correspond to regions of the moduli space where the relative systoles of the surfaces are bounded below by a positive number, while the thin parts correspond to regions where surfaces contain very short saddle connections.

We provide formal definitions of the thick and thin parts of the moduli space in Section~\ref{sec:multipleshortsubset}, and construct a finite cover of the moduli space in Section~\ref{sec:finitecovers}.

\subsection{Geometric forests associated to boundary strata}

Let $\mathbb{P}\Omega(\underline{k})$ denote the moduli space of convex flat cone spheres, and let $\overline{\mathbb{P}\Omega}(\underline{k})$ be the metric completion of $\mathbb{P}\Omega(\underline{k})$ with respect to Thurston's metric. Let $S$ be a boundary stratum of $\overline{\mathbb{P}\Omega}(\underline{k})$, and let $P$ be the associated admissible partition. Denote the non-singleton sets in $P$ by $p_1,\ldots,p_s$. For $[X] \in \mathbb{P}\Omega(\underline{k})$, we say that a geometric forest $F$ in $X$ is \emph{associated with} $S$ if $F$ consists of $s$ connected components $F_1,\ldots,F_s$, where the singularities in each $F_i$ are those with curvatures in $p_i$. Similarly, we say that disjoint convex hulls $D_1,\ldots,D_s$ in $X$ are \emph{associated with} $S$ if each $D_i$ is the convex hull of the singularities with curvatures in $p_i$.

We now recall the classical Fiala-Alexandrov inequality. Recall that a convex flat domain is a flat disk, and the curvature of any singularity in its interior is positive (see Definition~\ref{def:flatdomain}).

\begin{lemma}[\cite{Fiala1940/41}, \cite{Alexandrov45}]\label{lem:isoperimetric}
    Assume that $R$ is a convex flat domain. Let $k$ be sum of the curvatures in the interior of $R$. If $k<1$, then we have that
    $$
    4\pi (1- k) \mathrm{Area}(R) \le |\partial R|^2
    $$
    where $|\partial R|$ is the total length of the boundary.
\end{lemma}

For a geometric forest $F$, we define the \emph{maximal length} of $F$ to be the maximal length among the edges of $F$, denoted by $|F|_{\infty}$. Using Fiala-Alexandrov inequality, we obtain the following lemma.

\begin{lemma}\label{lem:lowerboundonnoncollapsing}
    Let $\underline{k}\in (0,1)^n$ be a curvature vector with positive curvature gap $\delta(\underline k)$. Let $X$ be an area-one convex flat cone sphere with $[X]\in \mathbb{P}\Omega(\underline{k})$. If $F$ is a geometric tree in $X$ satisfying that
    \begin{equation}\label{equ:forestsassociatedtoboundary}
        |F|_{\infty} \le \frac{\sqrt{\delta(\underline k)}}{n},
    \end{equation}
    then $F$ is associated to a certain boundary stratum of $\overline{\mathbb{P}\Omega}(\underline{k})$.
\end{lemma}

\begin{proof}
We prove the statement by contraposition.  

Assume that $F$ is not associated with any boundary stratum of $\overline{\mathbb{P}\Omega}(\underline{k})$.  
Then there exists a connected component $F'$ of $F$ such that the sum of the curvatures at the vertices of $F'$ is greater than~$1$:
\[
\sum_{x_i \in F'} k_i > 1.
\]

Let $R$ be the complement $X \setminus F'$, equipped with the induced length metric.  
Then $R$ is a convex flat domain.  
Since the total curvature satisfies $k_1 + \cdots + k_n = 2$, the sum of the curvatures at the singularities in the interior of $R$ is
\begin{equation}\label{equ:interiorcurvaturessumupto1}
\sum_{x_j \in \operatorname{Int}(R)} k_j = 2 - \sum_{x_i \in F'} k_i < 1,
\end{equation}
where $\operatorname{Int}(R)$ denotes the interior of $R$.

Applying Lemma~\ref{lem:isoperimetric} to the domain $R$, we obtain
\[
4\pi\!\left(1 - \sum_{x_j \in \operatorname{Int}(R)} k_j\right)\!\mathrm{Area}(R) \le |\partial R|^2.
\]
Since the boundary $\partial R$ is obtained by cutting along $F'$, we have
\[
|\partial R| = 2|F'| \le 2n|F'|_{\infty} \le 2n|F|_{\infty}.
\]
It follows that
\begin{equation}\label{equ:edgeoftreecontrolcomplement}
4\pi\!\left(1 - \sum_{x_j \in \operatorname{Int}(R)} k_j\right)\!\mathrm{Area}(R)
\le |\partial R|^2
\le 4n^2|F|_{\infty}^2.
\end{equation}
Since
\[
\delta(\underline{k}) \le 1 - \sum_{x_j \in \operatorname{Int}(R)} k_j,
\]
and $\mathrm{Area}(R) = \mathrm{Area}(X) = 1$, we deduce that
\[
4\pi\delta(\underline{k})
\le 4\pi\!\left(1 - \sum_{x_j \in \operatorname{Int}(R)} k_j\right)
= 4\pi\!\left(1 - \sum_{x_j \in \operatorname{Int}(R)} k_j\right)\!\mathrm{Area}(R)
\le 4n^2|F|_{\infty}^2.
\]
Hence,
\[
|F|_{\infty} \ge \frac{\sqrt{\pi\,\delta(\underline{k})}}{n} > \frac{\sqrt{\delta(\underline{k})}}{n},
\]
which gives the negation of assumption~\eqref{equ:forestsassociatedtoboundary}.

\end{proof}

\subsection{Thick and thin parts}\label{sec:multipleshortsubset}

Let $\mathbb{P}\Omega(\underline{k})$ be the moduli space of convex flat cone spheres with $n$ singularities. Given an integer $1 \leq d \leq n-3$ and a positive real number $\epsilon$, we define a subset $U_{d,\epsilon}$ of $\mathbb{P}\Omega(\underline{k})$ as follows:
\begin{equation}\label{equ:ngbh}
    \begin{aligned}
    U_{d,\epsilon} := &\left\{[X] \in \mathbb{P}\Omega(\underline{k}) :
    \mbox{there is a geometric forest on } X \mbox{ with } d \mbox{ edges} \right.\\
        &\left. \mbox{ such that } \ell_{e}(X) < \epsilon \mbox{ for each edge } e \right\},
    \end{aligned}
\end{equation}
where $\ell_{e}(X)$ is the \emph{normalized length} of the saddle connection $e$, defined as
\begin{equation}
    \ell_{e}(X) = \frac{|e|}{\sqrt{\mathrm{Area}(X)}}.
\end{equation}
We call a geometric forest in $X$ that satisfies the condition of $U_{d,\epsilon}$ an \emph{$\epsilon$-geometric forest}.

\begin{definition}
    Let $\mathbb{P}\Omega(\underline{k})$ be the moduli space of convex flat cone spheres with $n$ singularities. Given a positive real number $\epsilon$, we define the complement of $U_{1,\epsilon}$ as the \emph{thick part} of $\mathbb{P}\Omega(\underline{k})$ associated with $\epsilon$.
\end{definition}

Let $S$ be a boundary stratum of the metric completion of $\mathbb{P}\Omega(\underline{k})$. According to Theorem~\ref{lem:thurstoncompletion}, the boundary stratum $S$ is isometric to a certain lower dimensional moduli space of flat cone spheres. We refer to $S_{\lambda}$ as the thick part of $S$ where saddle connections on the surfaces have normalized lengths bounded below by $\lambda$. 

For a convex hull $D$ in $X$, we say that a geometric tree $F$ is a \emph{spanning tree} of the convex hull if it is contained in $D$ and connects all singularities in $D$. Additionally, if $F$ is a geometric forest, we say that $F$ is a \emph{spanning forest} of disjoint convex hulls $D_1,\ldots,D_s$ if each $D_i$ has a component of $F$ as its spanning tree.

\begin{definition}\label{def:ngbhassobdy}
    Let $\mathbb{P}\Omega(\underline{k})$ be the moduli space of convex flat cone spheres, and let $U_{d,\epsilon}$ be the subset defined in~\eqref{equ:ngbh}. Let $S$ be a $d$-codimensional boundary stratum of $\overline{\mathbb{P}\Omega}(\underline{k})$. 
    
    When $1\leq d \le n-3$, we define the \emph{thin part} $U_{d,\epsilon}(S_{\lambda})$ associated with $(S,\epsilon,\lambda)$ as the subset of $U_{d,\epsilon}$ consisting of the homothety classes $[X]$ such that there is an $\epsilon$-geometric forest $F$ in $X$ satisfying:
    \begin{itemize}
        \item the forest $F$ is associated with $S$,
        \item there exist convex hulls $D_1,\ldots,D_s$ such that $F$ is a spanning forest of them,
        \item by applying generalized Thurston surgeries along all $D_i$, the resulting surface $X^{(0)}$ satisfies $[X^{(0)}] \in S_{\lambda}$.
    \end{itemize}
    
    When $d = n-3$, we define the \emph{thin part} $U_{n-3,\epsilon}(S)$ associated with $(S,\epsilon)$ similarly, except that $S_{\lambda}$ is replaced by $S$, i.e., $\lambda = 0$.
\end{definition}

When $d = n - 3$, the boundary stratum has dimension zero and consists of a single convex flat cone sphere $X$, which is the pillowcase of a triangle. In particular, this boundary stratum $S$ coincides with $S_{\lambda_0}$, where $\lambda_0$ denotes the normalized length of the shortest saddle connection on $X$.

For simplicity, we use the notation $U_{d,\epsilon}(S_{\lambda})$ for a thin part associated with a boundary stratum $S$ of any codimension $1\le d\le n-3$. In the case $d = n-3$, we \emph{regard $U_{d,\epsilon}(S_{\lambda})$ as $U_{n-3,\epsilon}(S)$}; \emph{this convention will be used throughout the paper} to provide a unified notation for all thin parts.

The following lemma shows that the length of a spanning tree of a convex hull controls the size of the convex hull.

\begin{lemma}\label{lem:bound}
    Let $X$ be a convex flat cone sphere with $n$ singularities and positive curvature gap $\delta$. Assume that $D$ is a convex hull in $X$ and $F$ is a geometric spanning tree of $D$. Let $C_D$ be the added cone during the generalized Thurston surgery along $D$. Then, we have that $|\partial D| \leq 2(n-3)|F|_{\infty}$, where $|F|_{\infty}$ is the maximal length among the edges of $F$. In particular, the distance from the apex of $C_D$ to the boundary is bounded above by $\frac{n-3}{2\delta}|F|_{\infty}$.
\end{lemma}

\begin{proof}
    Since $F$ is spanning in $D$, any component of $D \setminus F$ contains exactly one edge in the boundary from $\partial D$ not in $F$. By the triangle inequality, the length of this edge is at most the sum of the lengths of the other edges in the component. Summing over the edges on the boundary $\partial D$, we get
    $$
    |\partial D| \leq 2|F|,
    $$
    where $|F|$ is the sum of the lengths of all edges in $F$. Since $F$ has at most $n-3$ edges, it follows that $|\partial D| \leq 2(n-3)|F|_{\infty}$.
    \par
    By Lemma~\ref{lem:estdiscone}, the distance from the apex of $C_D$ to the boundary $\partial D$ is bounded above by 
    $$
    \frac{1}{2} \max\left\{1, \frac{1}{2(1-\sum_{D}k)}\right\} |\partial D| \leq
    \frac{1}{2\delta}(n-3)|F|_{\infty},
    $$
    where $\sum_{D}k$ is the sum of the curvatures of the singularities in $D$. Since Lemma~\ref{lem:upperboundcurvaturegap} implies that $\delta \leq \frac{1}{3}$, we have $\max\left\{1, \frac{1}{2\delta}\right\} = \frac{1}{2\delta}$.
\end{proof}

\subsection{$(d, \epsilon)$-Convex hulls}

In this section, we study surfaces in the thin part $U_{d,\epsilon}(S_{\lambda})$. 

We begin by introducing the following notion. Let $X$ be a non-negative flat cone sphere, and let $D_1, \ldots, D_s$ be disjoint convex hulls in $X$. A trajectory is said to be \emph{maximal with respect to} $D_1, \ldots, D_s$ if its interior is disjoint from these convex hulls, and its endpoints lie on the boundaries of the convex hulls or at singularities.

The following lemma addresses the uniqueness of a list of convex hulls in a convex flat cone sphere whose union contains $\epsilon$-geometric forests.

\begin{prop}\label{prop:maxiamltrajectory}
    Let $\underline{k}\in (0,1)^n$ be a curvature vector with positive curvature gap $\delta = \delta(\underline k)$ and let $S$ be a boundary stratum of $\overline{\mathbb{P}\Omega}(\underline{k})$ of codimension $1\le d\le n-3$. Assume that $\epsilon$ and $\lambda$ are positive numbers satisfying that
    \begin{equation}\label{equ:restriction1}
        0<\epsilon \le \frac{\delta^2}{n} \mbox{ and }\lambda\ge \big(1+\frac{n}{\delta}\big)\epsilon.
    \end{equation}
    Let $X$ be a flat sphere with $[X]$ in a thin part $U_{d,\epsilon}(S_{\lambda})$. Let $D_1,\ldots, D_s$ be disjoint convex hulls in $X$ such that there is an $\epsilon$-geometric forest which is a spanning forest of these convex hulls. Then any maximal trajectory with respect to $D_1,\ldots,D_s$ is of normalized length larger than $\epsilon$.
    \par
    In particular, we have that $D_1,\ldots,D_s$ are the unique convex hulls in $X$ such that any $\epsilon$-geometric forest with $d$ edges are contained in these convex hulls. 
\end{prop}

\begin{proof}
    Let $\gamma$ be a maximal trajectory with respect to $D_1, \ldots, D_s$. Apply Thurston surgeries along these convex hulls, and let $X^{(0)}$ denote the resulting top convex flat cone sphere. By Lemma~\ref{lem:generalizedremains}, we can regard $\gamma$ as a trajectory in $X^{(0)}$. Next, we will estimate the length of $\gamma$ in $X^{(0)}$.
    \par
    We first consider the case when $1\le d < n-3$, or equivalently $\dim S > 0$. If $\gamma$ ends at singularities both outside the convex hulls, then $\gamma$ is a saddle connection in $X^{(0)}$. Recall that Lemma~\ref{lem:areaestimates} says that
    $$
    \mathrm{Area}(X)\le \mathrm{Area}(X^{(0)})
    $$
    By the definition of $S_{\lambda}$ and Lemma~\ref{lem:areaestimates},
    $$
    \frac{|\gamma|}{\sqrt{\mathrm{Area}(X)}}\ge \frac{|\gamma|}{\sqrt{\mathrm{Area}(X^{(0)})}}\ge\lambda \ge \big(1+\frac{n}{\delta}\big)\epsilon >\epsilon.
    $$
    \par
    If $\gamma$ has one endpoint at a singularity outside the convex hulls and the other endpoint on some boundary $\partial D_i$ for some $i$, denote this endpoint by $y$. Since $D_i$ is replaced by the added cone $C_{D_i}$ in the Thurston surgery of $X^{(0)}$, 
    we can join the endpoint $y$ to the apex of $C_{D_i}$ by a radial segment $r$ in $C_{D_i}$. Let $\gamma r$ be the concatenation of $\gamma$ and $r$. Notice that by taking the shortest representative of the homotopy class of $\gamma r$ in $X^{(0)}$, the length of $\gamma r$  is bounded below by the length of a saddle connection in $X^{(0)}$. Since $[X^{(0)}]$ is in $S_{\lambda}$, it follows that 
    $$
    \frac{|\gamma|}{\sqrt{\mathrm{Area}(X)}} 
    \ge \frac{|\gamma|}{\sqrt{\mathrm{Area}(X^{(0)})}}
    = \frac{|\gamma r|}{\sqrt{\mathrm{Area}(X^{(0)})}}-\frac{|r|}{\sqrt{\mathrm{Area}(X^{(0)})}}> \lambda-\frac{|r|}{\sqrt{\mathrm{Area}(X^{(0)})}}.
    $$
    By definition of $U_{d,\epsilon}(S_{\lambda})$ and Lemma~\ref{lem:bound}, 
    we have that $\frac{|r|}{\sqrt{\mathrm{Area}(X^{(0)})}}< \frac{n-3}{2\delta}\epsilon$. Hence, we have that 
    $$
    \frac{|\gamma|}{\sqrt{\mathrm{Area}(X)}}\ge \lambda - \frac{n-3}{2\delta}\epsilon \ge \big(1+\frac{n}{\delta}\big)\epsilon - \frac{n-3}{2\delta}\epsilon > \epsilon.$$
    \par
    If the endpoints of $\gamma$ are both on the boundary of the convex hulls,
    we denote the endpoints by $q_1$ and $q_2$ and assume that $q_i\in \partial D_{j_1}$ and $q_i\in \partial D_{j_2}$ (possibly $j_1 = j_2)$. Connect each $q_i$ to the apex of the cone $C_{D_{j_i}}$ by a radial segment $r_i$. Let $\gamma'$ be the concatenation $r_1\gamma r_2$ in $X^{(0)}$.
    \par
    If $C_{\partial D_{j_1}}$ and $C_{\partial D_{j_2}}$ are different, the corresponding $\gamma'$ is in a non-trivial homotopy class. If $C_{\partial D_{j_1}} = C_{\partial D_{j_2}}$, notice that $\gamma'$ is still in a non-trivial homotopy class, otherwise $\gamma$ forms a bigon with the boundary of $C_{D_{j_1}}$ which contradicts with the Gauss-Bonnet formula. 
    \par
    Then the above implies that the shortest representative of the class of $\gamma'$ consists of at least one saddle connection.
    Since $X^{(0)}$ is in $S_{\lambda}$, we obtain that $\gamma'$ is of normalized length at least $\lambda$ in $X^{(0)}$. Combined with Lemma~\ref{lem:areaestimates}, it follows that
    $$
    \frac{|\gamma|}{\sqrt{\mathrm{Area}(X)}} > \frac{|\gamma|}{\sqrt{\mathrm{Area}(X^{(0)})}} = \frac{|\gamma'|}{\sqrt{\mathrm{Area}(X^{(0)})}}-\frac{|r_1|+|r_2|}{\sqrt{\mathrm{Area}(X^{(0)})}}>\lambda - \frac{|r_1|+|r_2|}{\sqrt{\mathrm{Area}(X^{(0)})}}.
    $$
    Again, by Lemma~\ref{lem:bound} we have that $\frac{|r_i|}{\sqrt{\mathrm{Area}(X^{(0)})}}\le \frac{n-3}{2\delta}\epsilon$ for $i=1, 2$. Therefore, we have that
    $$\frac{|\gamma|}{\sqrt{\mathrm{Area}(X^{(0)})}} > \lambda - \frac{n-3}{\delta}\epsilon \ge \big(1+\frac{n}{\delta}\big)\epsilon - \frac{n-3}{\delta}\epsilon > \epsilon.$$
    \par
    Finally, we consider the case of $d = n-3$. In this case $S$ consists of a single class of a flat sphere $X_0$ with three singularities, and $X_0$ is a double of a triangle. We estimate the normalized length of the shortest saddle connection in $X_0$. Denote the three angles of the triangle by $A, B, C$ and the length of the edge opposite to these angles $a,b,c$ respectively. Assume that $a\le b\le c$. Hence we have that $\pi\delta<A\le B<\frac{\pi}{2}$. By the law of sines, the normalized length of the edge $a$ in $X_0$ is 
    $$\frac{a}{\sqrt{bc\sin A}} = \sqrt{\frac{\sin A}{\sin B\sin C}}$$
    and is bounded below by $\sqrt{\sin A}\ge\sqrt{\frac{2}{\pi}A} \ge \sqrt{2\delta}$. It implies that any saddle connection with different endpoints in $X^{(0)}$ is of normalized length at least $\sqrt{2\delta}$. 
    Notice that a simple closed saddle connection has length equal to twice a height of the triangle. Again, by the law of sines, we have that the normalized length of the shortest simple closed saddle connection is 
    $$\frac{(ab\sin C) / c}{\sqrt{ab\sin C}} = \sqrt{\frac{\sin A \sin B}{\sin C}} \ge \sqrt{\frac{2}{\pi} A\frac{2}{\pi} B}\ge 2\delta.$$
    It follows that any closed saddle connection in $X_0$ is of normalized length at least $2\delta$. Therefore any saddle connection in $X_0$ is of normalized length at least $2\delta$.
    \par
    For $[X]\in U_{n-3,\epsilon}(S)$ and any maximal trajectory $\gamma$ in $X$,
    we basically follow the same analysis as in the case of $1\le d < n-3$. The only difference is that we replace $\lambda$ by $2\delta$. Hence we obtain that
    $$
    \frac{|\gamma|}{\sqrt{\mathrm{Area}(X)}} \ge  2\delta - \frac{n-3}{\delta}\epsilon > \epsilon
    $$
    where the last inequality follows directly from $0<\epsilon \le \frac{\delta^2}{n}$.
    \par
    The last assertion follows directly from the above. Indeed, if an $\epsilon$-geodesic $F$ is not contained within $D_1, \ldots, D_s$, then there must be a maximal trajectory contained in some edge of $F$. By the previous analysis, the normalized length of this edge would exceed $\epsilon$, leading to a contradiction.
\end{proof}

\begin{definition}\label{def:uniqueconvexhulls}
    Let $X$ be a convex flat cone sphere with $n$ singularities. Given an integer $1 \leq d \leq n-3$ and a positive real number $\epsilon$, we say that $d$ disjoint convex hulls in $X$ are \emph{$(d, \epsilon)$-convex hulls} of $X$ if any $\epsilon$-geometric forest with $d$ edges in $X$ is a spanning forest of these convex hulls.
\end{definition}

\begin{remark}\label{rmk:disjointthin}
    Proposition~\ref{prop:maxiamltrajectory} implies that if $\epsilon$ and $\lambda$ satisfy condition~\eqref{equ:restriction1}, then any surface from a thin part $U_{d,\epsilon}(S_{\lambda})$ has unique $(d, \epsilon)$-convex hulls.
\end{remark}

\subsection{Technical lemmas}

In this section, we prove three technical lemmas that will be used in the proof of Theorem~\ref{thm:decomp}.

The following lemma provides a sufficient condition for determining when a surface from $U_{d,\epsilon}$ is contained in a thin part. It can be seen as an application of Lemma~\ref{lem:embeddedregion}.

\begin{lemma}\label{lem:additionalshort}
    Let $\mathbb{P}\Omega(\underline{k})$ be a moduli space of convex flat cone spheres with $n$ singularities and positive curvature gap $\delta = \delta(\underline k)$.
    Let $\epsilon$ be a real number satisfying the condition~\eqref{equ:restriction1} and let $d$ be an integer with $1\le d< n-3$.
    For $[X]\in U_{d,\epsilon}$, let $F$ be an $\epsilon$-geometric forest of $d$ edges in $X$.  Denote by $F_1, \ldots, F_s$ the components of $F$ and let $\alpha_i$ be the homotopy class of the loop surrounding $F_i$ for each $i$.
    Assume that for any simple saddle connection $\gamma$ in $X$ satisfying that 
    \begin{itemize}
        \item $\gamma$ is disjoint from $F$
        \item  or the homotopy class of $\gamma$ has a non-zero intersection number with $\alpha_i$ for some $i$,
    \end{itemize}
    we have that
    \begin{equation}\label{equ:1}
        \ell_{\gamma}(X)\ge\frac{4n}{\delta}\epsilon.
    \end{equation}
    Then it follows that $[X]$ lies in $U_{d,\epsilon}(S_{\lambda})$ for some boundary stratum $S$ of codimension~$d$, where $\lambda = (1 + \tfrac{n}{\delta})\epsilon$.
\end{lemma}

\begin{proof}
    According to the condition~\eqref{equ:restriction1}, we have that $0<\epsilon \le \frac{\delta^2}{n} < \frac{\sqrt{\delta}}{n}$. Then Lemma~\ref{lem:lowerboundonnoncollapsing} implies that $F$ is associated to some boundary stratum $S$.
    Since $\frac{4n}{\delta}\epsilon > 2d\epsilon \ge 2|F|$, Lemma~\ref{lem:embeddedregion} implies that the convex hull of the singularities in $F_i$ along the class $\alpha_i$ exists for each $i$, denoted by $D_i$. Applying the Thurston surgery along $D_1,\ldots,D_s$, we denote the resulting top surface by $X^{(0)}$. Recall that $C_{D_i}$ is the added cone during the Thurston surgery. Notice that $X^{(0)}$ is contained in $S$.
    \par
    Next, we need to show that $[X^{(0)}]$ is contained in $S_{\lambda}$ where $\lambda = 2(1+\frac{n}{\delta})\epsilon$. Equivalently, we need to show any saddle connection in $X^{(0)}$ is of normalized length bounded below by $\lambda$.
    \par
    First notice that by equation~\eqref{equ:area}, we have that
    $$
    \mathrm{Area}(X^{(0)}) - \mathrm{Area}(X) \le \sum_{i=1}^s\mathrm{Area}(C_{D_i}),
    $$
    so
    $$
    \frac{\mathrm{Area}(X)}{\mathrm{Area}(X^{(0)})} \ge 1 - \frac{\sum_{i=1}^s\mathrm{Area}(C_{D_i})}{\mathrm{Area}(X^{(0)})}.
    $$
    By using Lemma~\ref{lem:bound}, we get that the area of the added cone $C_{D_i}$ is bounded above by $\frac{\pi}{2\delta^2}(n-3)^2|F_i|^2_{\infty}$. Since $F$ is $\epsilon$-geometric, it follows that 
    $$
    \frac{\mathrm{Area}(X)}{\mathrm{Area}(X^{(0)})} \ge 
    1 - \frac{\pi}{2\delta^2}(n-3)^2\epsilon^2 \ge 1 - \frac{\pi}{2}\frac{(n-3)^2\delta^2}{n^2} > \frac{1}{4}
    $$
    where the last inequality is because $\delta<\frac{1}{3}$ by Lemma~\ref{lem:upperboundcurvaturegap}.
    \par
    Let $\gamma$ be a saddle connection in $X^{(0)}$.
    If $\gamma$ is disjoint from any $C_{D_i}$, $\gamma$ appears in $X$ as a saddle connection disjoint from $F$. By condition~\eqref{equ:1},
    $$
    \frac{|\gamma|}{\sqrt{\mathrm{Area}(X^{(0)})}}=\frac{|\gamma|}{\sqrt{\mathrm{Area}(X)}}\sqrt{\frac{\mathrm{Area}(X)}{\mathrm{Area}(X^{(0)})}}>\frac{1}{2}\frac{|\gamma|}{\sqrt{\mathrm{Area}(X)}}\ge\frac{2n}{\delta}\epsilon>\lambda.
    $$
    \par
    If $\gamma$ intersects with some added cones, let $\gamma_1$ be a component of $\gamma$ outside all $C_{D_i}$.
    Denote the endpoints of $\gamma_1$ by $q_1$ and $q_2$.
    The estimate is similar to the estimate of the length of a maximal trajectory in Proposition~\ref{prop:maxiamltrajectory}.  
    We state the proof for the case of $q_1$ and $q_2$ both in the boundaries of added cones. The proof of the case that only one of $q_1$ and $q_2$ is in the boundary of an added cone is basically the same.
    \par
    Assume that $q_1\in \partial D_{j_1}$ and $q_2\in \partial D_{j_2}$.
    Connect each $q_i$ to the nearest singularities in $\partial D_{j_i}$ by a sub-path $s_i$ contained the boundary $\partial D_{j_i}$.  
    Let $\gamma_1'$ be the concatenation of $s_1$, $\gamma'$ and $s_2$. By taking the shortest representative of the homotopy class of the path $\gamma'_1$, we could find a simple saddle connection $\gamma_2'$ contained in the shortest saddle connection such that its homotopy class has a non-zero geometric intersection number with $\alpha_{j_1}$ or $\alpha_{j_2}$. Hence our condition~\eqref{equ:1} implies that
    $$
    \frac{|\gamma_1'|}{\sqrt{\mathrm{Area}(X)}}\ge
    \frac{|\gamma_2'|}{\sqrt{\mathrm{Area}(X)}} \ge
    \frac{4n}{\delta}\epsilon.
    $$
    Moreover, Lemma~\ref{lem:bound} implies that $$
    \frac{|s_i|}{\sqrt{\mathrm{Area}(X)}} 
    \le \frac{|\partial D_{j_i}|}{\sqrt{\mathrm{Area}(X)}}
    <2(n-3)\epsilon.$$
    Thus we obtain that
    \begin{align*}
        \frac{|\gamma|}{\sqrt{\mathrm{Area}(X^{(0)})}}
        >\frac{1}{2}\frac{|\gamma_1|}{\sqrt{\mathrm{Area}(X)}} 
        &=\frac{1}{2}(\frac{|\gamma_1'|}{\sqrt{\mathrm{Area}(X)}}-\frac{|s_1|+|s_2|}{\sqrt{\mathrm{Area}(X)}})\\
        & \ge \frac{2n}{\delta}\epsilon - 2(n-3)\epsilon\\
        & > \lambda,
    \end{align*}
    where the last inequality is because Lemma~\ref{lem:upperboundcurvaturegap} implies that $\delta\le \frac{1}{3}$.
    \par
    Hence, it follows that $[X^{(0)}]$ is in $S_{\lambda}$ and by definition of thin parts $[X]$ is contained in $U_{d,\epsilon}(S_{\lambda})$.
\end{proof}

Next, we prove a technical lemma that shows, under certain conditions, a geometric forest can be enlarged by adding one edge, with a controlled bound on the length of the new edge.

\begin{lemma}\label{lem:zigzagclosed}
    Let $\mathbb{P}\Omega(\underline{k})$ be a moduli space of convex flat cone spheres with $n$ singularities and positive curvature gap $\delta = \delta(\underline k)$. For $[X]\in U_{d,\epsilon}$, let $F$ be an $\epsilon$-geometric forest with $d$ edges in $X$. Let $F_1, \ldots, F_s$ be the connected components of $F$ and let $\alpha_i$ be the homotopy class of the loop surrounding $F_i$ for each $i$. 
    Assume that there exists a simple saddle connection $\gamma$ in $X$ such that:
    \begin{itemize}
        \item the endpoints of $\gamma$ are in $F_i$ and $\gamma$ intersects with $F$ only at the endpoints,
        \item the homotopy class of $\gamma$ has a non-zero geometric intersection number with $\alpha_i$, and has zero geometric intersection numbers with $\alpha_j$ for $j\ne i$,
        \item the normalized length of $\gamma$ is less than a positive real number $L$.
    \end{itemize}
    Then one can find a simple saddle connection $\gamma'$ in $X$ satisfying that
    \begin{itemize}
        \item the geometric forest $F$ together with $\gamma'$ form a new geometric forest $F'$ with $d+1$ edges,
        \item the normalized length of $\gamma'$ satisfies that 
        $$
        \ell_{\gamma'}(X) < \frac{1}{\delta}(2d\epsilon + L).
        $$
    \end{itemize}
\end{lemma}

\begin{proof}
    Without loss of generality, we assume that $F_i$ is $F_1$. We also assume that the area of $X$ is one so that the normalized length is just the metric length.
    By the first assumption on $\gamma$, we know that 
    if we cut along $F_1\cup\gamma$, the sphere $X$ is decomposed into two connected components. Notice that the flat metric of $X$ induces a length metric on each component.
    We denote by $R_1$ and $R_2$ the metric completion of the induced length metric on the components of $X\setminus(F_1\cup\gamma)$ respectively. 
    According to the second assumption on $\gamma$, we know that each $R_i$ contains at least one singularity in the interior.
    Hence that $R_1$ and $R_2$ are flat domains with at least one singularity in the interior.
    Since $\partial R_1$ consists of $\gamma$ and edges of $F_1$, we have that
    $$
    |\partial R_1|\le 2d\epsilon + L.
    $$
    Moreover, one can assume that the sum of the curvatures in $R_1$ is less than $1$.
    \par
    If there is only one singularity in the interior of $R_1$. Let $\gamma'$ be a saddle connection from the singularity in the interior to $\partial R_1$.
    Notice that $F$ and $\gamma'$ form a geometric forest with $d+1$ edges.
    By Lemma~\ref{lem:estdiscone}, we obtain that 
    $$
    |\gamma'|<\frac{1}{2}\max\{1,\frac{1}{2\delta}\}(2d\epsilon + L)=\frac{1}{4\delta}(2d\epsilon + L)<\frac{1}{\delta}(2d\epsilon + L).
    $$
    where the second equality is due to Lemma~\ref{lem:upperboundcurvaturegap}.
    \par
    Next we consider the case that there are at least two singularities in the interior of $R_1$.
    Let
    $\alpha_1$ be the homotopy class of a loop surrounding $\partial R_1$ in $R_1$ and let
    $\gamma_1$ be a shortest representative of $\alpha_1$ in $R_1$. 
    Notice that any edge of $F_i$ for $i\ne 1$ is disjoint or contained entirely in $\gamma_1$.
    \par
    If $\gamma_1$ passes through no singularities in the interior of $R_1$, then as in the proof of Lemma~\ref{lem:embeddedregion}, one can compute by using the Gauss-Bonnet formula that the sum of the curvatures of the singularities in the interior of $R_1$ is at least $1$, which contradicts with the assumption on $R_1$. Therefore, the shortest representative $\gamma_1$ passes through at least one singularity in the interior of $R_1$.
    \par
    If $\gamma_1$ passes through singularities in $F_1$ and also singularities in the interior of $R_1$, then we select $\gamma'$ to be the saddle connection contained in $\gamma_1$ which joins a singularity from $F_1$ to an interior singularity.
    It follows that
    $$
    |\gamma'| \le |\gamma_1| \le |\partial R_1| \le  2d\epsilon + L<\frac{1}{\delta}(2d\epsilon + L).
    $$
    Thus $\gamma'$ also satisfy all the conclusions in this case.
    \par
    If $\gamma_1$ passes through no singularities in $F_1$, 
    then $\gamma_1$ is completely contained in the interior of $R_1$ and it encloses a convex hull $D$ of the singularities in the interior of $R_1$.
    Let $\gamma'$ be a saddle connection from a singularity in $\partial R_1$ to a singularity in $\partial D$. Notice that $F$ and $\gamma'$ form a geometric forest with $d+1$ edges.
    \par
    According to Lemma~\ref{lem:lengthbetweentwoconvexhulls}, the length of $\gamma'$ is bounded above by $2\max\left\{1,\frac{1}{2\delta}\right\}|\partial R|$. Combining with Lemma~\ref{lem:upperboundcurvaturegap} and $|\partial R|\le 2d\epsilon + L$, we obtain that $$|\gamma'|\le \frac{1}{\delta}(2d\epsilon + L).$$  Hence $\gamma'$ satisfies all the conclusions in this case.
    \par
    Therefore, we find a desired $\gamma'$ in all cases.
\end{proof}

The following lemma is a consequence of Lemma~\ref{lem:zigzagclosed} and will be used directly in the proof of Theorem~\ref{thm:decomp}.

\begin{lemma}\label{lem:oneadditionalshort}
    Let $\mathbb{P}\Omega(\underline{k})$ be a moduli space of convex flat cone spheres with $n$ singularities and positive curvature gap $\delta = \delta(\underline k)$. For $[X]\in U_{d,\epsilon}$, let $F$ be an $\epsilon$-geometric forest with $d$ edges in $X$. Assume that $\gamma$ is a simple saddle connection in $X$ with normalized length less than $L$ such that:
    \begin{enumerate}
        \item $\gamma$ is disjoint from $F$,
        \item or the homotopy class of $\gamma$ has a non-zero intersection number with the homotopy class of a loop surrounding some component of $F$.
    \end{enumerate}
    Then one can find a simple saddle connection $\gamma'$ in $X$ satisfying that
    \begin{itemize}
        \item the geometric forest $F$ together with $\gamma'$ form a new geometric forest $F'$ with $d+1$ edges,
        \item the length of $\gamma'$ satisfies that 
        $$
        \ell_{\gamma'}(X) < \frac{1}{\delta}(2d\epsilon + L + 2\epsilon).
        $$
    \end{itemize}
\end{lemma}

\begin{proof}
    We assume the the area of $X$ is one so that the  normalized length is just the metric length. We find $\gamma'$ for different cases of $\gamma$.
    \par
    For the case $(1)$, if $\gamma$ has different endpoints, then $F\cup\gamma$ is a geometric forest $F'$ with $d+1$ edges in $X$ and the normalized length of each edge of $F'$ is bounded above by $L$.
    \par
    If $\gamma$ has the same endpoints, then let $D_0$ be the side of $\gamma$ whose angle of the corner at the endpoint is at most $\pi$.
    If there is only one singularity in the interior of $D_0$, then we consider a radial segment $\gamma'$ from the apex of $D_0$ to the endpoint of $\gamma$. If there is more than one singularity in the interior of $D_0$, then let $D_1$ be the convex hull of the singularities in the interior (such a convex hull exists because the corner of $D_0$ is at most $\pi$). We take $\gamma'$ to be a saddle connection in $D_0\setminus D_1$ from the singularity in $\partial D_0$ to $\partial D_1$. According to Lemma~\ref{lem:lengthbetweentwoconvexhulls}, the length of $\gamma'$ is always bounded above by
    $$
    2\max\Big\{1, \frac{1}{2\delta}\Big\}|\gamma| = \frac{1}{\delta}|\gamma|  < \frac{1}{\delta}L,
    $$
    where the first equality is due to Lemma~\ref{lem:upperboundcurvaturegap}. Moreover $F\cup\gamma$ is a geometric forest $F'$ with $d+1$ edges in $X$.
    \par
    For case $(2)$, if $\gamma$ intersects with the interior of an edge of $F$,  consider a component of $\gamma$ in the complement of $F$ and joins each of its endpoints in $F$ to the nearest vertex along edges of $F$. Denote the resulting path by $\gamma_1$. Therefore, up to replacing $\gamma$ by a saddle connection $\gamma'$ in a shortest representative of the class of $\gamma_1$, we may further require that:
    \begin{itemize}
        \item $\gamma$ is in the case $(1)$ or $(2)$,
        \item $\gamma$ has different endpoints and it intersects with $F$ at most at endpoints,
        \item the normalized length of $\gamma$ is less than 
        $L + 2\epsilon$.
    \end{itemize}
    \par
    If $\gamma$ is in the case $(1)$ or $\gamma$ joins different components of $F$, then $F\cup\gamma$ forms a geodesic $F'$ forest with $d+1$ edges with the normalized length of each edge of $F'$ is less than $L+2\epsilon$.
    \par
    If $\gamma$ joins the same component of $F$, then Lemma~\ref{lem:zigzagclosed} implies that there is a saddle connection $\gamma'$ such that $F' = F\cup\gamma'$ is a geometric forest with $d+1$ edges and the normalized length of each edge of $F'$ is bounded above by
    $\frac{1}{\delta}\left(2d\epsilon +  L + 2\epsilon \right).$
\end{proof}

\subsection{A finite open cover of the moduli space}\label{sec:finitecovers}

In this section, we introduce a finite cover of the moduli space of convex flat cone spheres using thick and thin parts.

We begin by presenting a decomposition of the subset $U_{d,\epsilon}$.

\begin{thm}\label{thm:decomp}
    Let $\mathbb{P}\Omega(\underline{k})$ be a moduli space of convex flat cone spheres with $n$ singularities and positive curvature gap $\delta = \delta(\underline k)$. 
    Assume that $\epsilon$ and $\lambda$ satisfy the condition that
    \begin{equation}\label{equ:restriction2}
        0<\epsilon \le \frac{\delta^2}{4n^2},\mbox{ and }\lambda = \big(1+\frac{n}{\delta}\big)\epsilon.    
    \end{equation}
    Then we have that: 
    \begin{itemize}
        \item When $1\le d < n-3$, the subset $U_{d,\epsilon}$ is decomposed into a disjoint union
        $$\underset{\mathrm{codim}(S)=d}{\bigsqcup} U_{d,\epsilon}(S_\lambda)$$
        where the union is over the $d$-codimensional boundary strata of $\overline{\mathbb{P}\Omega}(\underline{k})$,
        and the complement which satisfies that
        $$U_{d,\epsilon}\setminus \underset{\mathrm{codim}(S)=d}{\bigsqcup} U_{d,\epsilon}(S_\lambda) 
        \subset
        U_{d + 1, \frac{6n}{\delta^2}\epsilon},$$
        \item When $d = n-3$, the subset $U_{n-3,\epsilon}$ is a disjoint union
        $$U_{n-3,\epsilon}=\underset{\dim S = 0}{\bigsqcup}U_{n-3,\epsilon}(S)$$
        where the union is over the $0$-dimensional boundary strata of $\overline{\mathbb{P}\Omega}(\underline{k})$.
    \end{itemize}
\end{thm}

\begin{proof}
    For any $1\le d \le
    n-3$, let $[X]$ be any class contained in $U_{d,\epsilon}$. 
    Let $[X]$ be a class contained in $U_{d,\epsilon}(S_\lambda)$ for some boundary stratum $S$. Since $\epsilon$ and $\lambda$ also satisfy the condition~\eqref{equ:restriction1},
    by Remark~\ref{rmk:disjointthin}, we know that any $\epsilon$-geometric forest in $X$ with $d$ edges is contained in the same $(d,\epsilon)$-convex hulls and hence associated to the same boundary stratum $S$. This implies that $U_{d,\epsilon}(S_\lambda)$ is disjoint for different $S$.
    
    Assume that $1\le d < n-3$. If $[X]\in U_{d,\epsilon}\setminus \underset{\mathrm{codim}(S)=d}{\bigsqcup} U_{d,\epsilon}(S_\lambda)$, then $[X]$ is not contained in $U_{d,\epsilon}(S_\lambda)$ for any boundary stratum $S$ of codimension~$d$. Since this contradicts the conclusion of Lemma~\ref{lem:additionalshort}, the assumption of Lemma~\ref{lem:additionalshort} must fail for~$X$. This means that for any $\epsilon$-geometric forest $F$ in $X$, there is a simple saddle connection $\gamma$ in $X$ of normalized length 
    $$\ell_{\gamma}(X) < \frac{4n}{\delta}\epsilon$$
    such that: 
    \begin{enumerate}
        \item $\gamma$ is disjoint from $F$,
        \item or the homotopy class of $\gamma$ has a non-zero intersection number with the homotopy class of a loop surrounding some component of $F$.
    \end{enumerate}
    Since $\gamma$ satisfies the assumption of Lemma~\ref{lem:oneadditionalshort}, we obtain that we can find a saddle connection $\gamma'$ such that $F$ and $\gamma'$ form a geodesic forest with $d+1$ edges and the normalized length of the edges of the forest are bounded above by
    $$
    \frac{1}{\delta}(2d\epsilon + \frac{4n}{\delta}\epsilon + 2\epsilon)
    < \frac{6n}{\delta^2}\epsilon
    $$
    where the inequality is because $d<n-3$.
    It follows that $[X]$ is contained in $U_{d + 1, \frac{6n}{\delta^2}\epsilon}$.
    Hence we have that
    $$U_{d,\epsilon}\setminus \underset{\mathrm{codim}(S)=d}{\bigsqcup} U_{d,\epsilon}(S_\lambda) 
        \subset
        U_{d + 1, \frac{6n}{\delta^2}\epsilon}.$$
    \par
    When $d = n-3$, for $[X]\in U_{n-3,\epsilon}$, let $F$ be an $\epsilon$-geometric forest in $X$. Since $F$ has $(n-3)$ vertices, the forest $F$ is indeed a tree.
    We first show that the convex hull exists for $F$.
    The strategy is  to use Lemma~\ref{lem:embeddedregion} to prove this.
    Assume that there is a simple saddle connection $\gamma$ whose homotopy class has a non-zero geometric intersection number with the homotopy class of a loop surrounding $F$ satisfying that 
    \begin{equation}\label{equ:assumptionforzerodimen}
        \ell_{\gamma}(X) < 2(n-3)\epsilon.
    \end{equation}
    Since $\gamma$ satisfies the condition of Lemma~\ref{lem:oneadditionalshort}, 
    one can use Lemma~\ref{lem:zigzagclosed} to show that there exists a saddle connection $\gamma'$ such that $F\cup \gamma'$ is a geometric tree and the normalized length of each edge is bounded above by
    $$
    \frac{1}{\delta}\left(2(n-3)\epsilon +\ell_{\gamma}(X) +2\epsilon \right) <\frac{4(n-3)+2}{\delta}\epsilon \le
    \frac{4(n-3)+2}{4n^2}\delta
    $$
    Since $\frac{4(n-3)+2}{4n^2}\delta < \frac{\sqrt{\delta}}{n}$, Lemma~\ref{lem:lowerboundonnoncollapsing} implies that $F\cup \gamma'$ is associated to a boundary stratum of $\overline{\mathbb{P}\Omega}(\underline{k})$. However, the tree $F\cup\gamma'$ has $n-2$ edges and codimension of a boundary stratum is at most $n-3$. Hence there is no such a boundary stratum associated to the tree. Therefore, the assumption~\eqref{equ:assumptionforzerodimen} cannot be true. Then we know that the length of $\gamma$ must be at least 
    $$
    2(n-3)\epsilon\ge \frac{2|F|}{\mathrm{Area}(X)}.
    $$
    Hence Lemma~\ref{lem:embeddedregion} implies that the convex hull of the vertices of $F$ exists.  By definition, the class $[X]$ is contained in $U_{n-3,\epsilon}(S)$ for some zero dimensional boundary stratum.
\end{proof}

Given a moduli space of convex flat cone spheres $\mathbb{P}\Omega(\underline{k})$ with $n$ singularities and a positive curvature gap $\delta(\underline{k})$, we define the notations $\sigma_1(\underline{k}), \ldots, \sigma_{n-3}(\underline{k})$ as follows:
\begin{equation}\label{equ:defepsilon}
    \sigma_d(\underline{k}) = \frac{\delta(\underline{k})^2}{4n^2} \left(\frac{\delta(\underline{k})^2}{6n}\right)^{n-2-d}
\end{equation}
for $d = 1, \ldots, n-3$. Note that $\sigma_{d+1}(\underline{k}) = \frac{6n}{\delta(\underline{k})^2}\sigma_d(\underline{k}) > \sigma_d(\underline{k})$.

\begin{cor}\label{cor:thindecomp}
    Let $\mathbb{P}\Omega(\underline{k})$ be a moduli space of convex flat cone spheres with $n$ singularities and positive curvature gap $\delta(\underline k)$. We consider $\lambda_d = \big(1+\frac{n}{\delta(\underline k)}\big)\epsilon_d(\underline k)$ for $1\le d < n-3$ and $\lambda_{n-3} = 0$. Then $\mathbb{P}\Omega(\underline{k})$ is covered by 
    $\mathbb{P}\Omega(\underline{k})\setminus U_{1,\sigma_1(\underline k)}$
    and 
    \begin{equation}\label{equ:decomp}
        \bigcup\limits_{d=1}^{n-3} \underset{\mathrm{codim}(S)=d}{\bigsqcup}U_{d,\sigma_{d}(\underline k)}(S_{\lambda_{d}})
    \end{equation}
    where the second union is over all $d$-codimensional boundary strata of $\overline{\mathbb{P}\Omega}(\underline{k})$.
\end{cor}

\begin{proof}
    We only need to prove that the union~\eqref{equ:decomp} covers the subset $U_{1,\sigma_1(\underline k)}$.
    Notice that 
    $$0<\sigma_{1}(\underline k)<\ldots<\sigma_{n-3}(\underline k) <  \frac{\delta(\underline k)^2}{4n^2}.
    $$
    Denote the set $U_{1,\sigma_{1}(\underline k)}$ by $A_1$. We define $B_1$ to be
    $$
    B_1 = \underset{\mathrm{codim}(S)=1}{\bigsqcup} U_{1,\sigma_{1}(\underline k)}(S_{\lambda_{1}})
    $$
    and we define $A_2$ to be $A_2 = A_1 \setminus  B_1.$ It follows that $A_1\subset B_1\cup A_2$. Moreover,
    Theorem~\ref{thm:decomp} implies that $A_2\subset U_{2,\sigma_d(\underline k)}.$
    \par
    Assume that $B_1,\ldots, B_{d-1}$ and $A_d$ have been defined such that
    \begin{itemize}
        \item $A_1 \subset B_1\cup\ldots\cup B_{d-1} \cup A_d$,
        \item $A_d \subset U_{d,\sigma_d(\underline{k})}$,
        \item $B_i = \underset{\mathrm{codim}(S)=i}{\bigsqcup} U_{i,\sigma_{i}(\underline k)}(S_{\lambda_{i}})$ for $1\le i\le d - 1$.
    \end{itemize}
    Then we define $B_{d}$ to be
    $$
    B_{d} = \underset{\mathrm{codim}(S)=d}{\bigsqcup} U_{d,\sigma_{d}(\underline k)}(S_{\lambda_{d}})
    $$
    and we define $A_{d+1}$ to be $A_{d+1} = A_{d} \setminus  B_{d}$.
    Combining with Theorem~\ref{thm:decomp}, we obtain that
    $$
    A_{d+1}\subset U_{d,\sigma_d(\underline k)}\setminus B_{d} \subset 
    U_{d+1,\frac{6n}{\delta(\underline k)^2}\sigma_d(\underline k)}
    =U_{d+1,\sigma_{d+1}(\underline k)}.
    $$
    Moreover, we have that
    \begin{align*}
        A_{1} &\subset B_1\cup\ldots\cup B_{d-1} \cup A_d \\
        &\subset B_1\cup\ldots\cup B_{d-1} \cup B_{d}\cup A_{d+1}.
    \end{align*}
    Finally when $d=n-3$, we obtain that
    $$
    A_{n-2} = A_{n-3}\setminus B_{n-3} \subset U_{n-3,\sigma_{n-3}(\underline{k})}\setminus \bigsqcup_{\operatorname{codim}(S) = n-3} U_{n-3,\sigma_{n-3}(\underline{k})} = \emptyset
    $$
    and 
    $$
    U_{1,\sigma_1(\underline k)} = A_1 \subset B_1\cup\ldots \cup B_{n-3} = \bigcup\limits_{d=1}^{n-3} \underset{\mathrm{codim}(S)=d}{\bigsqcup}U_{d,\sigma_{d}(\underline k)}(S_{\lambda_{d}}).
    $$
\end{proof}

\section{Proof of Theorem~\ref{thm:mainfavorite}}\label{bigsec:ivslen}

In this section, we provide the proof of Theorem~\ref{thm:mainfavorite}.

\firstmainfav*

We utilize the finite cover introduced in Corollary~\ref{cor:thindecomp}. Recall that the moduli space $\mathbb{P}\Omega(\underline{k})$ is covered by the thick part  
$\mathbb{P}\Omega(\underline{k}) \setminus U_{1,\sigma_1(\underline{k})}$ and the thin part $U_{d,\sigma_{d}(\underline{k})}(S_{\lambda_{d}})$ for each boundary stratum $S$, where $\lambda_d = \left(1 + \frac{n}{\delta(\underline{k})}\right) \sigma_d(\underline{k})$ and $\sigma_{d}(\underline{k})$ for $d = 1, \ldots, n-3$ is defined in~\eqref{equ:defepsilon}.

For simplicity, we will use $\sigma_d$ in place of $\sigma_d(\underline{k})$ throughout the remainder of this section.

\subsection{Comparison in the thick part}\label{sec:thickcomp}

We first recall the following lemma from~\cite{FT}.

\begin{lemma}[Lemma 4.3, \cite{FT}]\label{lem:bounddelaunayedge}
    Let $X$ be a flat cone sphere of unit area with $n$ conical singularities and curvature gap $\delta > 0$. 
    Let $T$ be a Delaunay triangulation of $X$, and let $f$ be a triangle of $T$. Denote the radius of the circumscribed circle of $T$ by $R$. Then we have
    $R < \frac{1}{2}c(\delta)$, where $c(\delta) = \sqrt{\frac{4}{\pi} + \frac{1}{2\pi\delta}}$.
    
    In particular, the length of any edge of $T$ in $X$ is bounded above by $c(\delta)$.
\end{lemma}

Using this lemma, we obtain the following estimate for surfaces in the thick part.

\begin{prop}\label{prop:estimatethick}
    Let $\mathbb{P}\Omega(\underline{k})$ be a moduli space of convex flat cone spheres with $n$ singularities. Assume that the curvature gap $\delta(\underline{k})$ is bounded below by a positive number $\delta$. Let $X$ be a flat cone sphere of area one, with $[X] \in \mathbb{P}\Omega(\underline{k}) \setminus U_{1,\sigma_1(\underline{k})}$. 
    
    Then for any trajectory $\gamma$ in $X$, we have
    $$
    |\gamma| \geq 
    \frac{\sigma_1^2}{\sqrt{\frac{4}{\pi} + \frac{1}{2\pi\delta}}}
    \left(\frac{\sqrt{2}\delta}{36(n-2)} \sqrt{\iota(\gamma,\gamma)} - 1 \right),
    $$
    and for any saddle connection or regular closed geodesic $\gamma$ in $X$, we have
    $$
    |\gamma| \geq 
    \frac{\sigma_1^2}{\sqrt{\frac{4}{\pi} + \frac{1}{2\pi\delta}}} \frac{\sqrt{2}\delta}{36(n-2)} \sqrt{\iota(\gamma,\gamma)}.
    $$
\end{prop}

\begin{proof}
    By Theorem~\ref{thm:mainindividual} and Remark~\ref{rmk:b1b2}, we obtain that for any trajectory $\gamma$,
    \begin{align*}
    |\gamma| &\ge b_1 \sqrt{\iota(\gamma,\gamma)} - b_2\\
    &\ge\sqrt{2}\frac{relsys(X)^2}{R(X)}
    \frac{\delta}{36(2n-4)}\sqrt{\iota(\gamma,\gamma)} - \frac{relsys(X)^2}{2R(X)}\\
    & = \sqrt{2}\frac{relsys(X)^2}{2R(X)}
    \frac{\delta}{36(n-2)}\sqrt{\iota(\gamma,\gamma)} - \frac{relsys(X)^2}{2R(X)}\\
    & = \frac{relsys(X)^2}{2R(X)}\left(\frac{\sqrt{2}\delta}{36(n-2)}\sqrt{\iota(\gamma,\gamma)} - 1\right)
    \end{align*}
    and for any saddle connection or regular closed geodesic $\gamma$,
    $$|\gamma| \ge 
   \frac{relsys(X)^2}{2R(X)}\frac{\sqrt{2}\delta}{36(n-2)}\sqrt{\iota(\gamma,\gamma)}.$$
    Since $[X]$ is in the thick part $\mathbb{P}\Omega(\underline{k})\setminus U_{1,\sigma_1(\underline k)}$, we know that $relsys(X) \ge 
    \sigma_1$.
    Recall that $R(X)$ is the maximal radius of the circumscribed circles of the Delaunay triangles in $X$. By Lemma~\ref{lem:bounddelaunayedge}, we have that 
    $$
    R(X)\le \frac{1}{2}\sqrt{\frac{4}{\pi}+\frac{1}{2\pi\delta}}.
    $$
    Therefore we obtain the proposition.
\end{proof}

\subsection{A lower bound on the lengths of regular closed geodesics}\label{sec:lowercg}

In this subsection, we present a uniform lower bound for the length of regular closed geodesics on a convex flat cone sphere with a positive curvature gap.

\begin{prop}\label{prop:lowerboundclosed}
    Let $X$ be an area-one convex flat cone sphere with $n$ singularities and positive curvature gap $\delta$. Then for any regular closed geodesic $\gamma$ on $X$, the length of $\gamma$ is bounded below by $\sqrt{\pi\delta}$.
\end{prop}

\begin{proof}
    Since $X$ has a positive curvature gap, a regular closed geodesic must have at least one self-intersection. Let $\gamma$ be a regular closed geodesic in $X$. Let $R_1,$ \ldots, $R_m$ be the metric completion of the components of the complement of $\gamma$ in $X$. Notice that each $R_i$ is a flat domain. Let $c_i$ be the sum of the curvatures of the singularities in $R_i$.
    Since the angles of the corners at $\partial R_i$ are less than $\pi$, we know that $c_i<1$.
    By Lemma~\ref{lem:isoperimetric}, we have that 
    $$
    4\pi\delta\mathrm{Area}(R_i)\le |\partial R_i|^2.
    $$
    By summing over all $i$, we obtain that
    $$
    4\pi\delta\mathrm{Area}(X) \le \sum\limits_{i = 1}^m |\partial R_i|^2 \le \left( \sum\limits_{i=1}^m |\partial R_i|  \right)^2
    = 4|\gamma|^2.
    $$
    Since the area of $X$ is one, we have that $|\gamma|\ge\sqrt{\pi\delta}$.
\end{proof}

\subsection{Comparison in the thin parts}\label{sec:thincomp}

Let $\mathbb{P}\Omega(\underline{k})$ be a moduli space of convex flat cone spheres with $n$ singularities and positive curvature gap $\delta(\underline{k})$. Let $U_{d,\epsilon}(S_{\lambda})$ be a thin part of $\mathbb{P}\Omega(\underline{k})$. Let $X$ be a flat cone sphere with $[X] \in U_{d,\epsilon}(S_{\lambda})$, and let $\gamma$ be a trajectory on $X$.

In Section~\ref{sec:cornerswitching}, we used a geometric triangulation to cut a trajectory into threads. In this section, we introduce another method to cut a trajectory on a surface from a thin part. Let $D_1, \ldots, D_s$ be the $(d,\epsilon)$-convex hulls in $X$ (see Definition~\ref{def:uniqueconvexhulls}). Let $X^{(0)}$ be the top surface obtained by applying the generalized Thurston surgery along $D_1, \ldots, D_s$. 

We first decompose $\gamma$ as follows:
\begin{itemize}
    \item Let $\gamma_0$ be the components of $\gamma$ outside the $(d, \epsilon)$-convex hulls,
    \item Let $\gamma_1$ be the components of $\gamma$ inside the $(d, \epsilon)$-convex hulls.
\end{itemize}
By Lemma~\ref{lem:generalizedremains}, we can regard $\gamma_0$ as being contained in $X^{(0)}$.

Next, we select a Delaunay triangulation $T_0$ of $X^{(0)}$ and a geometric triangulation $T_i$ for each $D_i$. We then decompose $\gamma_0$ and $\gamma_1$ into sub-paths as follows:
\begin{itemize}
    \item $T_0$ cuts $\gamma_0$ into sub-paths,
    \item $T_i$ cuts $\gamma_1$ into sub-paths.
\end{itemize}
If some component of $\gamma_0$ or $\gamma_1$ lies along an edge of the geometric triangulation, we define the sub-path as the entire component. If any convex hull $D_i$ degenerates, we disregard the sub-paths obtained within $D_i$.

We denote the sub-paths obtained as $t_1, \ldots, t_m$, where $t_i$ and $t_{i+1}$ are consecutive in $\gamma$. We call $t_1, \ldots, t_m$ the \emph{decomposition of} $\gamma$ with respect to $T_0, T_1, \ldots, T_s$. Each $t_i$ is still referred as a \emph{thread} of $\gamma$, as in Section~\ref{sec:cornerswitching}.

The following lemma can be seen as a version of Lemma~\ref{lem:consecutivethick}, adapted for the decomposition introduced above.

\begin{lemma}\label{lem:consecutivethin_modified}
    Let $\mathbb{P}\Omega(\underline{k})$ be a moduli space of convex flat cone spheres with $n$ singularities and positive curvature gap $\delta(\underline{k})$. Let $S$ be a codimension $d$ boundary stratum of $\overline{\mathbb{P}\Omega}(\underline{k})$. Let $X$ be a unit area flat cone sphere with $[X] \in U_{d,\sigma_d}(S_{\lambda_d})$, and let $D_1, \ldots, D_s$ be the $(d, \epsilon)$-convex hulls in $X$. Denote by $X^{(0)}$ the top surface obtained by applying Thurston surgeries along the $(d, \epsilon)$-convex hulls. Let $T_0$ be a Delaunay triangulation of $X^{(0)}$.

    Let $\gamma$ be a trajectory on $X$, and let $t_1, \ldots, t_m$ be the decomposition of $\gamma$ associated with $U_{d,\sigma_d}(S_{\lambda_d})$ and $T_0$. Consider the positive integer
    $$
    m_0 = 4 + \left\lfloor \frac{9}{\delta(\underline{k})}n^2 \right\rfloor.
    $$
    If $m \ge m_0$, then for any $m_0$ consecutive threads $t_i, \ldots, t_{i+m_0-1}$, where $1 \leq i \leq m - m_0 + 1$, one of the following holds:
    \begin{itemize}
        \item There are two consecutive threads in $t_{i}, \ldots, t_{i+m_0 -1}$ that form a corner-switch of $T_0$,
        \item The union of $t_{i}, \ldots, t_{i+m_0-1}$ contains a maximal trajectory with respect to $D_1, \ldots, D_s$ in $X$. 
    \end{itemize}
    Moreover, we have
    \begin{equation}\label{equ:consecutivethin}
        |t_i| + \ldots + |t_{i+m_0-1}| \geq 3\sqrt{3}n^2\sigma_1^2.
    \end{equation}
\end{lemma}

\begin{proof}
    We call a thread $t_i$ \emph{special} if it has an endpoint at the interior of a triangle face or at the boundary of $(d, \epsilon)$-convex hulls (see Figure~\ref{fig:special}).
    
    \begin{figure}[!htbp]
    	\centering
    	\begin{minipage}{0.6\linewidth}
    	\includegraphics[width=\linewidth]{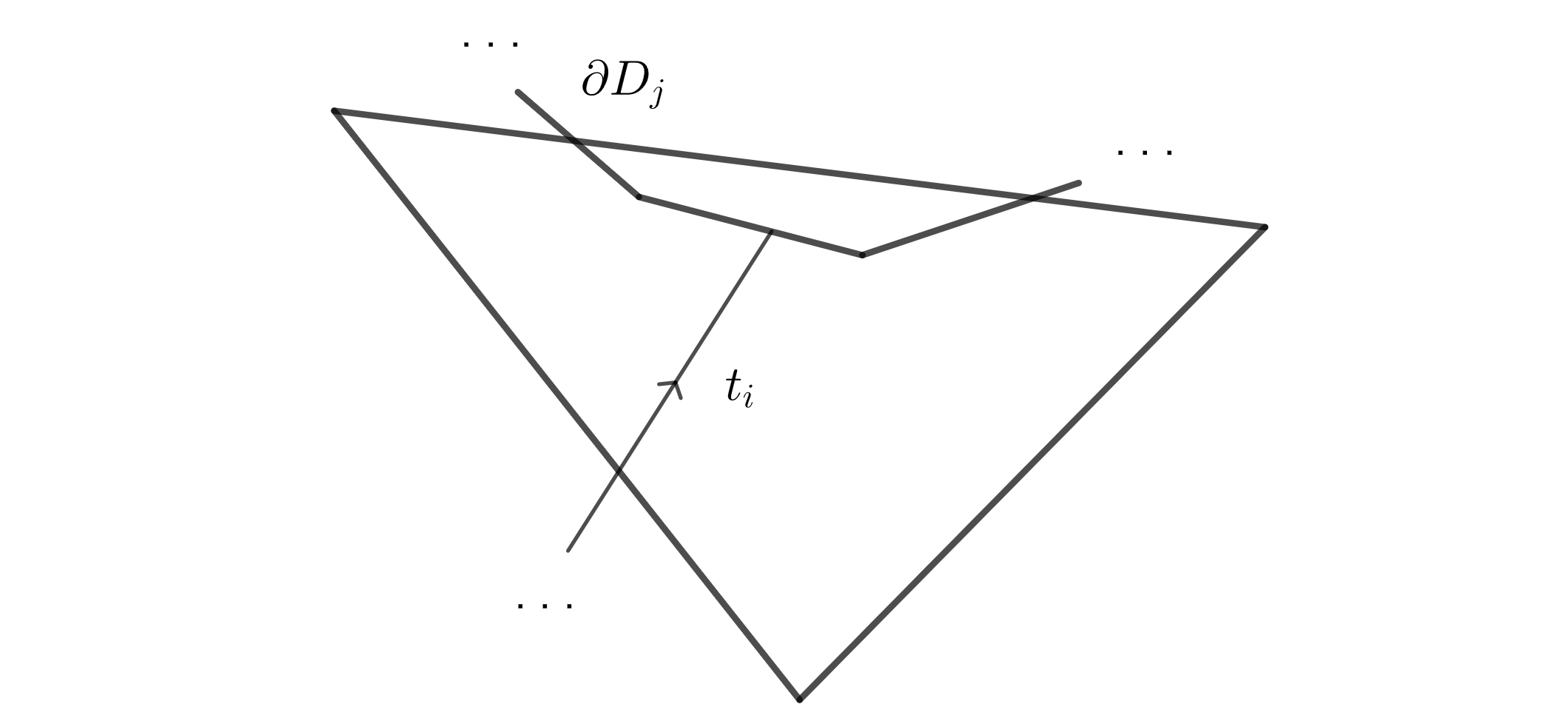}
    	\footnotesize
    	\end{minipage}
        \caption{The figure shows an example of a special thread $t_i$ which ends at the boundary of a convex hull $D_j$.}\label{fig:special}
    \end{figure}

    Let $t'$ be a union of consecutive threads in $\gamma$ such that $t'$ contains neither a corner-switch nor any maximal trajectory with respect to $D_1,\ldots,D_s$. 
    In order to maximize the number of threads contained in $t'$,  we assume that $t'$ first cuts corners of $T_0$ on the same side, then enters into a convex hull, then leaves the convex hull and cuts corners of $T_0$ on the same side again.
    \par
    To estimate the number of threads of $t'$, we note that:
    \begin{itemize}
        \item Theorem~\ref{thm:combinandmetriclengthinhull} implies that the combinatorial length of a trajectory in a convex hull on $X$ is bounded above by $\frac{8}{\delta(\underline k)}n^2$,
        \item According to the estimate~\eqref{equ:estimateforthesameside} obtained in the proof of Lemma~\ref{lem:consecutivethick}, 
        the number of consecutive threads cutting corners of $T_0$ on the same side is bounded above by  
        $$
        3(2n - 4)\Big\lceil\frac{1}{2\delta(\underline k)}\Big\rceil \le \frac{6(n-2)}{\delta(\underline k)}.
        $$
        \item There may be special threads at ends of $t'$ or before entering or after leaving convex hulls.
    \end{itemize} 
    
    By taking all of these into consideration, the 
    number of threads in $t'$ is bounded above by
    $$
    4 + \frac{6(n-2)}{\delta(\underline k)} + \frac{8}{\delta(\underline k)}n^2 < 4 + \left\lfloor\frac{9}{\delta(\underline k)}n^2\right\rfloor = m_0,
    $$ 
    where the constant~$4$ accounts for the possible special threads.

    Hence, it follows that the $m_0$ consecutive threads $t_{i},\ldots,t_{i+m_0 - 1}$ have to contain either a corner-switch or a maximal trajectory with respect to $D_1,\ldots,D_s$.
    \par
    For the last assertion of the lemma, we consider two cases.
    If $t_{i},\ldots,t_{i+m_0-1}$ contains a maximal trajectory, Proposition~\ref{prop:maxiamltrajectory} implies that
    $$
    |t_{i}|+\ldots+|t_{i+m_0-1}|\ge \sigma_{d} \ge  \sigma_{1}.
    $$
    By the definition~\eqref{equ:defepsilon}, one can compute that $\sigma_{1} > \frac{n^2}{\delta(\underline k)^2}\sigma^2_{1}$.
    
    If $t_{i},\ldots,t_{i+m_0-1}$ contains a corner-switch, then by Lemma~\ref{lem:consecutivethick}, one can show that 
    $$|t_{i}|+\ldots+|t_{i+m_0-1}|\ge d(T_0).$$
    Lemma~\ref{lem:widthvssys} and Lemma~\ref{lem:bounddelaunayedge} imply
    $$
    d(T_0)\ge \frac{relsys(X^{(0)})^2}{\sqrt{\frac{4}{\pi}+\frac{1}{2\pi\delta(\underline{k})}}}\ge \frac{\lambda^2_{d}}{\sqrt{\frac{4}{\pi}+\frac{1}{2\pi\delta(\underline{k})}}} = 
    \frac{\big(1+\frac{n}{\delta(\underline{k})}\big)^2\sigma_d^2}{\sqrt{\frac{4}{\pi}+\frac{1}{2\pi\delta(\underline{k})}}}
    $$
    According to Lemma~\ref{lem:upperboundcurvaturegap}, we have $\delta(\underline{k})\le1/3$ and then 
    $$
    \frac{4}{\pi}+\frac{1}{2\pi\delta(\underline{k})}<\frac{2}{\pi\delta(\underline{k})} + \frac{1}{2\pi\delta(\underline{k})} = \frac{5}{2\pi\delta(\underline{k})}<\frac{1}{\delta(\underline{k})}.
    $$
    It follows that
    $$
    d(T_0)\ge 
    \frac{\big(1+\frac{n}{\delta(\underline{k})}\big)^2\sigma_d^2}{\sqrt{\frac{4}{\pi}+\frac{1}{2\pi\delta(\underline{k})}}}>\sqrt{\delta(\underline{k})}\left(1+\frac{n}{\delta(\underline{k})}\right)^2\sigma_d^2\ge\sqrt{\delta(\underline{k})}\frac{n^2}{\delta(\underline{k})^2}\sigma_1^2 = \delta(\underline{k})^{-3/2}n^2\sigma_1^2\ge3\sqrt{3}n^2\sigma_1^2,
    $$
    where the last inequality uses Lemma~\ref{lem:upperboundcurvaturegap}.
    Then we obtain the estimate~\eqref{equ:consecutivethin}.
\end{proof}

Then we obtain the following estimate on thin parts.

\begin{prop}\label{prop:estimatethin_modified}
    Let $\mathbb{P}\Omega(\underline k)$ be a moduli space of convex flat cone spheres with $n$ singularities. Assume that the curvature gap $\delta(\underline k)$ is bounded below by a positive number $\delta$. Let $S$ be a codimensional $d$ boundary stratum. Let $X$ be a unit area flat cone sphere with $[X]\in U_{d,\sigma_d}(S_{\lambda_d})$. Then for any trajectory $\gamma$ in $X$, we have that
    $$
    |\gamma|\ge \sigma_1^2\Big(\frac{\sqrt{2}\delta}{4}\sqrt{\iota(\gamma,\gamma)} - 9n^2\Big),
    $$
    and for any regular closed geodesic $\gamma$ in $X$, we have that
    $$
    |\gamma| \ge \frac{\sqrt{2}\delta\sigma_1^2}{4} \sqrt{\iota(\gamma,\gamma)}.
    $$
\end{prop}

\begin{proof}
    Let $D_1,\ldots,D_s$ be the $(d, \epsilon)$-convex hulls in $X$.
    Let $X^{(0)}$ be the top surface obtained by Thurston surgeries along 
    $D_1,\ldots,D_s$.
    Let $t_1,\ldots,t_m$ be a decomposition of $\gamma$ associated to $U_{d,\sigma_d}(S_{\lambda_d})$ and $T_0$.
    Consider the integer 
    $$m_0 = 4 + \left\lfloor\frac{9}{\delta(\underline k)}n^2\right\rfloor$$
    as in Lemma~\ref{lem:consecutivethin_modified}.
    \par
    If $m\ge 2m_0$, Lemma~\ref{lem:consecutivethin_modified} implies that 
    \begin{align*}
        |\gamma| &\ge \sum\limits_{i = 1}^{\lfloor m/m_0 \rfloor}(|t_{(i-1)\cdot m_0+1}| + \ldots + |t_{(i-1)\cdot m_0+m_0}|)\\
        & \ge \left\lfloor \frac{m}{m_0} \right\rfloor 3\sqrt{3}n^2\sigma^2_{1} \ge \frac{m}{2m_0} 3\sqrt{3}n^2\sigma_1^2 = \frac{3\sqrt{3}n^2}{2m_0}\sigma_1^2\cdot m 
    \end{align*}
    Using $\delta\le 1/3$ and $n\ge 3$, we obtain that
    $$
    \frac{3\sqrt{3}n^2}{2m_0}\ge \frac{3\sqrt{3}n^2}{8 + \frac{18}{\delta}n^2} = \frac{3\sqrt{3}\cdot\delta}{8\cdot\delta/n^2 + 18}\ge\delta\cdot\frac{3\sqrt{3}}{(8\cdot1/3)/3^2 + 18}>\delta\cdot\frac{1}{4}
    $$
    Then 
    \begin{equation}
        |\gamma| \ge \frac{3\sqrt{3}n^2}{2m_0}\sigma_1^2\cdot m\ge \frac{\delta\sigma_1^2}{4}\cdot m
    \end{equation}
    As in Lemma~\ref{lem:combinvsintersect}, we know that $m\ge \sqrt{2\iota(\gamma,\gamma)}$. Hence when $m\ge 2m_0$, we obtain that
    \begin{equation}\label{equ:3}
        |\gamma|\ge  \frac{\sqrt{2}\delta\sigma_1^2}{4}
        \sqrt{\iota(\gamma,\gamma)}.
    \end{equation}
    \par
    If $m < 2m_0$, we have that 
    \begin{align*}
        |\gamma| > 0 &\ge 9n^2\sigma_1^2\big(\frac{m}{2m_0} - 1\big) \\
        &\ge \sigma_1^2\Big(\frac{9n^2}{2m_0}m - 9n^2\Big)\\
        &\ge \sigma_1^2\Big(\frac{\sqrt{2}\delta}{4}\sqrt{\iota(\gamma,\gamma)} - 9n^2\Big)
    \end{align*}
    Therefore, combining with~\eqref{equ:3}, we obtain that 
    $$
    |\gamma|\ge \sigma_1^2\Big(\frac{\sqrt{2}\delta}{4}\sqrt{\iota(\gamma,\gamma)} - 9n^2\Big)
    $$
    for any trajectory in $X$.
    \par
    We further assume that $\gamma$ is a regular closed geodesic. Proposition~\ref{prop:lowerboundclosed} implies that $|\gamma|\ge\sqrt{\pi\delta}$. If $m<2m_0$, since $m\ge \sqrt{2\iota(\gamma,\gamma)}$, we have that 
    $$\sqrt{\iota(\gamma,\gamma)}\le \frac{1}{\sqrt{2}}m<\sqrt{2}m_0 \le \sqrt{2}\big(4 +\frac{9}{\delta}n^2\big).$$
    Hence, when $m<2m_0$, it follows that
    \begin{align*}
        |\gamma|\ge \sqrt{\pi\delta} \ge \frac{\sqrt{\pi\delta}}{\sqrt{2}\big(4 +\frac{9}{\delta}n^2\big)}\sqrt{\iota(\gamma,\gamma)}.
    \end{align*}
    Combining with the inequality~\eqref{equ:3}, we obtain that
    \begin{align*}
        |\gamma| \ge \min\left\{\frac{\sqrt{2}\delta\sigma_1^2}{4},  \frac{\sqrt{\pi\delta}}{\sqrt{2}\big(4 +\frac{9}{\delta}n^2\big)} \right\}\sqrt{\iota(\gamma,\gamma)} = \frac{\sqrt{2}\delta\sigma_1^2}{4} \sqrt{\iota(\gamma,\gamma)}
    \end{align*}
    where the last equality is computed by the formula~\eqref{equ:defepsilon} of $\sigma_1$.
\end{proof}

\subsection{Proof of Theorem~\ref{thm:mainfavorite}}\label{sec:main}

\begin{proof}[Proof of Theorem~\ref{thm:mainfavorite}]
    Let $X$ be a unit area convex flat cone sphere with $n$ singularities and let $\underline k$ be the curvature vector of $X$. Assume that the curvature gap of $X$ satisfies that $\delta(\underline k)\ge\delta>0$. Notice that $[X]$ is contained in  $\mathbb{P}\Omega(\underline k)$. Corollary~\ref{cor:thindecomp} implies that 
    $\mathbb{P}\Omega(\underline k)$ is covered by the thick part 
    $\mathbb{P}\Omega(\underline k)\setminus U_{1,\sigma_1}$ and the thin parts $U_{d,\sigma_d}(S_{\lambda_d})$ for each boundary stratum $S$. By Proposition~\ref{prop:estimatethick} and Proposition~\ref{prop:estimatethin_modified}, for any trajectory $\gamma$ in $X$, we have that
    $$
    |\gamma| \ge a_1(\underline k)\sqrt{\iota(\gamma,\gamma)} - a_2(\underline k),
    $$ 
    and for any regular closed geodesic $\gamma$, we have that
    $$
    |\gamma| \ge a_1(\underline k)\sqrt{\iota(\gamma,\gamma)},
    $$ 
    where
    \begin{align*}
        & a_1(\underline{k})  
    = \sqrt{2}\cdot\delta\cdot\sigma_1(\underline{k})^2\cdot
    \min\left\{ \frac{1}{\sqrt{\frac{4}{\pi}+\frac{1}{2\pi\delta}}}\frac{1}{36(n-2)}, \frac{1}{4}\right\}\\
        & a_2(\underline{k}) = \sigma_1(\underline{k})^2\cdot\max\left\{\frac{1}{\sqrt{\frac{4}{\pi}+\frac{1}{2\pi\delta}}} ,9n^2\right\}.
    \end{align*}

    Since $\frac{1}{\sqrt{\frac{4}{\pi}+\frac{1}{2\pi\delta}}}\le
    \frac{1}{\sqrt{\frac{4}{\pi}}} <1$,
    we deduce that
    $$
    a_1(\underline{k}) = \sqrt{2}\cdot\delta\cdot\sigma_1(\underline{k})^2\cdot \frac{1}{\sqrt{\frac{4}{\pi}+\frac{1}{2\pi\delta}}}\frac{1}{36(n-2)}
    $$
    and 
    \begin{equation}\label{equ:estimatefora2}
        a_2(\underline{k}) = \sigma_1(\underline{k})^2\cdot9n^2.
    \end{equation}

    By Lemma~\ref{lem:upperboundcurvaturegap}, we know that
    \begin{equation}\label{equ:useboundofthecurvaturegaptoestimate}
        \delta(\underline{k})\le\frac{1}{3}.
    \end{equation}
    It follows that $\frac{1}{\delta}\ge \frac{1}{\delta(\underline{k})}\ge3>2$ and
    $$
    \frac{1}{\sqrt{\frac{4}{\pi}+\frac{1}{2\pi\delta}}}\ge\frac{1}{\sqrt{2+\frac{1}{2\pi\delta}}}>\frac{1}{\sqrt{\frac{1}{\delta}+\frac{1}{2\pi\delta}}} > \frac{1}{\sqrt{\frac{1}{\delta}+\frac{1}{\delta}}} = \frac{\sqrt{\delta}}{\sqrt{2}}. 
    $$
    Then,
\begin{equation}\label{equ:estimatefora1}
        a_1(\underline{k}) \ge \sqrt{2}\cdot\delta\cdot\sigma_1(\underline{k})^2\cdot\frac{\sqrt{\delta}}{\sqrt{2}}\cdot\frac{1}{36(n-2)} = \sigma_1(\underline{k})^2\cdot\frac{\delta^{3/2}}{36(n-2)}.
\end{equation}
    
    Recall that 
    $$\sigma_1(\underline{k}) = \frac{\delta(\underline{k})^2}{4n^2} \left(\frac{\delta(\underline{k})^2}{6n}\right)^{n-3}.$$
    Using inequality~\eqref{equ:useboundofthecurvaturegaptoestimate}, we obtain
    $$\frac{3}{2n}\left(\frac{\delta^2}{6n}\right)^{n-2}\le \sigma_1(\underline k)\le 81\left(\frac{1}{54n}\right)^{n-1}.$$ 
    Substituting these bounds into~\eqref{equ:estimatefora1} and~\eqref{equ:estimatefora2}, we obtain
    \begin{align*}
        &a_1(\underline{k})\ge \sigma_1(\underline{k})^2\cdot\frac{\delta^{3/2}}{36(n-2)} \ge \frac{9\delta^{3/2}}{144n^2(n-2)}\left(\frac{\delta^2}{6n}\right)^{2n-4}\\
        &a_2(\underline{k}) = \sigma_1(\underline{k})^2\cdot9n^2\le9n^2\cdot81^2\cdot\left(\frac{1}{54n}\right)^{2n-2} = \frac{81}{4}\left(\frac{1}{54n}\right)^{2n-4}.
    \end{align*}

    Define
    \begin{align*}
        c_1 :=  \frac{9\delta^{3/2}}{144n^2(n-2)}\left(\frac{\delta^2}{6n}\right)^{2n-4}
        \quad \mbox{ and }\quad
        c_2 := \frac{81}{4} \left(\frac{1}{54n}\right)^{2n-4}.
    \end{align*}
    Note that $c_1$ depends only on $n$ and $\delta$ and $c_2$ depends only on $n$. It follows that for any trajectory $\gamma$ in $X$, we have that
    $$
    |\gamma| \ge c_1\sqrt{\iota(\gamma,\gamma)} - c_2,
    $$ 
    and for any regular closed geodesic $\gamma$, we have that
    $$
    |\gamma| \ge c_1\sqrt{\iota(\gamma,\gamma)}.
    $$ 
\end{proof}

\subsection{Applications of Theorem~\ref{thm:mainfavorite}}\label{sec:appli}

We first apply Theorem~\ref{thm:mainfavorite} to counting problems on convex flat cone spheres. Let $X$ be a flat cone sphere. 

Recall that a regular closed geodesic in $X$ is a geodesic that returns to its starting point with the same tangent direction. Two regular closed geodesics are said to be \emph{parallel} if they bound an immersed flat cylinder with no singularities inside. A family of parallel regular closed geodesics is called \emph{maximal} if it contains all regular closed geodesics parallel to a given regular closed geodesic.

Given a positive real number $R$, we define $N^{sc}(X,R)$ and $N^{cg}(X,R)$ as follows:
\begin{itemize}
    \item $N^{sc}(X,R)$ is the number of saddle connections on $X$ with length at most $R$.
    \item $N^{cg}(X,R)$ is the number of maximal families of parallel regular closed geodesics with lengths at most $R$.
\end{itemize}

\begin{proof}[Proof of Corollary~\ref{cor:integrability}]
    For a non-negative integer $s$, we define $N_1^{sc}(X, s)$ to be the number of saddle connections on $X$ with number of self-intersection at most $s$. By Theorem~\ref{thm:mainfavorite}, a saddle connection $\gamma$ of length $\le R$ satisfies that
    $$
    \iota(\gamma,\gamma)\le c_1^{-2}(|\gamma|+c_2)^2 \le c_1^{-2}(R+c_2)^2.
    $$
    It follows    that 
    $$N^{sc}(X,R)\le N_1^{sc}(X, c_1^{-2}(R+c_2)^2)$$
    where the constant $c_1$ and $c_2$ are as in Theorem~\ref{thm:mainfavorite} and depend only on the number of singularities and  the curvature gap of $X$. According to Theorem~\ref{thm:finite}, we have that $N^{sc}(X,R)$ is bounded above by $$
    (3n-6)2^{20n\delta^{-1}\big(c_1^{-1}(n-1)(R+c_2) + 1\big)}.
    $$ 
    Notice that every saddle connection is in the boundary of at most two maximal families of parallel regular closed geodesics. Moreover, since a closed flat cylinder has two boundary components, we can find two saddle connections on the boundary of a maximal family of parallel regular closed geodesics, or a saddle connection on the boundary of the family such that regular closed geodesics are on both sides of the saddle connection. Hence, it follows that 
    $$
    N^{cg}(X,R)\le N^{sc}(X,R).
    $$
\end{proof}

Next, we apply Theorem~\ref{thm:mainfavorite} to counting problems on polygonal billiards.

Let $P$ be a polygon in the plane. There is a classical construction from~\cite{FoxKer} that associates a flat cone sphere to $P$. Given a polygon $P$, let $P_1$ and $P_2$ be two copies of $P$, and identify the corresponding edges of $P_1$ and $P_2$ via Euclidean isometries. This construction results in a flat cone sphere, which we denote by $X_P$.

Given a billiard path $\{s_1, \ldots, s_m\}$ in $P$, we now describe how the billiard path induces a trajectory in $X_P$ (following~\cite{Zor}, Section~2.1). We start the billiard path from $s_1$ in one copy of $P$, say $P_1$. Each time $s_i$ hits the edge of a copy of $P$ in $X_P$, the trajectory continues as $s_{i+1}$ in the other copy of $P$. In this way, we obtain a trajectory in $X_P$, denoted by $\gamma$. The length of $\gamma$ is the same as the billiard path.

From this construction, we know that $\gamma$ is a saddle connection in $X_P$ if and only if the billiard path is a generalized diagonal in $P$. When $m$ is even, $\gamma$ is a regular closed geodesic in $X_P$ if and only if the billiard path is periodic in $P$. However, when $m$ is odd, the trajectory $\gamma$ ends in the same copy, $P_1$, where it started. In this case, if the billiard path is periodic, the induced trajectory $\gamma$ is not a regular closed geodesic in $X_P$. To obtain a regular closed geodesic, we continue $\gamma$ by repeating the operation starting from $s_1$ in $P_2$. As a result, we obtain a regular closed geodesic whose length is twice that of the billiard path.
From now on, we refer to this regular closed geodesic as $\gamma$ when $m$ is odd and the billiard path is periodic.

We say that two periodic billiard paths in $P$ are \emph{parallel} if they induce parallel regular closed geodesics in $X_P$. A family of parallel periodic billiard paths is called \emph{maximal} if it induces a maximal family of regular closed geodesics in $X_P$.

Let $R$ be a positive real number. We define $N^{diag}(P,R)$ and $N^{per}(P,R)$ as follows:
\begin{itemize}
    \item $N^{diag}(P,R)$ is the number of generalized diagonals in $P$ with length at most $R$.
    \item $N^{per}(P,R)$ is the number of maximal families of parallel periodic billiard paths with length at most $R$.
\end{itemize}

\begin{proof}[Proof of Corollary~\ref{cor:billiard}]
    Let $X_P$ be the flat cone sphere associated to $P$ which is obtained by the above construction. Since $P$ is of area one, the area of $X_P$ is two.
    We know that a general diagonal in $P$ induces a saddle connection of the same length in $X_P$. 
    It follows that
    $$
    N^{diag}(P,R)\le
    N^{sc}(X_P,\frac{R}{\sqrt{2}})\le
    (3n-6)2^{20n\delta^{-1}\big(c_1^{-1}(n-1)(\frac{1}{\sqrt{2}}R+c_2) + 1\big)}
    $$
    where the constant $c_1$ and $c_2$ are as in Theorem~\ref{thm:mainfavorite}. 
    We also know that  a periodic billiard path in $P$ induces a regular closed geodesic of at most twice the length in $X_P$. It follows that
    $
    N^{per}(P,R)\le N^{cg}(X_P,\frac{2R}{\sqrt{2}})\le (3n-6)2^{20n\delta^{-1}\big(c_1^{-1}(n-1)(\sqrt{2}R+c_2) + 1\big)}.
    $
    Hence, we proved the results.
\end{proof}

\subsection{Examples}
In this subsection, we present several examples to explore the dependence of the constants $c_1$ and $c_2$ in Theorem~\ref{thm:mainfavorite} on the conditions stated in the theorem.

We begin by noting that there is a counterexample to the validity of Theorem~\ref{thm:mainfavorite} when negative curvatures are permitted.

\begin{Ex}\label{ex:negativecurvature}
    We construct a sequence of flat cone spheres with negative curvatures and positive curvature gaps such that there is a regular closed geodesic on each of them of normalized length tending to zero.

    For each $0 < t < 1$, consider an equilateral triangle $P_t$ with side length $t$ and vertices labeled $x_1, x_2, x_3$; see the left picture of Figure~\ref{fig:example11}.  
    There exists a special periodic billiard path $\gamma_t$ in $P_t$, called the \emph{Fagnano path}, which reflects at the midpoint of each edge. Consequently, its length is $|\gamma_t| = \tfrac{3}{2}t$.

    Next, consider a quadrilateral $Q$ with vertices labeled $y_1, y_2, y_3, y_4$ and cone angles $\tfrac{\pi}{2}$, $\tfrac{9\pi}{20}$, $\tfrac{3\pi}{5}$, and $\tfrac{9\pi}{20}$, respectively, as shown in the right picture of Figure~\ref{fig:example11}.

    \begin{figure}[!htbp]
    \centering
    \includegraphics[width=0.8\linewidth]{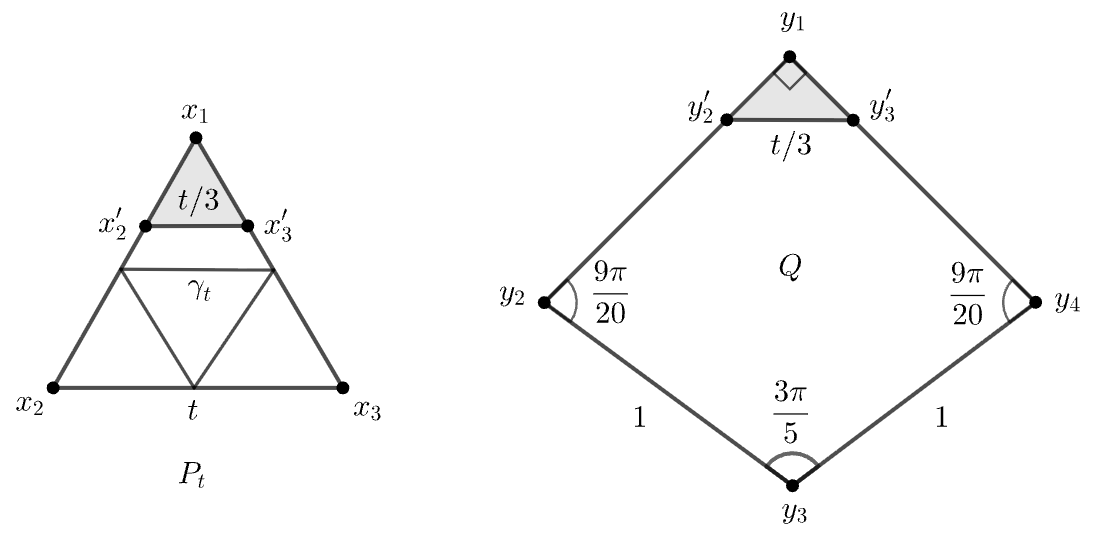}
    \caption{$P_t$: an equilateral triangle of side length $t$ with a gray subtriangle $P_{t/3}$ at $x_1$; $Q$: a quadrilateral with the indicated angles and side lengths, containing a gray isosceles right triangle of hypotenuse length $t/3$.}\label{fig:example11}
    \end{figure}

    As shown in Figure~\ref{fig:example11}, we mark a smaller gray equilateral triangle $x_1x_2'x_3'$ at the vertex $x_1$ of $P_t$, with side length $t/3$, and an isosceles right triangle $y_1y'_2y'_3$ at the vertex $y_1$ of $Q$, with hypotenuse length $t/3$.  
    We then cut out these gray triangles and glue the edge $x'_2x'_3$ to the edge $y'_2y'_3$, as illustrated in Figure~\ref{fig:example12}. Denote the resulting non-convex polygon by $P'_t$. Let $X_t$ be the flat cone sphere by the double construction associate to $P'_t$ (see Section~\ref{sec:appli}).

    Since the cone angles of $X_t$ are twice the interior angles of $P'_t$, the curvatures of $X_t$ at $x_2, x_3, x'_3, y_2, y_3, y_4, x'_2$ are
    \[
    \frac{2}{3},\ \frac{2}{3},\ -\frac{5}{12},\ \frac{11}{20},\ \frac{2}{5},\ \frac{11}{20},\ -\frac{5}{12},
    \]
    respectively.
    Writing them with denominator~$60$ gives
    \[
    \frac{40}{60},\ \frac{40}{60},\ -\frac{25}{60},\ \frac{33}{60},\ \frac{24}{60},\ \frac{33}{60},\ -\frac{25}{60}.
    \]
    A direct check over all subset sums shows that the closest value to~$1$ is $57/60=19/20$, attained for the subset $\{2/5,\,11/20\}$. 
    Hence,
    \[
    \delta=\inf_{I\subseteq\{1,\ldots,7\}}\left|1-\sum_{i\in I}k_i\right|
    =\frac{1}{20}>0.
    \]
    
    \begin{figure}[!htbp]
    	\centering
        	\includegraphics[width=0.7\linewidth]{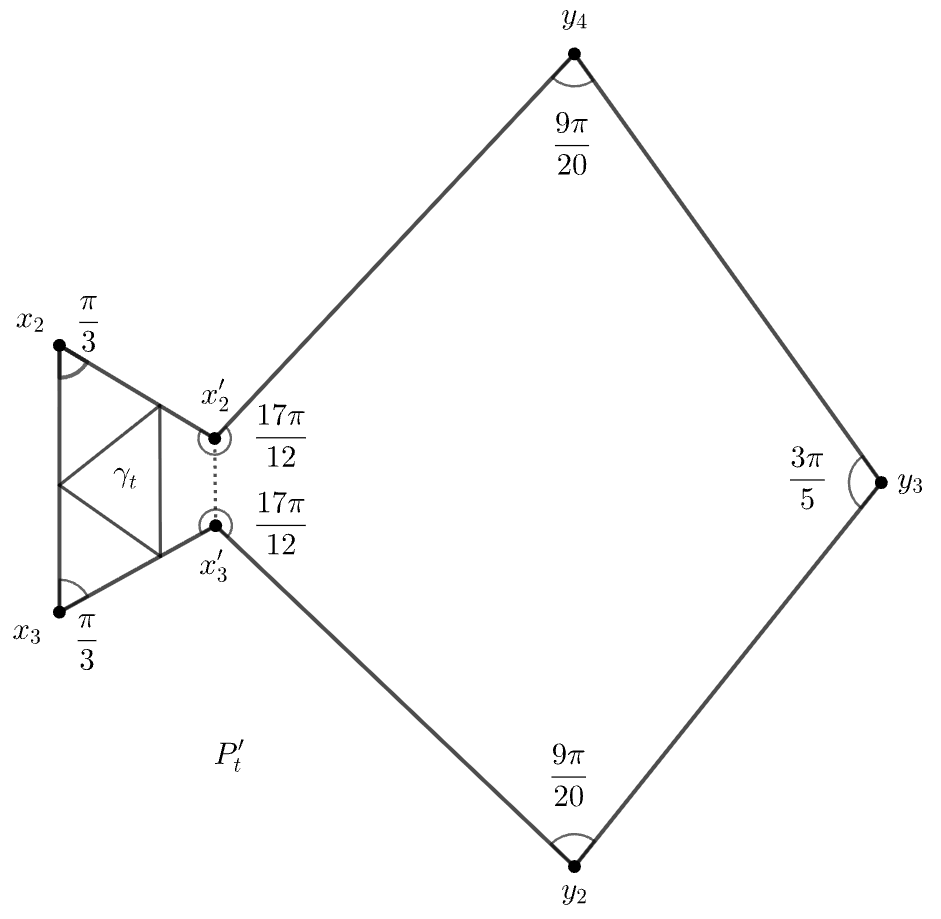}
        \caption{$P'_t$: a non-convex polygon obtained by cutting the gray triangles from $P_t$ and $Q$, and then gluing along $x'_2x'_3$ and $y'_2y'_3$.}\label{fig:example12}
    \end{figure}
    
    Note that the Fagnano path $\gamma_t$ induces a regular closed geodesic on $X_t$, denoted by $\gamma'_t$. The length of $\gamma'_t$ is twice the length of $\gamma_t$, so it is $3t$. The area of $X_t$ converges to $2\cdot\operatorname{Area}(Q)>0$ as $t\to 0$. Hence, the normalized length of $\gamma_t$ converges to zero as $t\to 0$. 
\end{Ex}

Next, we present the following example to illustrate that the inequality~\eqref{equ:compare2} cannot hold for every saddle connection $\gamma$.

\begin{Ex}\label{ex:c2notzero}
    We construct a sequence of convex flat cone spheres with positive curvature gaps such that there is a saddle connection on each of them whose self-intersection number is one but normalized length tends to zero.
    \par
    Let $P_1$ be a equilateral triangle with side length one and let $P_2$ be an isosceles triangle with angles $\frac{\pi}{6}, \frac{5\pi}{12}, \frac{5\pi}{12}$ and the base side length one. Denote the vertices of $P_1$ by $x_1,x_2,x'_2$ and denote the vertices of $P_2$ by $y_1,y_2,y'_2$.
    For any real number $0<t<1$, we denoted the polygon by $(P_1)_t$ which is obtained by scaling $P_1$ by $t$. We cut out a small isosceles triangle in $P_2$ with that apex at $y_1$ and base length $t$. Denote the other two vertices in the base by $\widetilde x_2$ and $\widetilde{x}'_2$.

    \begin{figure}[!htbp]
    	\centering
    	\begin{minipage}{0.8\linewidth}
        	\includegraphics[width=\linewidth]{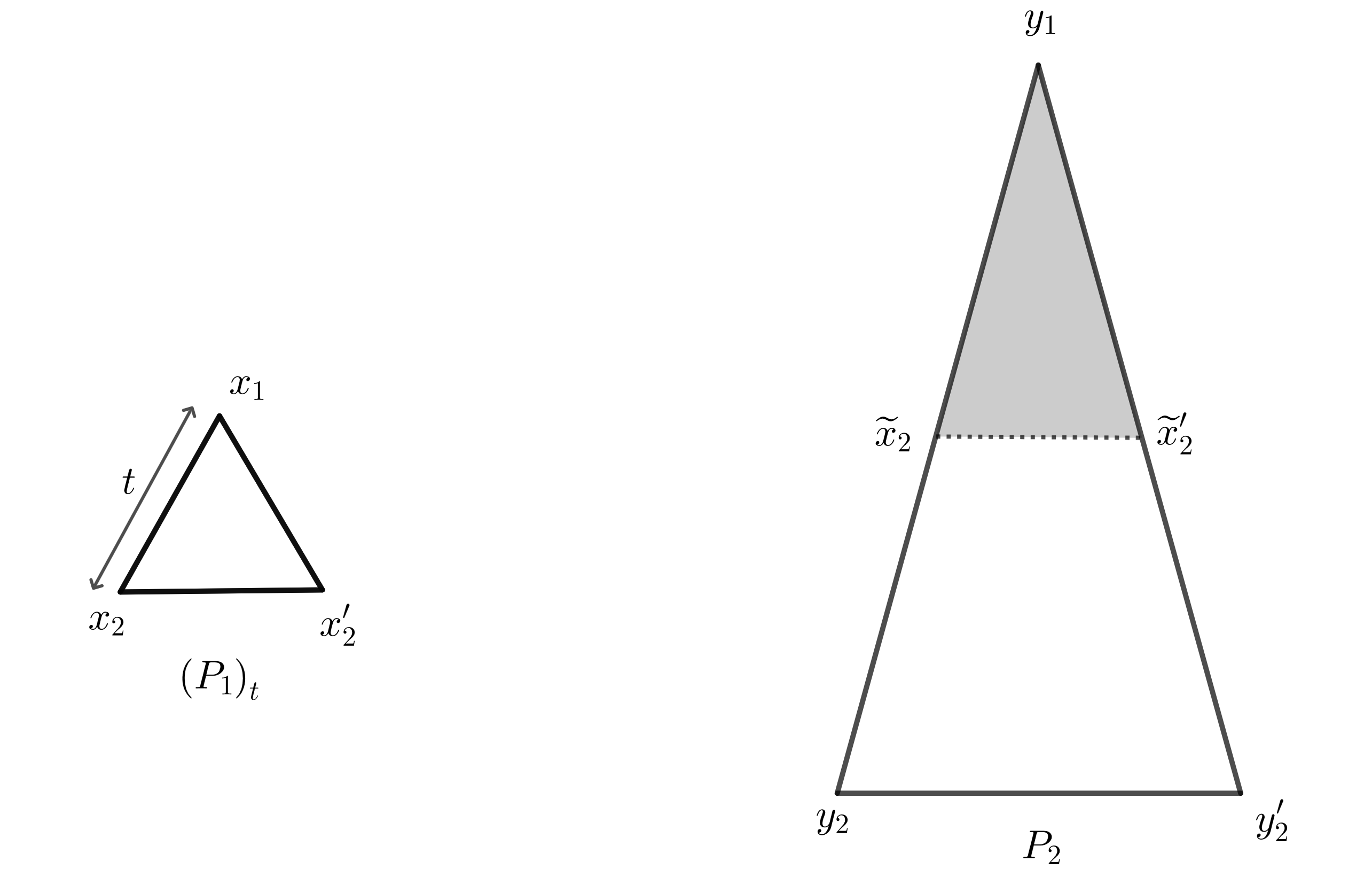}
        	\footnotesize
    	\end{minipage}
        \caption{We cut out the gray region in the isosceles triangle on the right and attach a scaled equilateral triangle $(P_1)_t$ by identifying $x_2x_2'$ and $\widetilde{x}_2\widetilde{x}'_2$.}\label{fig:example21}
    \end{figure}

    Next we glue $(P_1)_t$ to the truncated triangle of $P_2$ by an isometry from the scaled edge of $x_2x'_2$ to the edge $\widetilde x_2\widetilde{x}'_2$. We denote the resulting convex polygon by $P$. We identify the edges of $P$ as shown in Figure~\ref{fig:example22}, where $y_3$ is the midpoint of the side $y_2y'_2$. We denote the obtained flat cone sphere by $X_t$. From the figure, we know that $x_1$, $x_2$, $y_2$ and $y_3$ induce four singularities in $X_t$ of curvatures $\frac{5}{6},\frac{1}{12},\frac{7}{12},\frac{1}{2}$ respectively. It follows that $X_t$ has the same positive curvature gap for all $t$.

    \begin{figure}[!htbp]
    	\centering
    	\begin{minipage}{0.6\linewidth}
        	\includegraphics[width=\linewidth]{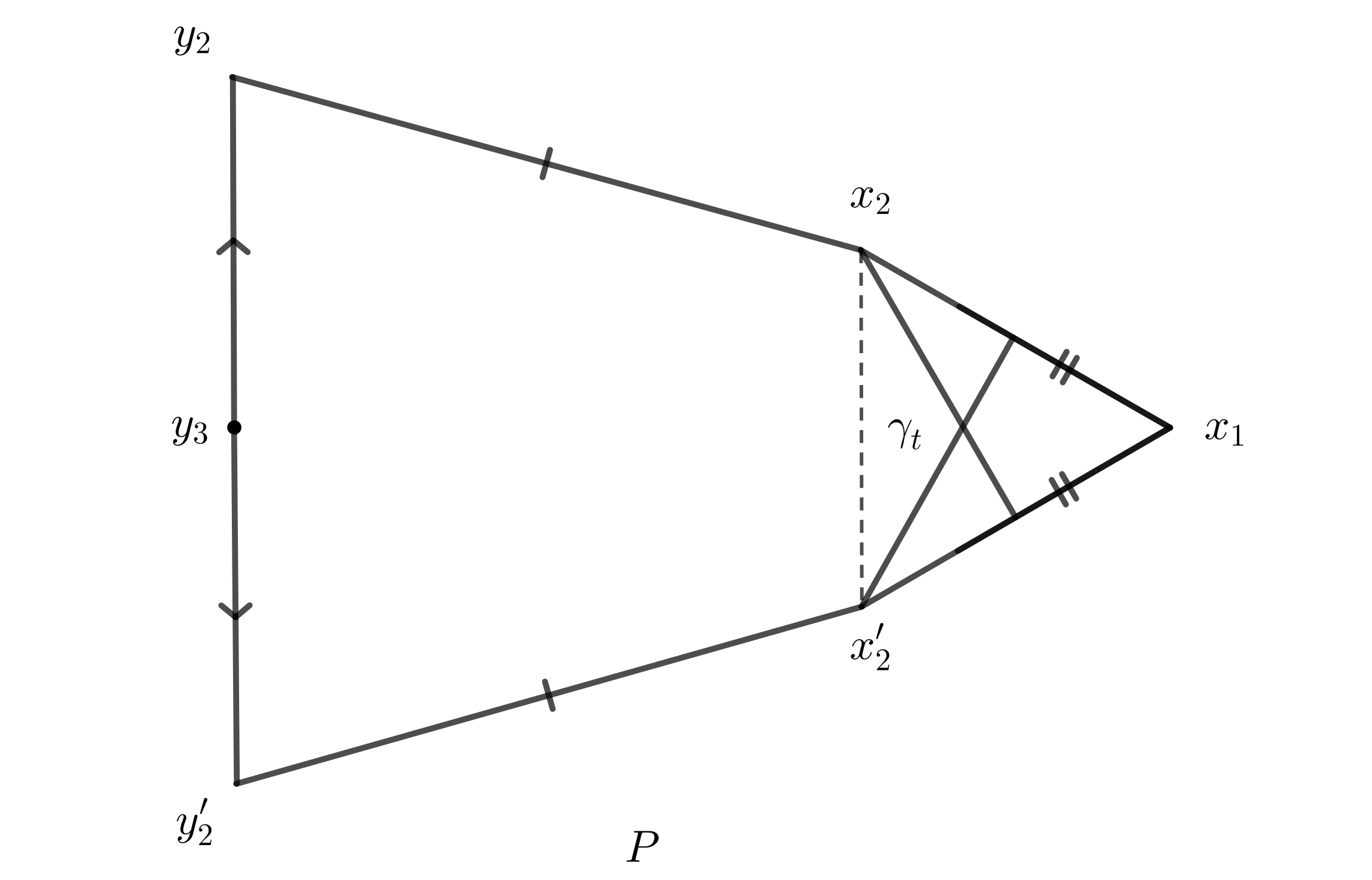}
        	\footnotesize
    	\end{minipage}
        \caption{Under the pairing of the edges of $P$, two heights of the equilateral triangle form a saddle connection on the resulting surface.}\label{fig:example22}
    \end{figure}

    Under the identification in $P$, two heights of the scaled triangle $(P_1)_t$ on the sides $x_1x_2$ and $x_1x'_2$ form a saddle connection in $X_t$, denoted 
    by $\gamma_t$. Notice that $\iota(\gamma_t,\gamma_t) = 1$ and the length of $\gamma_t$ is $\sqrt{3}t$. It can be directly computed that the normalized length of $\gamma_t$ in $X_t$ goes to zero as $t\to 0$. Therefore, there is no positive constant $c$ depending only on the number of singularities and the curvature gap such that $c = c\sqrt{\iota(\gamma_t,\gamma_t)}\le \ell_{\gamma_t}(X_t)$ for all sufficiently small $t$.
\end{Ex}

In the final example, we emphasize that the dependence of $c_1$ and $c_2$ on the curvature gap $\delta$ is necessary.

\begin{Ex}\label{ex:deltaisnecessary}
    We first construct a family of convex flat cone spheres with curvature gaps going to zero. We consider an isosceles trapezoid with base angles $\frac{\pi+\theta}{2}$ and $\frac{\pi-\theta}{2}$ for $0<\theta<\frac{\pi}{6}$. We attach an isosceles triangle with apex angle $\frac{2\pi}{3}$ to the trapezoid by identifying the base of the triangle and the longer base of the trapezoid.
    Denote the resulting polygon by $P_{\theta}$.
    We put a vertex in the shorter base of the trapezoid. The vertices of $P_{\theta}$ are labeled and the sides of $P_{\theta}$ are paired as shown in Figure~\ref{fig:example31}. Hence we obtain a flat cone sphere $X_{\theta}$. The vertices $x_1$, $x_2$, $x_3$, $x_4$ induce four singularities in $X_{\theta}$ with cone angles $\pi+\theta$, $\frac{4\pi}{3}-\theta$, $\frac{2\pi}{3}$, $\pi$ respectively. Since $\theta<\frac{\pi}{6}$, the curvature gap of $X_{\theta}$ is $\frac{\theta}{2\pi}$ which goes to zero as $\theta\to 0$.

    \begin{figure}[!htbp]
    	\centering
    	\begin{minipage}{0.8\linewidth}
        	\includegraphics[width=\linewidth]{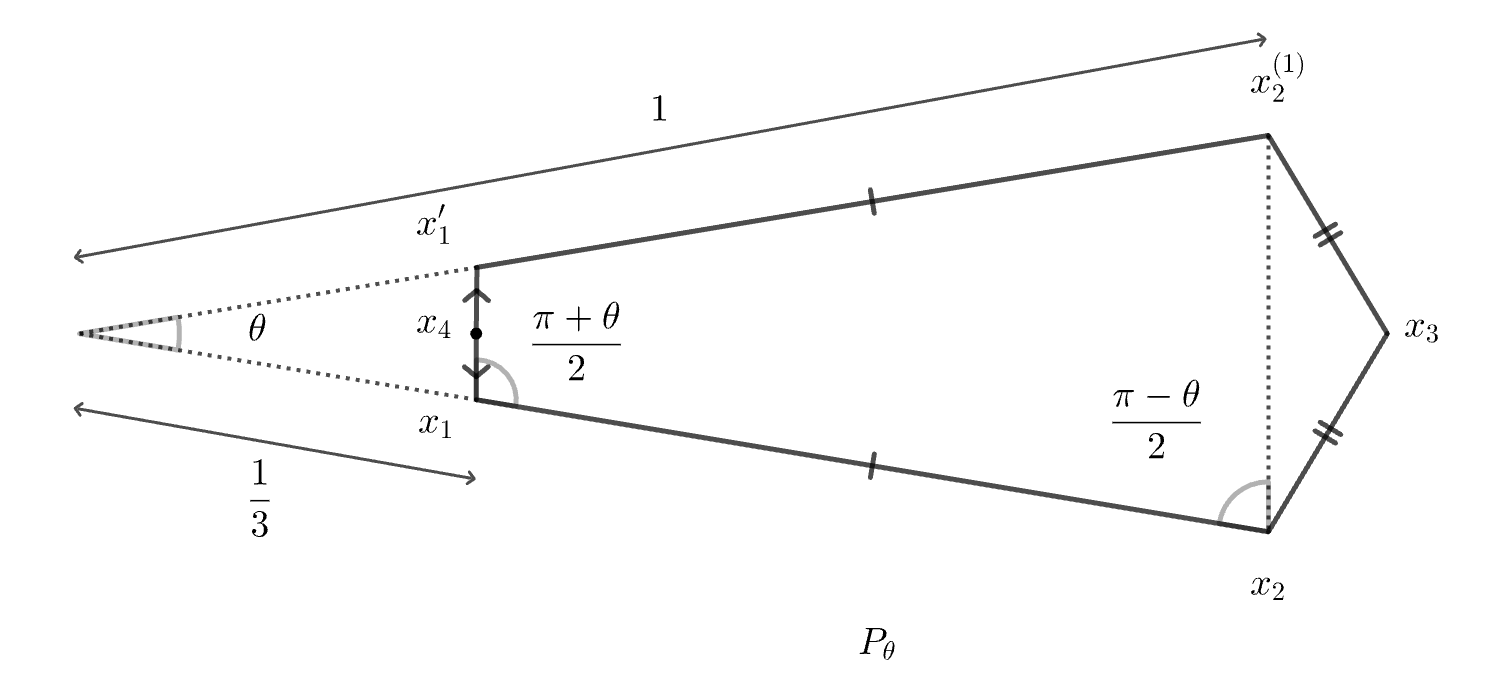}
                \caption{We glue an isosceles triangle to a trapezoid along a base. The pairing of the edges of the polygon is marked in the picture.}\label{fig:example31}
        	\footnotesize
    	\end{minipage}
    \end{figure}
    
    Next we show that there are no positive constants $c_1$ and $c_2$ such that for any $\theta\in(0,\frac{\pi}{6})$ and any trajectory $\gamma$ in $X_{\theta}$, one has that $c_1\sqrt{\iota(\gamma,\gamma)} - c_2 \le \ell_{\gamma}(X_{\theta})$.
    \par
    We develop $X_{\theta}$ on the plane as $P_{\theta}$ and consider the copies of $P_{\theta}$ as in Figure~\ref{fig:example32}. To simplify, we only draw the copies of the trapezoid. Denote the copies of $x_2$ in the longer base by $x_2^{(1)},x_2^{(2)},\ldots$. We join the vertices $x_2$ and $x_2^{(m)}$ by a segment $\gamma_m$. If $\gamma_m$ is contained in the interior of  the copies, then it induces a saddle connection in $X_{\theta}$, denoted by $\gamma_m$ still. 
    
    By construction, the holonomy of a small loop around $x_1$ and $x_4$ is a rotation $r_\theta$ by angle $\theta$. Hence, for each $1 \le i < m$, the image $r_\theta^i(\gamma_m)$ projects to the same saddle connection on $X_\theta$ as $\gamma_m$, where $r_\theta^i$ denotes the $i$-fold composition of $r_\theta$. Since $r_\theta^i(\gamma_m)$ intersects $\gamma_m$ exactly once, we conclude that $\iota(\gamma_m,\gamma_m) = m-1$.
    
    \begin{figure}[!htbp]
    	\centering
    	\begin{minipage}{0.8\linewidth}
        	\includegraphics[width=\linewidth]{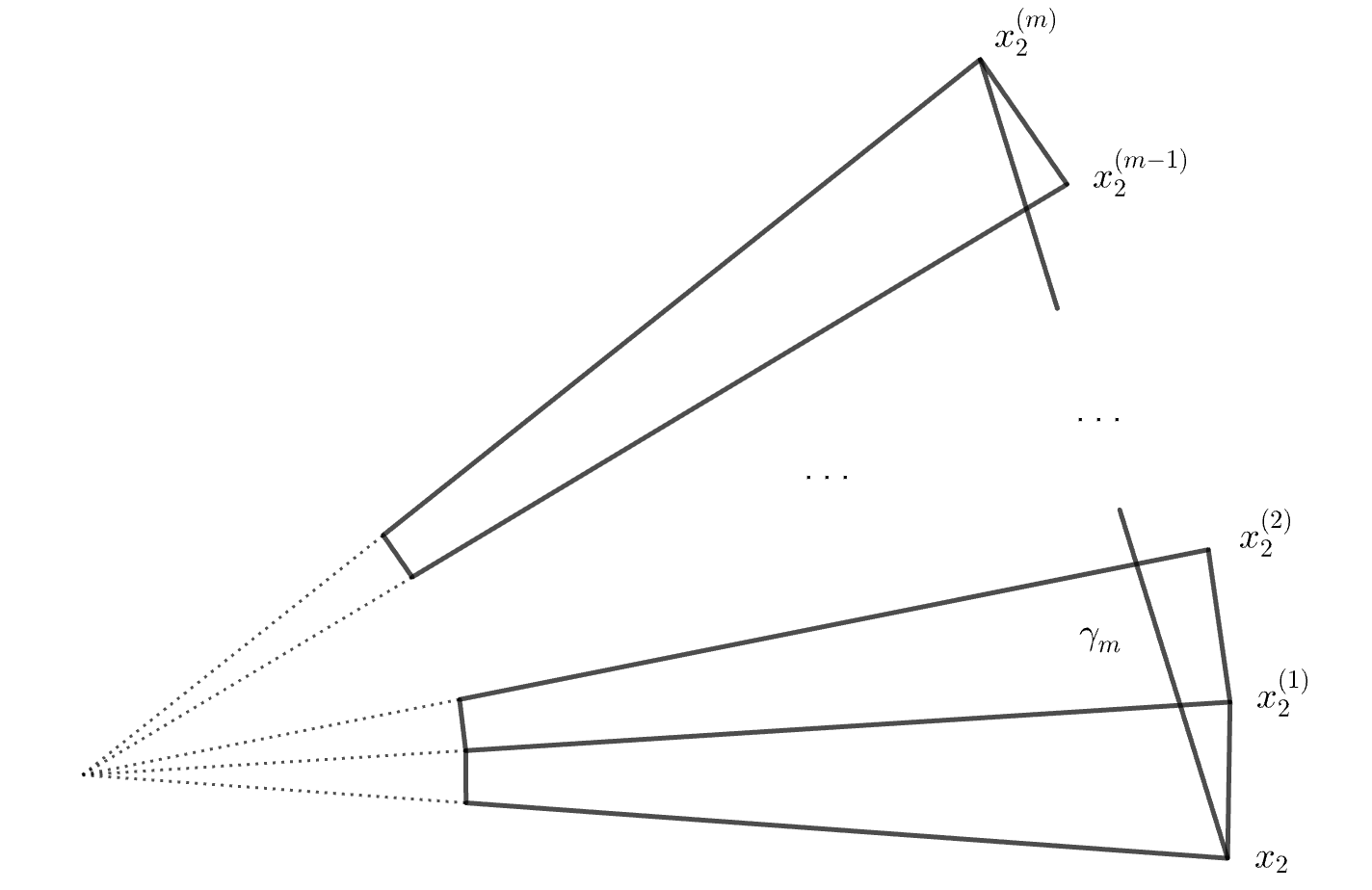}
                \caption{We unfold the trapezoid by rotations around the apex. Note that the segment $\gamma_m$ induces a saddle connection on the flat cone sphere $X_{\theta}$.}\label{fig:example32}
        	\footnotesize
    	\end{minipage}
    \end{figure}
    
    When $\theta$ is small enough, we can select $m$ such that the segment $\gamma_m$ induces a saddle connection in $X_{\theta}$ and $0<c_1(m-1)-c_2=c_1\sqrt{\iota(\gamma_m,\gamma_m)}-c_2$.  
    On the other hand, it can be computed that the length of $\gamma_m$ is $2\sin{\frac{m\theta}{2}}\sim m\theta$ as $\theta\to 0$, and the area of $X_{\theta}$ is $\frac{4}{9}\sin{\theta} + \frac{1}{\sqrt{3}}\sin^2\frac{\theta}{2}\sim \frac{4}{9}\theta$ as $\theta\to 0$. It follows that the normalzied length of $\gamma_m$ is asymptotically equivalent to $\frac{3m}{2}\sqrt{\theta}$ and hence goes to zero as $\theta\to 0$. Therefore, there is no positive constants $c_1$ and $c_2$ such that for any $\theta\in(0,\frac{\pi}{6})$, one has that $0< c_1\sqrt{\iota(\gamma_m,\gamma_m)} - c_2 \le \ell_{\gamma_m}(X_{\theta})$.
\end{Ex}

\bibliographystyle{plainurl}
\bibliography{ref}

@article{thu,
  author = {Thurston, William P.},
  title = {Shapes of polyhedra and triangulations of the sphere},
  journal = {Geometry \& Topology Monographs},
  year = {1998},
  volume = {1},
  pages = {511--549},
  doi = {10.2140/gtm.1998.1.511},
  url = {https://arxiv.org/abs/math/9801088}
}

@article{McMullen,
  title   = {The Gauss-Bonnet theorem for cone manifolds and volumes of moduli spaces},
  author  = {McMullen, Curtis T.},
  journal = {American Journal of Mathematics},
  year    = {2017},
  volume  = {139},
  pages   = {261--291},
  url     = {https://www.jstor.org/stable/24906289}
}

@article{Veech,
  author  = {Veech, William A.},
  title   = {Siegel Measures},
  journal = {Annals of Mathematics},
  volume  = {148},
  number  = {3},
  pages   = {895--944},
  year    = {1998},
  doi     = {10.2307/121033},
  url     = {https://doi.org/10.2307/121033}
}

@book{FM,
  author    = {Benson Farb and Dan Margalit},
  title     = {A Primer on Mapping Class Groups},
  publisher = {Princeton University Press},
  year      = {2012},
  isbn      = {9780691147949},
  url       = {https://press.princeton.edu/books/hardcover/9780691147949/a-primer-on-mapping-class-groups}
}

@article{KN,
  title   = {Complex hyperbolic volume and intersection of boundary divisors in moduli spaces of genus zero curves},
  author  = {Koziarz, Vincent and Nguyen, Duc-Manh},
  journal = {Annales scientifiques de l'École normale supérieure},
  volume  = {51},
  number  = {6},
  pages   = {1549--1597},
  year    = {2018},
  doi     = {10.24033/asens.2381},
  url     = {https://doi.org/10.24033/asens.2381}
}

@article{HKK,
  title={Flat surfaces and stability structures},
  author={Haiden, Fabian and Katzarkov, Ludmil and Kontsevich, Maxim},
  journal={Publications mathématiques de l'IHÉS},
  volume={126},
  pages={247--318},
  year={2017},
  doi={10.1007/s10240-017-0095-y},
  url={https://doi.org/10.1007/s10240-017-0095-y}
}

@article{DM,
  author = {Deligne, Pierre and Mostow, G. D.},
  title = {Monodromy of hypergeometric functions and non-lattice integral monodromy},
  journal = {Publications Mathématiques de l'IHÉS},
  year = {1986},
  volume = {63},
  pages = {5--89},
  url = {https://doi.org/10.1007/BF02831622}
}

@misc{FT,
      title={Bounds on saddle connections for flat spheres}, 
      author={Kai Fu and Guillaume Tahar},
      year={2023},
      eprint={2308.08940},
      archivePrefix={arXiv},
      primaryClass={math.GT},
      url={https://arxiv.org/abs/2308.08940}, 
}

@article{Ara,
  title = {Universal length bounds for non-simple closed geodesics on hyperbolic surfaces},
  author = {Basmajian, Ara},
  journal = {Journal of Topology},
  year = {2013},
  volume = {6},
  pages = {513--524},
  doi = {10.1112/jtopol/jtt005},
  url = {https://doi.org/10.1112/jtopol/jtt005}
}

@article{BaPV,
  title={The shortest non-simple closed geodesics on hyperbolic surfaces},
  author={Basmajian, Ara and Parlier, Hugo and Vo, Hanh},
  journal={Mathematische Zeitschrift},
  year={2024},
  volume={306},
  number={1},
  pages={Paper No. 8, 19},
  doi={10.1007/s00209-023-03401-8},
  url={https://doi.org/10.1007/s00209-023-03401-8}
}

@incollection{Ma1,
  author    = {Masur, Howard},
  title     = {Lower bounds for the number of saddle connections and closed trajectories of a quadratic differential},
  booktitle = {Holomorphic Functions and Moduli I},
  editor    = {Drasin, David and Kra, Irwin and Earle, Clifford J. and Marden, Albert and Gehring, Frederick W.},
  series    = {Mathematical Sciences Research Institute Publications},
  volume    = {10},
  pages     = {215--228},
  publisher = {Springer},
  year      = {1988},
  doi       = {10.1007/978-1-4613-9602-4_20},
  url       = {https://link.springer.com/chapter/10.1007/978-1-4613-9602-4_20}
}

@article{Ma2,
  title   = {The growth rate of trajectories of a quadratic differential},
  author  = {Masur, Howard},
  journal = {Ergodic Theory and Dynamical Systems},
  year    = {1990},
  volume  = {10},
  number  = {1},
  pages   = {151--176},
  doi     = {10.1017/S0143385700005632},
  url     = {https://doi.org/10.1017/S0143385700005632}
}

@incollection{Zor,
  author    = {Zorich, Anton},
  title     = {Flat Surfaces},
  booktitle = {Frontiers in Number Theory, Physics, and Geometry I},
  editor    = {Cartier, Pierre and Julia, Bernard and Moussa, Pierre and Vanhove, Pierre},
  pages     = {437--583},
  publisher = {Springer},
  year      = {2006},
  doi       = {10.1007/978-3-540-31347-2_9},
  url       = {https://doi.org/10.1007/978-3-540-31347-2_9}
}

@article{Kat,
  title   = {The growth rate for the number of singular and periodic orbits for a polygonal billiard},
  author  = {Katok, Anatole},
  journal = {Communications in Mathematical Physics},
  year    = {1987},
  volume  = {111},
  pages   = {151--160},
  doi     = {10.1007/BF01239021},
  url     = {https://doi.org/10.1007/BF01239021}
}

@article{Sch,
  title   = {Complexity growth of a typical triangular billiard is weakly exponential},
  author  = {Scheglov, Dmitri},
  journal = {Journal d'Analyse Mathématique},
  year    = {2020},
  volume  = {142},
  pages   = {105--124},
  doi     = {10.1007/s11854-020-0134-3},
  url     = {https://doi.org/10.1007/s11854-020-0134-3}
}

@article{FoxKer,
  title   = {Concerning the transitive properties of geodesics on a rational polyhedron},
  author  = {Fox, Ralph H. and Kershner, Richard B.},
  journal = {Duke Mathematical Journal},
  year    = {1936},
  volume  = {2},
  pages   = {147--150},
  doi     = {10.1215/S0012-7094-36-00213-2},
  url     = {https://doi.org/10.1215/S0012-7094-36-00213-2}
}

@article{Mas86,
  author  = {Masur, Howard},
  title   = {Closed trajectories for quadratic differentials with an application to billiards},
  journal = {Duke Mathematical Journal},
  year    = {1986},
  volume  = {53},
  number  = {2},
  pages   = {307--314},
  doi     = {10.1215/S0012-7094-86-05321-8},
  url     = {https://doi.org/10.1215/S0012-7094-86-05321-8}
}

@article{RSch06,
  author  = {Schwartz, Richard E.},
  title   = {Obtuse Triangular Billiards I: Near the $(2,3,6)$ Triangle},
  journal = {Experimental Mathematics},
  volume  = {15},
  number  = {2},
  pages   = {161--182},
  year    = {2006},
  doi     = {10.1080/10586458.2006.10128943},
  url     = {https://doi.org/10.1080/10586458.2006.10128943}
}

@article{RSch08,
  author  = {Schwartz, Richard E.},
  title   = {Obtuse Triangular Billiards II: One Hundred Degrees Worth of Periodic Trajectories},
  journal = {Experimental Mathematics},
  volume  = {18},
  number  = {2},
  pages   = {137--171},
  year    = {2009},
  doi     = {10.1080/10586458.2009.10128999},
  url     = {https://doi.org/10.1080/10586458.2009.10128999}
}

@misc{tokarsky2018point,
      title={One Hundred and Twelve Point Three Degree Theorem}, 
      author={George Tokarsky and Jacob Garber and Boyan Marinov and Kenneth Moore},
      year={2018},
      eprint={1808.06667},
      archivePrefix={arXiv},
      primaryClass={math.DS}
}

@article{Ngu24,
  author  = {Nguyen, Duc-Manh},
  title   = {Intersection theory and volumes of moduli spaces of flat metrics on the sphere},
  journal = {Geometriae Dedicata},
  volume  = {218},
  number  = {32},
  year    = {2024},
  doi     = {10.1007/s10711-023-00883-y},
  url     = {https://doi.org/10.1007/s10711-023-00883-y}
}

@misc{sauvaget2024volumesmodulispacesflat,
      title={Volumes of moduli spaces of flat surfaces}, 
      author={Adrien Sauvaget},
      year={2024},
      eprint={2004.03198},
      archivePrefix={arXiv},
      primaryClass={math.AG},
      url={https://arxiv.org/abs/2004.03198}, 
}

@book{AthreyaMasur2024,
  author = {Athreya, Jayadev S. and Masur, Howard},
  title = {Translation Surfaces},
  series = {Graduate Studies in Mathematics},
  volume = {242},
  isbn = {978-1-4704-7655-7},
  year = {2024},
  publisher = {American Mathematical Society (AMS)},
  url = {https://bookstore.ams.org/gsm-242}
}

@article{Filip2024,
  author  = {Filip, Simion},
  title   = {Translation surfaces: Dynamics and Hodge theory},
  journal = {EMS Surveys in Mathematical Sciences},
  volume  = {11},
  number  = {1},
  pages   = {63--151},
  year    = {2024},
  doi     = {10.4171/EMSS/78},
  url     = {https://ems.press/journals/emss/articles/14297772}
}

@Article{McMullen2023,
 Author = {McMullen, Curtis T.},
 Title = {Billiards and {Teichm{\"u}ller} curves},
 FJournal = {Bulletin of the American Mathematical Society. New Series},
 Journal = {Bulletin (New Series) of the American Mathematical Society},
 ISSN = {0273-0979},
 Volume = {60},
 Number = {2},
 Pages = {195--250},
 Year = {2023},
 Language = {English},
 DOI = {10.1090/bull/1782},
 zbMATH = {7666127},
 Zbl = {1511.32011}
}

@incollection{MasurTabachnikov2002,
  author    = {Masur, Howard and Tabachnikov, Serge},
  title     = {Rational billiards and flat structures},
  booktitle = {Handbook of Dynamical Systems, Volume 1A},
  pages     = {1015--1089},
  year      = {2002},
  publisher = {Elsevier},
  doi       = {10.1016/S1874-575X(02)80015-7},
  url       = {https://doi.org/10.1016/S1874-575X(02)80015-7}
}

@Article{Wright2015,
 Author = {Wright, Alex},
 Title = {Translation surfaces and their orbit closures: an introduction for a broad audience},
 FJournal = {EMS Surveys in Mathematical Sciences},
 Journal = {European Mathematical Society Surveys in Mathematical Sciences},
 ISSN = {2308-2151},
 Volume = {2},
 Number = {1},
 Pages = {63--108},
 Year = {2015},
 Language = {English},
 DOI = {10.4171/EMSS/9},
 zbMATH = {6455787},
 Zbl = {1372.37090}
}

@article{Tro, 
title={Les surfaces euclidiennes à singularités coniques. (Euclidean surfaces with cone singularities).}, 
journal = {L'Enseignement mathématique},
volume={32}, 
url={https://infoscience.epfl.ch/handle/20.500.14299/15756}, 
DOI={10.5169/seals-55079}, 
abstractNote={We prove in this paper that evry compact Riemann surface carries an euclidean (flat) conformal metric with precribed conical singularities of given angles, provided the Gauss-Bonnet relation is satisfied. This metric is unique up to homothety.}, 
number={1–2}, 
author={Troyanov, Marc}, year={1986}, pages={79–94}}

@article{Thurston1997,
  title={Three-Dimensional Geometry and Topology, Volume 1},
  author={James, Ioan Mackenzie and Thurston, William P. and Levy, Sylvio},
  journal={The Mathematical Gazette},
  year={1997},
  volume={82},
  pages={345--346},
  url={https://api.semanticscholar.org/CorpusID:126174690}
}

@article{MasurSmillie1991,
  author  = {Masur, Howard and Smillie, John},
  title   = {Hausdorff dimension of sets of nonergodic measured foliations},
  journal = {Annals of Mathematics},
  volume  = {134},
  number  = {3},
  pages   = {455--543},
  year    = {1991},
  doi     = {10.2307/2944356},
  url     = {https://doi.org/10.2307/2944356}
}

@article{Fiala1940/41,
author = {Fiala, Félix},
journal = {Commentarii mathematici Helvetici},
pages = {293-346},
title = {Le problème des isopérimètres sur les surfaces ouvertes à courbure positive.},
url = {http://eudml.org/doc/138776},
volume = {13},
year = {1940/41},
}

@article{Alexandrov45,
  author  = {Alexandrov, Alexander D.},
  title   = {An isoperimetric problem},
  journal = {Doklady Akademii Nauk SSSR},
  volume  = {50},
  number  = {1},
  pages   = {31--34},
  year    = {1945},
  zbl     = {0063.39603}
}

@misc{Fu25,
      title={Siegel-Veech Measures of Convex Flat Cone Spheres}, 
      author={Fu, Kai},
      year={2025},
      eprint={2504.14731},
      archivePrefix={arXiv},
      primaryClass={math.GT},
      url={https://arxiv.org/abs/2504.14731}, 
}

@article{athreya2019siegel,
  title = {Siegel-Veech transforms are in $L^2$},
  author = {Athreya, Jayadev S. and Cheung, Yitwah and Masur, Howard},
  journal = {Journal of Modern Dynamics},
  volume = {15},
  pages = {1--27},
  year = {2019},
  publisher = {American Institute of Mathematical Sciences},
  doi = {10.3934/jmd.2019001},
  url = {https://doi.org/10.3934/jmd.2019001},
  eprint = {arXiv:1711.08537},
  eprinttype = {arxiv},
  eprintclass = {math.DS}
}

\end{document}